\documentclass{article}

\usepackage{amssymb}
\usepackage{amsmath}
\usepackage{enumerate}
\usepackage{mathrsfs}
\usepackage{graphics}
\usepackage{graphicx}
\usepackage{xcolor}
\usepackage{datetime}

\usepackage{subfigure,dsfont}
\usepackage[T1]{fontenc}
\usepackage{latexsym,amssymb,amsmath,amsfonts,amsthm,txfonts,thmtools}
\usepackage[colorlinks=true, linkcolor=blue, citecolor=blue]{hyperref}

\usepackage{graphicx}
\usepackage{amsmath}
\usepackage{amsfonts}
\usepackage{amsthm}
\usepackage{amssymb,mathtools,mathrsfs}
\usepackage{esint}
\usepackage{enumitem}
\usepackage{dsfont}

\usepackage[T1]{fontenc}
\usepackage[toc,page]{appendix}
\DeclareMathOperator*{\esssup}{ess\,sup}
\usepackage[paperwidth=14cm,textwidth=12cm]{geometry}

\usepackage{setspace}
\usepackage{fouridx}

\newcommand{\nun}{\nu}

\newcommand\eps{\varepsilon}
\newcommand{\toup}{\nearrow}
\newcommand{\dom}{\mathcal{O}}
\newcommand{\rA}{\mathrm{A}}
\newcommand{\rV}{\mathrm{V}}
\newcommand{\rH}{\mathrm{H}}

\newcommand{\rv}{\mathrm{v}}
\newcommand{\divv}{\mathrm{div}}
\newcommand{\Leb}{\mathcal{L}}
\newcommand{\loc}{\mathrm{loc}}
\newcommand{\rM}{M}

\newtheorem{theorem}{Theorem}[section]
\newtheorem{lemma}[theorem]{Lemma}
\newtheorem{prp}[theorem]{Proposition}
\newtheorem{prop}[theorem]{Proposition}
\newtheorem{cor}[theorem]{Corollary}
\newtheorem{remark}[theorem]{Remark}
\newtheorem{definition}[theorem]{Definition}
\newtheorem{assumption}{Assumption}[section]
\newtheorem{example}[theorem]{Example}

\numberwithin{equation}{section}

\begin{document}

\title{Global Existence and Stability of 3D Stochastic NSEs in Bounded and Unbounded Domains Driven by a Special Multiplicative Wiener Process}
\author{Zdzis{\l}aw Brze{\'z}niak$^{a}$, Shijia Zhang$^{b}$ and Guoli Zhou$^{c}$\\
{\small{$a.$   Department of Mathematics, University of York,  York, UK}}\\
{\small $b.$ Academy of Mathematics and Systems Science,}\\
{\small Chinese Academy of Sciences, Beijing 100190, China}\\
{\small{ $c.$ School of Statistics and Mathematics, Chongqing University, Chongqing 400044, China}}\\
({\small zdzislaw.brzezniak@york.ac.uk, shijiazhang@amss.ac.cn, zhouguoli736@126.com })
}

\date{}
\maketitle

\abstract{The aim of this work is to extend the results from a recent paper  \cite{Hong+Li+Liu_2024} by Hong, Li and Liu,  from bounded domains to  both bounded and  unbounded domains. We show the global existence and uniqueness of 3D stochastic Navier-Stokes equations with nonlinear multiplicative noise for every initial data from the Sobolev space $\mathbb{H}^1$. We do not use any tightness argument. Instead, we firstly show the existence of a local maximal solution and then we use Lyapunov function to prove the solution is global.
Our approach is motivated by a recent paper \cite{Brz+Ferrario+Maurelli+Zanella_2025} of the first named author with Ferrario, Maurelli and Zanella about a similar result for stochastic nonlinear Schr\"odinger Equations. The main idea is that once the local existence of strong solutions is established, a very strong noise pushing toward the origin, will make blow-up impossible.}

\section{Introduction}\label{sec-Introduction}
Let $\dom \subset \mathbb{R}^3$ be an open, connected (possibly unbounded) domain with smooth boundary $\partial \dom$. Let $\left(\Omega,\mathscr{F},\mathbb{F},\mathbb{P}\right)$ be a complete probability space endowed with a filtration $\mathbb{F} := \{\mathscr{F}_t\}_{t\geq 0}$.

We consider the stochastic Navier–Stokes equations:
\begin{equation}\label{eqn-NSE-01'}
\frac{\partial u}{\partial t} + (u \cdot \nabla) u - \nu \Delta u + \nabla p = f + \phi(u) u \, \mathrm{d}W, \;\; x \in \dom,
\end{equation}
with the incompressibility condition
\begin{equation}\label{eq-incompressibility}
\divv u = 0.
\end{equation}
Here, $W = \{W(t): t \in [0,\infty)\}$ is a real-valued Wiener process adapted to $\mathbb{F}$, $\nu > 0$ denotes the viscosity and the function $\phi: \rV \to \mathbb{R}$ satisfies appropriate conditions.

Our main result can be viewed as a kind of regularization by noise for the deterministic Navier–Stokes equations. We establish sufficient conditions on the diffusion coefficient $\phi$ that guarantee global-in-time existence of solutions to the stochastic Navier–Stokes equation \eqref{eqn-NSE-01'}, while such a result does not hold for the corresponding deterministic equation (i.e., \eqref{eqn-NSE-01'} with $\phi = 0$).
More precisely, we show that if the initial data $u_0 \in \mathbb{H}^1$ and $\phi$ satisfies certain assumptions, then the stochastic Navier–Stokes equation \eqref{eqn-NSE-01'} admits a unique global strong solution.
Our proof relies on the contraction mapping Theorem and the choice of a suitable Lyapunov function.

Let us review in more detail the known results.

\subsection{Existence of global solutions}
The three-dimensional stochastic Navier–Stokes equations (3D SNSEs) are a central model in stochastic fluid dynamics, describing the motion of an incompressible viscous fluid subject to random perturbations. The global well-posedness of the three-dimensional stochastic Navier–Stokes equations has been a long-standing and challenging open problem in the theory of stochastic partial differential equations (SPDEs).

The local and global well-posedness for the SNSE with multiplicative noise is not as well understood. Initially, the local existence of solutions was addressed in Bensoussan and Temam \cite{Bensoussan+Temam_1973} with additive noise. The first paper which consider the SNSEs driven by the so-called transport noise was the paper \cite{Brz+Capinski+Flandoli_1992}, by the first named author, Capi\'nski and Flandoli. This direction was later taken up by Flandoli and Gatarek \cite{Flandoli+Gatarek_1995}, who also treated other types of Gaussian noise. All these papers and others, which we have not cited, have dealt with NSEs in bounded domains.
The first paper which dealt with the existence of solutions to
the SNSEs (as well as Euler Equations) in unbounded domains was the paper \cite{Brz+Peszat_2001}. They consider the two dimensional SNSEs with multiplicative noise. After that Mikulevicius and Rozovskii published the paper \cite{Mikulevicius+Rozovskii_2005} which studied the SNSEs in the whole space. The method used by these authors was different and used weak topologies instead of weighted spaces used in the former paper.
Later, Mikulevicius and Rozovskii \cite{Mikulevicius+Rozovskii_2005} established the local existence of solutions to the SNSE with the initial data in $\mathbb{W}^{1,p}(\mathbb{R}^d)$, where $p > d$. They also proved that the solution becomes global when the equation is posed in 2D. The regularity of initial data was then further reduced by Glatt-Holtz and Ziane, who in \cite{Glatt-Holtz+Ziane_2009} obtained the local well-posedness of the SNSE with multiplicative noise on a 3D bounded domain and proved the global existence
of solutions in 2D, assuming that the initial data belong to $\mathbb{H}^1$. In \cite{Kim_2010}, Kim obtained the local existence of solutions for the 3D SNSE with initial data in $\mathbb{H}^s$, for $s > 1/2$, and the global existence with a large probability for small data.
Brze\'zniak and Motyl \cite{Brz+Motyl_2013} proved the existence of weak martingale solutions to SNSEs for general unbounded domains.
In \cite{Flandoli+Luo_2021}, Flandoli and Luo show that a suitable transport noise provides a bound on vorticity which gives well-posedness of the 3D SNSE with high probability.

The recent paper by Hong, Li and Liu \cite{Hong+Li+Liu_2024}  establishes a breakthrough result by proving the existence and uniqueness of global strong solutions to the stochastic 3D Navier–Stokes equations on bounded domains, driven by nonlinear multiplicative noise. Their framework, based on a generalized coercivity condition and a variational approach, shows that nonlinear noise plays a key role in preventing singularities.

The aim of the present work is to extend the results of Hong, Li and Liu in \cite{Hong+Li+Liu_2024} from bounded domains to the setting of unbounded domains. That is to say, we intend to show the global existence and uniqueness of 3D stochastic Navier-Stokes equations with nonlinear multiplicative noise. The unbounded domain introduces fundamental analytical difficulties, as the key compact embedding $\rV \subset \rH$—essential to the tightness argument in \cite{Hong+Li+Liu_2024}—no longer holds. To overcome this challenge, we adopt a fixed-point approach inspired by Brzeźniak, Hausenblas, and Razafimandimby \cite{Brz+Haus+Raza_2021}, which avoids compactness arguments entirely.
Note that this approach can be seen as a generalization of the  the classical method of proving the existence of local strong solutions used for instance by the first named author in \cite{Brz_1991}, which in turn was based on ideas from the monograph \cite{Vishik+Fursikov_1988},
can be generalised to the setting of stochastic Navier-Stokes Equations.

\subsection{No blow-up by noise}
A central question in the study of stochastic partial differential equations is whether suitable multiplicative noise can prevent finite-time blow-up and ensure global existence of solutions. This phenomenon, often called no blow-up by noise or regularization by noise, describes how noise can improve the behavior of nonlinear partial differential equations and has been observed in many nonlinear systems in recent years.

Cerrai \cite{Cerrai_2005} studied stochastic reaction–diffusion equations with multiplicative noise. She proved that noise can stabilize systems that are not asymptotically stable in the deterministic case, yielding uniqueness, ergodicity, and strong mixing of invariant measures. Hence, her meaning of "stabilization by noise" was very interesting, yet different from ours.

Alonso-Orán, Miao and Tang \cite{Alonso+Miao+Tang_2022} proved that fast-growing nonlinear noise (strong noise) prevents blow-up for a stochastic non-local one-dimensional transport equation, while in the particular linear noise case, they show that singularities occur in finite time with positive probability.
Tang and Wang \cite{Tang+Wang_2022} then proposed a general framework for singular SPDEs driven by a special pseudo-differential noise, which covers many fluid-type models including Navier-Stokes equations and provides general non-explosion criteria.
Ren, Tang and Wang considered in \cite{Ren+Tang+Wang_2024} a general framework for no blow-up for SPDEs and used it to show no blow-up by a superlinear noise in two examples, the first one is a general nonlinear stochastic transport equation, and the second one is the distribution-path dependent stochastic Camassa-Holm equation.
Crisan and Lang \cite{lang+Crisan_2025} introduced superlinear noise structures that ensure global solutions for a class of evolution equations that possibly have only local solutions, such as the Navier-Stokes equation on the 3D torus $\mathbb{T}^3$.
Flandoli and Rehmeier \cite{Flandoli+Rehmeier_2024} explored the effects of small-scale perturbations on large-scale dynamics. They reviewed how regularization by noise can prevent blow-up, an effect that is sometimes the opposite of that seen in convex integration, while both may have links to spontaneous stochasticity.
Brze\'zniak, Ferrario, Maurelli and Zanella \cite{Brz+Ferrario+Maurelli+Zanella_2025} established, through a Lyapunov functional argument, that sufficiently strong nonlinear noise prevents blow-up for the stochastic nonlinear Schrödinger equation.
Bagnara, Maurelli and Xu \cite{Bagnara+Maurelli+Xu_2025} obtained a similar no blow-up result for the stochastic Euler equations under a multiplicative Itô noise with radial structure and with more than linear growth.
Based on Gess and Yaroslavtsev's work \cite{Gess+Yaroslavtsev_2025}, transport noise can imply global asymptotic stability. Their results confirm that transport noise enables arbitrary large exponential rate of convergence, demonstrating its enhanced dissipation.

These works clearly show that strong nonlinear multiplicative noise can stabilize nonlinear PDEs and prevent singularities. Motivated by these works, the present paper extends this no-blow-up-by-noise phenomenon to the three-dimensional stochastic Navier–Stokes equations on unbounded domains. Using a Lyapunov functional argument similar to \cite{Brz+Ferrario+Maurelli+Zanella_2025}, we prove that sufficiently strong nonlinear multiplicative noise prevents blow-up and yields global strong solutions for initial data in $\mathbb{H}^1$.

\subsection{Main results}
The novelty of the present paper is twofold. First, compared with the existing
stochastic literature, we establish global strong well-posedness for the
three-dimensional stochastic Navier--Stokes equations with nonlinear
multiplicative noise on possibly unbounded domains. In particular, this extends
the bounded-domain result of \cite{Hong+Li+Liu_2024} to a setting where compact
embedding arguments are no longer available. Second, compared with the
deterministic literature, our result can be understood as a regularization (no blow-up) by noise result.

A classical way to obtain
global well-posedness for three-dimensional Navier--Stokes type systems is to
modify the equation by adding an explicit dissipative drift. This leads, for
example, to the tamed Navier--Stokes equations or the
Brinkman--Forchheimer--Navier--Stokes equations, see
\cite{Rockner+Zhang_2009,Brz+Dhariwal_2020,Kinra+Mohan_2022}.
In contrast, in the present work the deterministic Navier--Stokes drift is not
altered. The stabilizing effect comes instead from the It\^o correction
associated with the nonlinear multiplicative noise
\[
\phi(u)u\,\mathrm dW.
\]
Thus the nonlinear noise does not simply add a deterministic damping term to
the equation. Rather, it produces an effective damping mechanism at the level
of the stochastic energy estimates. This distinction is important in our proof
and is one of the main reasons why the It\^o formulation is essential here.

In \autoref{sec-deterministic damping}, we further compare this stochastic
mechanism with the corresponding deterministic damping model
\begin{equation}\label{eq-damped}
\partial_t u+\nu\rA u+B(u)=f(t)-\phi(u)u.
\end{equation}
In the stochastic case, the growth condition is
\[
\vert \phi(u)\vert\geq b\vert \rA^{\frac12}u\vert_\rH^\kappa,
\;\; \kappa>2,\quad b>0.
\]
By contrast, in the
deterministic damping model \eqref{eq-damped}, one needs a stronger growth condition,
\[
\vert\phi(u)\vert\geq b\vert \rA^{\frac12}u\vert_\rH^\kappa,
\;\; \kappa>4,\quad b>0.
\]
The precise comparison is given in \autoref{sec-deterministic damping}.

Our proof proceeds in two main steps. First, we establish the existence of a
unique local maximal strong solution via a fixed-point argument in suitable
Sobolev spaces. Second, we prove that the local maximal solution cannot blow up
in finite time by means of a Lyapunov function argument. Importantly, our method
does not rely on tightness, Skorokhod representation, or compactness arguments.
It therefore provides a flexible approach to nonlinear SPDEs on unbounded
domains.

The rest of the paper is organized as follows. In \autoref{sec-Preliminaries}, we introduce the functional setting and notations used throughout the paper. In \autoref{sec-main results}, we introduce the notions of local mild and global solutions, and state our main results. \autoref{sec-A priori estimates} is devoted to establishing an approximated problem that plays a crucial role in the construction of solutions. In \autoref{sec_Definition and existence of a local solution}, we prove the existence of a local mild solution via a fixed-point argument.
\autoref{subsec-local maximal} extends this to a local maximal solution via a stopping time procedure.
In \autoref{sec-generalization}, we collect several auxiliary results and lemmas that will be used in the proof of the main theorem.
In \autoref{sec-proof of the main result}, we show that the local maximal solution can be continued globally in time.
In \autoref{Examples}, we present several examples of the noise and verifies that the nonlinear noise considered in \cite{Hong+Li+Liu_2024} fulfills our assumptions.
Finally, in \autoref{sec-deterministic damping}, we compare the stochastic mechanism with deterministic nonlinear damping.

\subsection*{Concluding remarks}
Some further developments on this topic can be found in the following papers: \cite{Brz+Peng_2026},   \cite{Brz+Mohan_2026} and \cite{Brz+Ye_2026}.

\section{Preliminaries}\label{sec-Preliminaries}
Assume that $\dom$ is an open subset of $\mathbb{R}^3$ and $1 \leq p \leq  \infty$.
On $\dom$ we consider the  Lebesgue measure  $\Leb$, integration with respect to which will be denoted by $dx = dx_1 dx_2 dx_3$, where $x=(x_1,x_2,x_3)\in\mathbb{R}^3$.

Let $\mathbb{L}^p(\dom)=L^p(\dom, \mathbb{R}^3)$ denote the space of all (equivalence classes) of  $\mathbb{R}^3$-valued functions defined on $\dom$ that are $p$-th power absolutely integrable if $p\in (1,\infty)$ or essentially bounded if $p=\infty$. These are Banach spaces equipped with the following norms
\begin{align}
\Vert u\Vert_{L^p}=\left(\int_{\dom}\vert u(x)\vert^p \,d x\right)^{1 / p},\,\,p\in [1,\infty),
\end{align}
and for $p=\infty$,
\begin{align}
\Vert u\Vert_{L^{\infty}}=\esssup \;\{ \vert u(x)\vert: x \in \dom \}.
\end{align}
In particular, for $p=2$, $\mathbb{L}^2(\dom)$ is a Hilbert space endowed with the scalar product and the norm
\begin{align}
\langle u, \rv\rangle_{L^2}=\int_{\dom} u(x) \rv(x) \,d x,\;\; \Vert u\Vert_{L^2}^2=\langle u, u\rangle_{L^2} ,\;\;u,\rv\in \mathbb{L}^2(\dom).
\end{align}

The Sobolev space $\mathbb{W}^{m, p}(\dom)=W^{m, p}(\dom,\mathbb{R}^3)$, where $m\in\mathbb{N}$ and $1 < p \leq+\infty$,  consists of all elements of  $\mathbb{L}^p(\dom)$,  whose distributional derivatives up to order  $m$ also belong to $\mathbb{L}^p(\dom)$. It is a Banach space with the norm
\begin{align}
\Vert u\Vert_{W^{m, p}}=\left(\sum_{|j| \leq m}\Vert D^j u\Vert_{L^p}^p\right)^{1 / p}.
\end{align}
For $p = 2$, we write $\mathbb{H}^m(\dom) := \mathbb{W}^{m,2}(\dom)$, which is a Hilbert space equipped with the norm $\Vert \cdot \Vert_m$.
In particular, for $m=1$, $\mathbb{H}^1(\dom)$ is a Hilbert space with the scalar product given by
\begin{align}\label{eq-H^1 inner product}
\langle u, \mathrm{v}\rangle_{H^1}:=\langle u, \mathrm{v}\rangle_{L^2}+((u, \mathrm{v})), \quad u, \mathrm{v} \in \mathbb{H}^1\left(\dom\right)
\end{align}
where
\begin{align}\label{eq-((cdot))}
((u, \mathrm{v})):=\langle\nabla u, \nabla \mathrm{v}\rangle_{L^2}=\sum_{i=1}^3 \int_{\dom} \frac{\partial u}{\partial x_i} \cdot \frac{\partial \mathrm{v}}{\partial x_i} \,d x, \quad u, \mathrm{v} \in \mathbb{H}^1\left(\dom\right).
\end{align}
We define the following space of smooth functions with compact support:
\begin{equation}
\begin{aligned}
\mathcal{D}(\dom)&:=\{ u \in \mathcal{C}^{\infty}(\dom,\mathbb{R}^3): u \text{ has compact support in } \dom \}.
\end{aligned}
\end{equation}
The closure of $\mathcal{D}(\dom)$ in $\mathbb{W}^{m, p}(\dom)$ and $\mathbb{H}^m(\dom)$ are defined as follows:
\begin{equation}
\begin{aligned}
\mathbb{W}_0^{m, p}(\dom) &:= \text{the closure of } \mathcal{D}(\dom) \text{ in } \mathbb{W}^{m, p}(\dom),\\
\mathbb{H}_0^m(\dom) &:= \text{the closure of } \mathcal{D}(\dom) \text{ in } \mathbb{H}^m(\dom).
\end{aligned}
\end{equation}
We introduce the following function spaces:
\begin{equation}\label{eq-space H V}
\begin{aligned}
\mathcal{V} &:= \{ u \in \mathcal{D}(\dom): \ \mathrm{div}\, u = 0 \}, \\
\rH &:= \text{the closure of } \mathcal{V} \text{ in } \mathbb{L}^2(\dom), \\
\rV &:= \text{the closure of } \mathcal{V} \text{ in } \mathbb{H}^1(\dom).
\end{aligned}
\end{equation}
One can show that $\rV$  is equal to the closure of  $\mathcal{V}$ in  $\mathbb{H}^1_0(\dom)$.

We equip $\rH$ and $\rV$ with the scalar products and norms inherited from $\mathbb{L}^2(\dom)$ and $\mathbb{H}^1(\dom)$, respectively:
\begin{align}\label{eqn-H-norm}
   & \langle u, \rv\rangle_{\rH}:=\langle u, \rv\rangle_{L^2}, \;\;\vert u\vert_{\rH}:=\Vert u\Vert_{L^2}, \;\; u,\rv \in \rH.
\\
\label{eqn-V-norm}
&\langle u, \rv\rangle_\rV:=\langle u, \rv\rangle_{L^2}+((u, \rv)), \;\;\vert u\vert_\rV^2:=\vert u\vert_{\rH}^2+\Vert \nabla u\Vert_{L^2}^2,\;\; u, \rv \in \rV.
\end{align}

Let $\rA$ be the Stokes operator defined by
\begin{align}
\label{eq-D(A)}
D(\rA)&=\mathbb{H}^2(\dom)\cap \rV,
\\
\label{operator A}
\rA u&=  -P(\Delta u) \in \rH, \mbox{ if } u \in D(\rA),
\end{align}
where
\begin{align}\label{eqn-P}
P: \mathbb{L}^2(\dom)\to \rH
\end{align}
is the orthogonal projection called the Leray-Helmholtz projection. It is well known that $\rA$ is a self-adjoint operator in the Hilbert space $\rH$.

The domain $D(\rA)$ is a Hilbert space when endowed with the graph norm
\begin{align}\label{eqn-D(A)-norm}
\vert u\vert_{D(\rA)}^2:=\Vert u\Vert_{L^2}^2+\vert \rA u\vert_{\rH}^2, \;\; u \in D(\rA).
\end{align}
This norm is equivalent to the $\mathbb{H}^2$-norm, i.e., there exists $C_1>0$ such that
\begin{align}\label{eqn-D(A)-H^2}
    C_1^{-1} \Vert u\Vert_2 \leq \vert u\vert_{D(\rA)} \leq C_1 \Vert u\Vert_2
    , \;\; u \in D(\rA).
\end{align}
The following property is well known.
\begin{prop}
Since $\rA$ is a nonnegative self-adjoint operator on the Hilbert space $\rH$, by the spectral theorem there exist an orthonormal basis
$\{e_k\}_{k\ge1}$ and eigenvalues $\{\lambda_k\}_{k\ge1}\subset[0,\infty)$ such that
$\rA e_k=\lambda_k e_k$.
\end{prop}
\begin{definition}[Fractional powers of the Stokes operator]
For $u=\sum_{k\ge1} u_k e_k$ with $u_k=\langle u,e_k\rangle_{\rH}$ we define
\begin{equation}
    \begin{aligned}
D(\rA^{1/2}) &:= \Big\{u\in\rH:\ \sum_{k\ge1}\lambda_k\,|u_k|^2<\infty\Big\},
\\
\rA^{1/2}u &:= \sum_{k\ge1} \lambda_k^{1/2} u_k e_k, \;\; u \in D(\rA^{1/2}).
    \end{aligned}
\end{equation}
\end{definition}
It follows easily that if $u \in D(A)$, then $\rA^{1/2}u  \in D(\rA^{1/2})$ and
$\rA^{1/2}(\rA^{1/2}u)=\rA u $.
This implies that
\[
|\rA^{1/2}u|_\rH=\Vert\nabla u\Vert_{L^2(\dom)}\le |u|_\rV, \;\; u \in \rV.
\]
It is known that, see e.g. \cite[page 57]{Teman_1997}
\[
D(\rA^{1/2})=\rV.
\]

Next, for $u,\rv,w$ such that the integral is well-defined, set
\[
b(u, \rv, w)=\int_\dom (u \cdot\nabla) \rv\cdot w \,d x=\sum_{i, j=1}^3 \int_\dom u^i(x) D_i \rv^j(x) w^j(x) \,d x.
\]
An important property of the form $b$ (see \cite[Chapter II, Lemma 1.3]{Temam_1977}) is
\begin{equation}\label{eq-b}
\begin{aligned}
&b(u, \rv, \rv)=0, \mbox{ if } u \in \rV,\; \rv \in \mathbb{H}_0^1(\dom)\cap \mathbb{L}^3(\dom)\\
& b(u, \rv, w)=-b(u, w, \rv), \mbox{ if } u \in \rV,\; \rv, w \in \mathbb{H}_0^1(\dom)\cap \mathbb{L}^3(\dom).
\end{aligned}
\end{equation}

If $(u,\rv) \in \rH\times \rH$ is such that $\rV \ni w\mapsto b(u,\rv,w)\in \mathbb{R}$ is continuous, we denote the corresponding element of $\rV^\prime$ by $B(u,\rv)$, i.e.
\[
\langle B(u,\rv),w\rangle_{\rV^\prime,\rV}=b(u,\rv,w),\;\;w\in \rV.
\]
The set of all such pairs will be denoted by $D(B)$. In view of \eqref{eq-b}, we infer
\begin{align}\label{eq-<B(u),u>}
\langle B(u,\rv),\rv\rangle =0 \mbox{ if } u \in \rH,\; \rv \in \mathbb{H}_0^1(\dom)\cap \mathbb{L}^3(\dom).
\end{align}
We also write $B(u):=B(u,u)$. If  $(u,\rv) \in \rH\times \rH$ is  that $(u\cdot\nabla)\rv\in \mathbb{L}^2(\dom)$, then
$(u,\rv) \in D(B)$ and
\begin{align}\label{eqn-B}
B(u):=P(u\cdot \nabla)u.
\end{align}
In particular, if $u\in \rH$ satisfies $(u\cdot\nabla)u\in \mathbb L^2(\dom)$, then
\[
B(u)=P\big((u\cdot \nabla)u\big).
\]

We will recall the fundamental properties of the form $B$. Note that
\begin{align}\label{eqn-B-H norm L2}
\vert B(u)\vert_\rH\leq \Vert (u\cdot \nabla)u\Vert_{L^2}.
\end{align}

In the following, we introduce some function spaces used throughout the paper.
\begin{definition}
\label{def-X_a,b}
For $b>a \geq 0$, we define the Banach space $X_{a, b}$ by
\begin{equation}\label{defnition X_a,b}
 \begin{aligned}
X_{a, b} & :=L^2(a, b ; D(\rA)) \cap C([a, b] ; \rV).
\end{aligned}
\end{equation}
The norm in $X_{a, b}$ will be denoted by
\begin{align}\label{norm X_a,b}
|u|_{X_{a, b}}^2 & =\sup _{t \in[a, b]}\vert u(t) \vert_\rV^2+\int_a^b|u(t)|_{D(\rA)}^2 \,d t.
\end{align}
For $a=0$ and $b=T>0$, we write $X_T:=X_{0, T}$, i.e.
\[
X_T:=L^2(0,T;D(\rA))\cap C([0,T];\rV).
\]
\end{definition}
and
\begin{align}\label{norm X_T}
\vert u\vert_{X_T}^2 & =\sup _{t \in[0, T]}\vert u(t) \vert_\rV^2+\int_0^T\vert u(t)\vert_{D(\rA)}^2 \,d t.
\end{align}
Note that if $u \in X_T$, then the map $[0, T] \ni t \mapsto \vert u\vert_{X_t}$ is non-decreasing function.
Moreover, for every $t\in[0,T]$ and $u\in X_T$, we infer
\begin{align}\label{eqn-A1/2u X_t}
\vert \rA^\frac{1}{2}u(t)\vert_\rH^2\leq \vert u(t) \vert_\rV^2\leq \vert u\vert_{X_T}^2\;\;\mbox{ and  }\int_0^t\vert u(t)\vert_{D(\rA)}^2 \,d t\leq \vert u\vert_{X_T}^2.
\end{align}

We make the following assumption about the probabilistic part of our settings.
\begin{assumption}\label{ass-prob}
We assume that  \[\left(\Omega,\mathscr{F},\mathbb{F},\mathbb{P}\right)\] is a complete probability space endowed with a filtration $\mathbb{F} := \{\mathscr{F}_t\}_{t\geq 0}$ satisfying the usual
conditions, i.e., the filtration is right continuous and all null sets of $\mathscr{F}$ are elements of $ \mathscr{F}_0$.

$W=\{W(t): t \in [0,\infty)\}$  is a $\mathbb{R}$-valued Wiener process defined on a complete filtered probability space $\left(\Omega,\mathscr{F},\mathbb{F},\mathbb{P}\right)$, where $\mathbb{F}=\{\mathscr{F}_t\}_{t\geq 0}$.
\end{assumption}
\begin{definition}
\label{def-M^2X_T}
Denote by $\mathbb{M}^2\left(X_T\right)$, the separable Banach space of all (equivalence classes of) $D(\rA)$-valued $\mathbb{F}$-progressively measurable processes $u$ whose trajectories almost surely belong to $X_T$, and satisfying
\begin{align}\label{M^2 X_T}
\vert u\vert_{\mathbb{M}^2\left(X_T\right)}^2=\mathbb{E}\left(\vert u\vert_{X_T}^2\right)=\mathbb{E}\left(\sup _{t \in[0, T]}\vert u(t) \vert_\rV^2+\int_0^T\vert u(t)\vert_{D(\rA)}^2\, d t\right)<\infty.
\end{align}
Similarly, by $\mathbb{M}^2(0, T ; \rV)$, we denote the separable Banach space of all (equivalence classes of) $\rV$-valued $\mathbb{F}$-progressively measurable processes $u$ whose trajectories almost surely belong to $L^2(0, T ; \rV)$, and such that
\begin{align}\label{M^2(0,T;V)}
\vert u\vert_{\mathbb{M}^2(0, T ; \rV)}^2=\mathbb{E}\left(\Vert u\Vert_{L^2(0, T ; \rV)}^2\right)=\mathbb{E}\left(\int_0^T\vert u(t) \vert_\rV^2 d t\right)<\infty.
\end{align}
\end{definition}
\subsection{Poincar\'e Domain}\label{subsec-Poincar\'e Domain}
If $\dom$ is a bounded domain with smooth boundary or, more generally, a Poincar\'e domain (i.e., an unbounded domain satisfying the Poincar\'e condition, see e.g. \cite{Brz+Li_2006}), then it is known that
the operator $\rA$ is an isomorphism between $D(\rA)$ and $\rH$. In particular, there exists a constant $C_2>0$ such that
\begin{align}\label{eqn-Poincar\'e}
    C_2^{-1} \vert \rA u\vert_{\rH} \leq \vert u\vert_{D(\rA)} \leq C_2 \vert \rA u\vert_{\rH}
    , \;\; u \in D(\rA).
\end{align}
Equivalently, there exists a constant $C_3>0$ such that
\begin{align}\label{eqn-Poincar\'e-0}
    C_3^{-1} \Vert  u\Vert_2 \leq \vert \rA u\vert_{\rH} \leq C_3 \Vert u \Vert_2
    , \;\; u \in D(\rA).
\end{align}

Moreover, there exists a constant $C_4>0$ such that
\begin{align}\label{eqn-Poincar\'e-1}
C_4^{-1}\vert u\vert_\rV \leq \Vert \nabla u\Vert_{L^2}\leq C_4\vert u\vert_\rV  , \;\; u \in \rV.
\end{align}

In view of \eqref{eqn-Poincar\'e} and \eqref{eqn-Poincar\'e-1}, the $X_{a,b}$ norm defined in \eqref{norm X_a,b} is equivalent to $\vert \cdot\vert_{X_T^0}$, where $\vert \cdot\vert_{X_T^0}$ is defined by
\begin{align}\label{eq-X_T^0}
\vert u\vert_{X_T^0}^2:=
\sup _{t \in[0, T]}\vert \rA^\frac{1}{2}u(t)\vert_\rH^2 +\int_0^{T}\vert \rA u(t)\vert_{\rH}^2 \, d t,\;\;T>0.
\end{align}

\subsection{Inequalities}\label{subsec-Inequalities}
In the following, we state some useful variant of the classical Agmon inequality, before giving some estimates for $B$. It will also be used in \autoref{sec-generalization}.
\begin{lemma}\label{lem-Agmon}
Assume that one of the following two assumptions is satisfied.
\begin{trivlist}
\item[(i)] $\dom$ is a bounded domain or an unbounded domain satisfying the Poincar\'e condition;
\item[(ii)] $\dom=\mathbb{R}^3$.
\end{trivlist}
Then, there exists $C>0$ such
\begin{equation}\label{eqn-the Agmon inequality}
\Vert u\Vert_{L^{\infty}} \leq C\Vert \nabla u\Vert_{L^2}^\frac12\vert \rA u\vert_{\rH}^\frac12, \;\; u\in D(\rA).
\end{equation}
\end{lemma}
\begin{proof}[Proof of \autoref{lem-Agmon} part (i)]
By the classical Agmon inequality, see  \cite[chapter II, (1.40)]{Temam_1997} or  \cite[Theorem 1.20]{RRS}, there exists $C>0$ such that
 \begin{equation}\label{eqn-the Agmon inequality-2}
\Vert u\Vert_{L^{\infty}} \leq C\Vert u\Vert_1^{1/2}\Vert u\Vert_2^{1/2}, \;\; u\in \mathbb{H}^2(\dom).
\end{equation}
In view of \eqref{eqn-Poincar\'e-0} and \eqref{eqn-Poincar\'e-1},
 \[
\Vert u\Vert_1 \leq C_4\Vert \nabla u\Vert_{L^2}, \;\; u \in \rV,
  \]
and
\[
\Vert u\Vert_2 \leq C_3 \vert \rA u\vert_{\rH}, \;\; u \in D(\rA).
\]
Since $D(\rA)= \mathbb{H}^2(\dom)\cap \rV$, substituting these bounds into
\eqref{eqn-the Agmon inequality-2} yields \eqref{eqn-the Agmon inequality}. The proof of (i) of \autoref{lem-Agmon} is complete.
\end{proof}

\begin{proof}[Proof of \autoref{lem-Agmon} part (ii)]
Firstly, we aim to show that for $u\in \mathbb{H}^2(\mathbb{R}^3)$,
\begin{align}\label{eqn-agmon unbounded}
\Vert u\Vert_{L^\infty}\leq C\Vert \Delta u\Vert_{L^2}^\frac{1}{2}\Vert \nabla u\Vert_{L^2}^\frac{1}{2}.
\end{align}
Using the Sobolev
inequality \cite[Theorem 4.31]{Adams+Fournier_2003} and  \cite[Theorem I.9.3]{Friedman_1969},
there exists a constant $C>0$ such that
for every $\tilde{u}\in C_0^\infty(\mathbb{R}^3)$ we have
\begin{align}
\label{eqn-Theorem I.9.3-Friedman_1969}
\Vert \tilde{u}\Vert_{L^\infty}&\leq C\Vert \Delta \tilde{u}\Vert_{L^2}^\frac{1}{2}\Vert \tilde{u}\Vert_{L^6}^\frac{1}{2},
\\
\label{eqn-Sobolev embedding}
\Vert \tilde{u}\Vert_{L^6}& \leq C\Vert \nabla \tilde{u}\Vert_{L^2}.
\end{align}
Hence, we infer that
\begin{align}\label{eqn-agmon unbounded-C}
\Vert \tilde{u}\Vert_{L^\infty}\leq C\Vert \Delta \tilde{u}\Vert_{L^2}^\frac{1}{2}\Vert \nabla \tilde{u}\Vert_{L^2}^\frac{1}{2},\;\;\tilde{u}\in C_0^\infty(\mathbb{R}^3).
\end{align}
Now, we need to extend it
for every $u\in \mathbb{H}^2(\mathbb{R}^3)$.

For this purpose let us choose and fix $u\in \mathbb{H}^2(\mathbb{R}^3)$.
Since $C_0^\infty(\mathbb{R}^3)$ is dense in $\mathbb{H}^2(\mathbb{R}^3)$, there exists a sequence $ \{\tilde{u}_n\} \subset C_0^\infty(\mathbb{R}^3) $ such that
\begin{align}\label{lim-nabla u_n& Delta u_n}
\Vert  \tilde{u}_n -  u \Vert_{L^2} \to 0, \;\;\Vert \nabla \tilde{u}_n - \nabla u \Vert_{L^2} \to 0 \text{ and } \Vert \Delta \tilde{u}_n - \Delta u \Vert_{L^2} \to 0,  \text{ as } n \to \infty.
\end{align}
It follows that there exists a subsequence still denoted by $\{\tilde{u}_{n}\}$ such that
\begin{align}\label{lim-u_n to u}
\tilde{u}_n\to u,\;\; \text{ for almost every }x\in \mathbb{R}^3.
\end{align}

In view of \eqref{eqn-agmon unbounded-C} and \eqref{lim-nabla u_n& Delta u_n}, we infer that
\[
\Vert \tilde{u}_n - \tilde{u}_m \Vert_{L^\infty}\leq C\Vert \Delta \tilde{u}_n -\Delta \tilde{u}_m \Vert_{L^2}^\frac{1}{2}\Vert \nabla \tilde{u}_n - \nabla \tilde{u}_m \Vert_{L^2}^\frac{1}{2}\to 0, \mbox{ as } n,m\to \infty.
\]
which implies $\{\tilde{u}_n\}$ is a Cauchy sequence in $\mathbb{L}^\infty(\mathbb{R}^3)$.
Therefore, there exists a unique $\rv \in \mathbb{L}^\infty(\mathbb{R}^3)$ such that
\[
\tilde{u}_n \to \rv \mbox{ in } \mathbb{L}^\infty(\mathbb{R}^3).
\]
It follows that there exists a subsequence still denoted by $\{\tilde{u}_{n}\}$ such that
\begin{align}\label{lim-u_n to v1}
\tilde{u}_n\to \rv,\;\; \text{ for almost every }x\in \mathbb{R}^3.
\end{align}
Together with \eqref{lim-u_n to u}, we deduce
\[ u(x) = \rv(x),\;\; \text{ for almost every }x\in \mathbb{R}^3. \]
Therefore, using the continuity of the norm functions and inequality \eqref{eqn-agmon unbounded-C}, we infer that for every $u \in \mathbb{H}^2(\mathbb{R}^3)$
\begin{align}\label{eqn-agmon unbounded1}
\Vert u\Vert_{L^\infty}=\lim_{n\to \infty}\Vert \tilde{u}_n\Vert_{L^\infty}\leq C\lim_{n\to \infty}\Vert \Delta \tilde{u}_n\Vert_{L^2}^\frac{1}{2}\Vert \nabla \tilde{u}_n \Vert_{L^2}^\frac{1}{2}=C \Vert \Delta u\Vert_{L^2}^\frac{1}{2}\Vert \nabla u \Vert_{L^2}^\frac{1}{2},
\end{align}
which implies inequality \eqref{eqn-agmon unbounded} holds for every $u\in \mathbb{H}^2(\mathbb{R}^3)$.

Note that $D(\rA)\subset \mathbb{H}^2(\mathbb{R}^3)$. Therefore, \eqref{eqn-agmon unbounded} implies \eqref{eqn-the Agmon inequality} holds. The proof of (ii) of \autoref{lem-Agmon} is complete.
\end{proof}

The following lemma is a direct consequence of \autoref{lem-Agmon}.
\begin{cor}\label{cor-Agmon}
There exists $C>0$ such that
\begin{align}\label{eqn-the Agmon inequality-1}
\Vert u\Vert_{L^{\infty}} \leq C\vert  u\vert_\rV^{1/2}\vert u\vert_{D(\rA)}^{1/2}, \;\; u\in D(\rA).
\end{align}
\end{cor}
\begin{proof}
In view of the definitions of $\vert \cdot \vert_\rV$ and $\vert \cdot \vert _{D(\rA)}$ given in \eqref{eqn-V-norm} and \eqref{eqn-D(A)-norm}, respectively, we infer
\begin{align}\label{eqn-nabla u-V}
\Vert \nabla u\Vert_{L^2}\leq \vert u \vert_\rV
\end{align}
and
\begin{align}\label{eqn-A u-DA}
\vert \rA u\vert_{\rH}\leq \vert u \vert _{D(\rA)}.
\end{align}
Therefore, the Agmon inequality \autoref{lem-Agmon} implies
\begin{align}\label{eqn-the Agmon inequality-3}
\Vert u\Vert_{L^{\infty}} \leq C\Vert \nabla u\Vert_{L^2}^{1/2}\vert \rA u\vert_{\rH}^{1/2}\leq C\vert  u\vert_\rV^{1/2}\vert u\vert_{D(\rA)}^{1/2}, \;\; u\in D(\rA).
\end{align}
The proof is complete.
\end{proof}

Based on the Agmon inequality, we establish the following estimates for the nonlinear term $B$. While the first inequality can be found in \cite{Temam_1997}, we provide a proof here for completeness.
\begin{lemma}\label{lem-inequality} There exists $C>0$ such that for every $u \in D(\rA)$,
\begin{equation}\label{eqn-Fu-0}
\begin{aligned}
 \vert B(u)\vert_\rH &\leq  C \vert u \vert _{D(\rA)}^{1/2}\vert u \vert_\rV^{3/2}.
\end{aligned}
\end{equation}
There exists $C>0$ such that for every $u, \rv \in D(\rA)$,
\begin{equation}\label{eqn-Fu-Fv-0}
\begin{aligned}
\vert B(u)-B(\rv)\vert_\rH &\leq C(\vert u \vert _{D(\rA)}^{1/2}\vert u \vert_\rV^{1/2}\vert u-\rv \vert _\rV+\vert u-\rv \vert _{D(\rA)}^{1/2}\vert u-\rv \vert _\rV^{1/2}\vert \rv \vert_\rV).
\end{aligned}
\end{equation}
\end{lemma}
\begin{proof}
Using the Agmon inequality \autoref{cor-Agmon}, inequalities \eqref{eqn-B-H norm L2} and \eqref{eqn-V-norm} we infer that for every $u \in D(\rA)$,
\begin{equation}\label{eqn-Fu}
\begin{aligned}
 \vert B(u)\vert_\rH &\leq \Vert u\cdot \nabla u\Vert_{L^2} \leq \Vert u \Vert _{L^\infty}\Vert \nabla u\Vert_{L^2} \\
 &\leq C\vert u \vert _{D(\rA)}^{1/2}\vert u \vert_\rV^{1/2}\Vert \nabla u\Vert_{L^2}\leq C \vert u \vert _{D(\rA)}^{1/2}\vert u \vert_\rV^{3/2},
\end{aligned}
\end{equation}
which implies inequality \eqref{eqn-Fu-0} holds.

Using the triangle inequality, the Agmon inequality \autoref{cor-Agmon}, inequalities \eqref{eqn-B-H norm L2} and \eqref{eqn-V-norm} again, we infer that for every $u, \rv \in D(\rA)$,
\begin{equation}\label{eqn-Fu-Fv}
\begin{aligned}
\vert B(u)-B(\rv)\vert_\rH &\leq \Vert u\cdot \nabla u - \rv\cdot \nabla \rv\Vert_{L^2}=\Vert u\cdot \nabla (u - \rv)+(u - \rv)\cdot \nabla \rv\Vert_{L^2}  \\
&\leq \Vert u\Vert_{L^\infty}\Vert \nabla(u - \rv)\Vert_{L^2}+\Vert u-\rv\Vert_{L^\infty}\Vert \nabla \rv\Vert_{L^2}  \\
&\leq C\vert u \vert _{D(\rA)}^{1/2}\vert u \vert_\rV^{1/2}\Vert \nabla(u - \rv)\Vert_{L^2}+C\vert u-\rv \vert _{D(\rA)}^{1/2}\vert u-\rv \vert_\rV^{1/2}\Vert \nabla \rv\Vert_{L^2}  \\
&\leq C (\vert u \vert _{D(\rA)}^{1/2}\vert u \vert_\rV^{1/2}\vert u-\rv \vert_\rV+\vert u-\rv \vert _{D(\rA)}^{1/2}\vert u-\rv \vert_\rV^{1/2}\vert \rv \vert_\rV).
\end{aligned}
\end{equation}
We complete the proof of \autoref{lem-inequality}.
\end{proof}

\section{Main results}\label{sec-main results}

\subsection{Abstract Formulation}
We study the stochastic Navier–Stokes equations \eqref{eqn-NSE-01'} on a domain $\dom\subset\mathbb{R}^3$.  Without loss of generality we assume the viscosity
\begin{equation}\label{eqn-nu=1}
\nun=1.
\end{equation}
By applying the Leray–Helmholtz projection $P$ to both sides of \eqref{eqn-NSE-01'}, and using the Stokes operator $\rA: D(\rA)\to \rH$ and the bilinear term $B(u):=P\big[(u\cdot\nabla)u\big]$, this equation can be written in the following abstract form:
\begin{equation}
\begin{aligned}\label{eqn-NS03}
du(t)&+\big[\rA u(t)+B(u(t))\big]\,dt=g(u(t))  \,dW(t), \;\; t \in[0, \infty) ,
\\
u(0)&=u_0,
\end{aligned}
\end{equation}
where,
\begin{equation}\label{eqn-g}
g: \rV \ni u \mapsto \phi(u) u  \in \rV,
\end{equation}
and
\begin{equation}\label{eqn-phi}
\phi: \rV \to  \mathbb{R}
\end{equation}
which satisfies the following assumptions.

We impose the following assumptions on the noise term, where \autoref{ass-g-loc Lip} will be used for the local existence, \autoref{assumption-phi} and \autoref{assumption-phi-1} will be used for the global existence.

\begin{assumption}\label{ass-g-loc Lip}
For every $R>0$,  there exists a constant $C_6:=C_6(R)>0$ such that
for all $u,\rv \in V$, if $ \vert u\vert_{\rV} \leq R$ and $ \vert \rv\vert_{\rV} \leq R$, then
\begin{align}\label{eqn-g Lipschitz-100}
 \vert g(u)-g(\rv) \vert_\rV  \leq C_6     \vert u-\rv\vert_\rV.
\end{align}
\end{assumption}

\begin{assumption}\label{assumption-phi}
\begin{trivlist}
\item[(i)]
There exist constants $\kappa>2$ and $b>0$ or $\kappa=2$ with $b$ sufficiently large such that,
\begin{align}\label{eqn-A1}
\vert \phi(u)  \vert \geq b\vert \rA^\frac{1}{2}u\vert_\rH^\kappa \mbox{ for every $u\in \rV$ }
\end{align}

\item[(ii)] For every $T>0$ and $R>0$,
if $u \in X_T$ satisfies $\vert u \vert_{X_T} \leq R$, then there exists a constant $C:=C(R)>0$ such that

\begin{equation}\label{eqn-A3}
\begin{aligned}
\sup_{t\in [0,T]}\vert \phi(u(t))\vert &\leq C.
\end{aligned}
\end{equation}
\end{trivlist}
\end{assumption}
For $\dom=\mathbb{R}^3$, we further need the following assumption:
\begin{assumption}\label{assumption-phi-1}
There exists a continuous increasing  function $\Gamma:[0,\infty)\to \mathbb{R}$ such that
\begin{equation}\label{eqn-lim fx=infty}
\lim_{x\to \infty} \Gamma(x)= \infty,
    \end{equation}
and for every $u\in X_T$, the function $\phi$ satisfies
\begin{align}
\sup_{t\in [0,T]}\vert\phi(u(t))\vert^2\leq \Gamma(\vert u\vert_{X_T^0}^2),
\end{align}
where $\vert  \cdot \vert_{X_T^0}^2$ is  defined in \eqref{eq-X_T^0}.
\end{assumption}
\begin{remark}
The function $\phi$ is defined on
$\rV$, not on the path space $X_T$. For $u\in X_T$, we have
$u(t)\in \rV$ for every $t\in[0,T]$. \autoref{assumption-phi-1} means that
$\sup_{t\in[0,T]}\vert \phi(u(t))\vert$ can be controlled by
$\vert u\vert_{X_T^0}$.
\end{remark}

\subsection{Definitions  of local, local maximal and global  solutions}

\begin{definition}[Admissible process]\label{def-admissible}
Assume that $X$ is a topological space and $\tau$ is a stopping time. An $X$-valued  local process $u:[0, \tau) \times \Omega \rightarrow X$, where
\begin{align*}
[0, \tau) \times \Omega&:= \left\{ (t,\omega) \in [0,\infty) \times \Omega: 0\leq t< \tau(\omega) \right\},
\\
\Omega_t(\tau)&:= \left\{ \omega \in  \Omega: t< \tau(\omega) \right\}, \;\; t\geq 0,
\end{align*}
is called admissible if and only if  the following conditions are satisfied:
\begin{trivlist}
\item[(i)] it is $\mathbb{F}$-adapted, i.e. $u\vert_{\Omega_t(\tau)}: \Omega_t(\tau) \rightarrow X$ is $\mathscr{F}_t$-measurable, for all $t \geq 0$;
\item[(ii)] for almost all $\omega \in \Omega$, the function $[0, \tau(\omega)) \ni t \mapsto u(t, \omega) \in X$ is continuous.
\end{trivlist}
\end{definition}
\begin{definition}[Equivalence of local processes]\label{def-local processes equivalent}
Two local processes $u:[0, \tau) \times \Omega \rightarrow X$, and $\rv:[0, \sigma) \times \Omega \rightarrow X$, are called equivalent (denoted by $(u, \tau) \sim(\rv, \sigma)$) if and only if
\begin{trivlist}
\item[(i)] $\tau=\sigma$, $\mathbb{P}$-a.s. ;
\item[(ii)] for all $t>0$,
\begin{align}
u(\cdot, \omega)=\rv(\cdot, \omega) \text { on }[0, t], \quad \text { for a.e. } \omega \in \Omega_t(\tau) \cap \Omega_t(\sigma).
\end{align}
\end{trivlist}
Note that if two local admissible processes $u:[0, \tau) \times \Omega \rightarrow X$ and $\rv:[0, \sigma) \times \Omega \rightarrow$ $X$, satisfy all $t>0$ , $u(t)\vert_{\Omega_t(\tau)}=\rv(t)\vert_{\Omega_t(\sigma)} ,\mathbb{P}$-a.s., then they are equivalent.
\end{definition}

\begin{definition}[Accessible stopping time]\label{def-stopping time accessible}
A stopping time $\tau$ is called accessible if and only if there exists an increasing sequence $(\tau_m)_{m \in \mathbb{N}}$ of stopping times such that, on the set $\{\tau>0\}$,
\begin{trivlist}
\item[(i)] $\tau_m<\tau$ ;
\item[(ii)] $lim _{m \rightarrow \infty} \tau_m=\tau$, $\mathbb{P}$-a.s..
\end{trivlist}
Such a sequence $\left(\tau_m\right)_{m \in \mathbb{N}}$ is called an approximating sequence for $\tau$.
\end{definition}

\begin{definition}[Local mild solution]\label{def-local solution}
Assume that   $u_0$ is an $\rV$-valued $\mathscr{F}_0$-measurable random variable  such that
\begin{equation}\label{eqn-u_0 2 integrable}
 \mathbb{E}\vert u_0\vert_\rV^2<\infty
\end{equation}
A local mild solution to problem (\ref{eqn-NS03}) is a pair $(u, \tau)$ such that
\begin{trivlist}
\item[(i)]  $\tau$ is an accessible stopping time;
\item[(ii)]$u:[0, \tau) \times \Omega \mapsto \rV$ is an admissible process;
\item[(iii)] There exists an approximating sequence $\left(\tau_m\right)_{m \in \mathbb{N}}$ which approximates the finite stopping times $\tau$,
such that for all $m \in \mathbb{N}$ and $t \geq 0$, we have
\begin{align}\label{eq-local mild solution-1}
\mathbb{E}\left(\sup _{s \in\left[0, t \wedge \tau_m\right]} \vert u(s) \vert_\rV^2+\int_0^{t \wedge \tau_m}
\vert u(s) \vert_{D(\rA)}^2\,ds\right)<\infty,
\end{align}
and
\begin{equation}\label{eq-local mild solution}
\begin{aligned}
& u\left(t \wedge \tau_m\right)=  S\left(t \wedge \tau_m\right) u_0-\int_0^{t \wedge \tau_m} S\left(t \wedge \tau_m-r\right) B(u(r)) \,d r \\
&+ \int_0^{t\wedge \tau_m} \mathbf{1}_{\left[0, \tau_m\right)}(r) S(t\wedge \tau_m-r) g\left(u\left(r \wedge \tau_m\right)\right)\, d W(r),  \quad
\mathbb{P} \text {-a.s. },
\end{aligned}
\end{equation}
where $S(t)=e^{-t\rA}, t \geq 0$ is the analytic semigroup generated by operator $-\rA$.
\end{trivlist}

A local solution $(u, \tau)$ to problem (\ref{eqn-NS03}) is unique if and only if  for all other local solution $(\rv, \sigma)$ to (\ref{eqn-NS03}), the restricted processes $u\vert_{[0, \tau \wedge \sigma) \times \Omega}$ and $\rv\vert_{[0, \tau \wedge \sigma) \times \Omega}$ are equivalent.
\end{definition}

According to the next result, the "mild" formulation \eqref{eq-local mild solution} is equivalent to the strong formulation \eqref{eq-local strong solution}.

\begin{prop}\label{prop-equivalence} Assume that   $u_0$ is $\rV$-valued $\mathscr{F}_0$-measurable random variable  satisfying condition   \eqref{eqn-u_0 2 integrable}.
Assume that a pair  $(u, \tau)$ satisfies conditions (i), (ii) and the first part of condition (iii) of
    Definition \ref{def-local solution}. Then property  \eqref{eq-local mild solution} and the following one
\begin{equation}\label{eq-local strong solution}
\begin{aligned}
 u\left(t \wedge \tau_m\right)&=  u_0- \int_0^{t \wedge \tau_m} \rA u(r) \, dr -
\int_0^{t \wedge \tau_m} B(u(r)) \,d r \\
&+ \int_0^{t\wedge \tau_m}  g(u(r))\, d W(r),  \quad \mbox{ in } \rH, \;\;
\mathbb{P} \text {-a.s. }
\end{aligned}
\end{equation}
are  equivalent.
\end{prop}
\begin{proof}
Let us choose and fix $m\in\mathbb{N}$.
By part (iii) of \autoref{def-local solution},
\[
\mathbb{E}\Big(\sup_{s\in[0,t\wedge\tau_m]}|u(s)|_\rV^2
+\int_0^{t\wedge\tau_m}|u(s)|_{D(\rA)}^2\,ds\Big)<\infty,\quad \forall\,t\geq 0,
\]
which guarantees that all Bochner and stochastic integrals below are well-defined and that the (stochastic) Fubini theorems we invoke apply.

Firstly, we aim to show that \eqref{eq-local strong solution} implies \eqref{eq-local mild solution}.
Assume that $u$ satisfies \eqref{eq-local strong solution}.
Let us choose and fix $\xi\in D(\rA )$ and $t>0$. Set
\[
f(s) := S(t\wedge\tau_m-s)\xi,\;\;s\in[0,t\wedge\tau_m].
\]
Therefore, $f\in C^1([0,t\wedge\tau_m];D(\rA))$. By the product rule,
\[
d\langle f(s),u(s)\rangle
= \langle f^\prime(s),u(s)\rangle\,ds + \langle f(s),du(s)\rangle.
\]
Thus, in view of \eqref{eq-local strong solution}, for every $s\in[0,t\wedge\tau_m]$
\begin{equation}
\begin{aligned}
d\langle f(s),u(s)\rangle&=\langle f^\prime(s),u(s)\rangle\,ds-\langle f(s),\rA u(s)\rangle\, ds\\
&-\langle f(s),B( u(s))\rangle\, ds+\langle f(s),g( u(s))\, dW(s)\rangle, \;\;
\mathbb{P} \text {-a.s. }.
\end{aligned}
\end{equation}
Integrating from $0$ to $t\wedge\tau_m$, it implies that
\begin{align*}
\langle &f(t\wedge\tau_m),u(t\wedge\tau_m)\rangle - \langle f(0),u_0\rangle
= \int_0^{t\wedge\tau_m} \langle f^\prime(s),u(s)\rangle \,ds
   - \int_0^{t\wedge\tau_m} \langle f(s),\rA u(s)\rangle \,ds \\
& -\int_0^{t\wedge\tau_m} \langle f(s),B(u(s))\rangle \,ds
   + \int_0^{t\wedge\tau_m} \langle f(s),g(u(s))\,dW(s)\rangle, \;\;
\mathbb{P} \text {-a.s. } .
\end{align*}
Note that $f(t\wedge\tau_m)=\xi$, $f(0)=S(t\wedge\tau_m)\xi$ and $f^\prime(s)=\rA S(t\wedge\tau_m-s)\xi$. Hence,
\begin{align*}
&\langle u(t\wedge\tau_m),\xi\rangle - \langle S(t\wedge\tau_m)\xi,u_0\rangle\\
&= \int_0^{t\wedge\tau_m} \langle \rA S(t\wedge\tau_m-s)\xi,u(s)\rangle ds - \int_0^{t\wedge\tau_m} \langle S(t\wedge\tau_m-s)\xi,\rA u(s)\rangle \,ds\\
&- \int_0^{t\wedge\tau_m} \langle S(t\wedge\tau_m-s)\xi,B(u(s))\rangle \,ds\\
&+ \int_0^{t\wedge\tau_m} \langle \mathbf{1}_{\left[0, \tau_m\right)}(s)S(t\wedge\tau_m-s)\xi,g(u(s))\,dW(s)\rangle, \;\;
\mathbb{P} \text {-a.s. } .
\end{align*}
Recall that $\rA$ is self-adjoint and the semigroup $S(t)=e^{-t\rA}$ generated by $-\rA$ is also self-adjoint. Therefore
\begin{equation}\label{eq-3.17}
\begin{aligned}
\langle u(t\wedge\tau_m),\xi\rangle
&= \langle S(t\wedge\tau_m)u_0,\xi\rangle -  \langle \int_0^{t\wedge\tau_m}S(t\wedge\tau_m-s)B(u(s)) \,ds,\xi\rangle
\\
&
+  \langle \int_0^{t\wedge\tau_m}\mathbf{1}_{\left[0, \tau_m\right)}(s)S(t\wedge\tau_m-s)g(u(s))\,dW(s),\xi\rangle , \;\;
\mathbb{P} \text {-a.s. }.
\end{aligned}
\end{equation}
Since the above identity holds for all $\xi\in D(\rA)$, and $D(\rA)$ is dense in $\rH$, we infer that \eqref{eq-3.17} holds for every $\xi\in \rH$, which implies \eqref{eq-local mild solution} holds.

In the following, we aim to show that \eqref{eq-local mild solution} implies \eqref{eq-local strong solution}. Suppose $u$ is a mild solution satisfying \eqref{eq-local mild solution}.

Let us choose and fix $\xi\in D(\rA)$ and $t\leq \tau_m$. At first, we aim to compute $\int_0^t \langle u(s),\rA \xi\rangle \,ds$.
In view of \eqref{eq-local mild solution}, taking the inner product with $\rA\xi$ gives
\begin{equation}\label{eq-pairing-1}
\begin{aligned}
\langle u(t),\rA\xi\rangle
&= \langle S(t)u_0,\rA\xi\rangle
  - \int_0^t \langle S(t-s)B(u(s)),\rA\xi\rangle\, ds\\
  &+ \int_0^t \langle S(t-s)g(u(s))\,dW(s),\rA\xi\rangle , \;\;
\mathbb{P} \text {-a.s. }.
\end{aligned}
\end{equation}

Recall \cite[Theorem 2.4(b)]{Pazy_1983}, for every $x\in\rH$
\begin{align}\label{eq-pazy thm2.4}
\rA \left(\int_0^t S(s)x\,ds \right)=x-S(t)x.
\end{align}
By the self-adjointness of $\rA$ and equality \eqref{eq-pazy thm2.4} we infer that
\begin{equation}\label{eq-semigroup1}
\begin{aligned}
\int_0^t \langle S(s)u_0,\rA \xi\rangle \,ds
&= \langle \int_0^t S(s)u_0\,ds,\rA \xi\rangle = \langle \rA \int_0^t S(s)u_0\,ds,\xi\rangle \\
 &= \langle u_0 - S(t)u_0,\xi\rangle.
\end{aligned}
\end{equation}
Next, by the deterministic Fubini theorem, the self-adjointness of $\rA$ and equality \eqref{eq-pazy thm2.4}, we infer that
\begin{align}\label{eq-semigroup2}
&\int_0^t\int_0^s \langle S(s-r)B(u(r)),\rA\xi\rangle\, dr\,ds=\int_0^t \Big\langle \int_0^s S(s-r)B(u(r))\,dr, \rA \xi\Big\rangle \,ds\\
&=\int_0^t \Big\langle \int_r^t S(s-r)B(u(r))\,ds,  \rA\xi\Big\rangle \,dr=\int_0^t \Big\langle \rA\int_0^{t-r} S(s)B(u(r))\,ds,  \xi\Big\rangle \,dr\nonumber\\
&= \int_0^t\Big\langle  B(u(r)) -  S(t-r)B(u(r)), \xi\Big\rangle\,dr.
\nonumber
\end{align}
For the stochastic term, we use the stochastic Fubini theorem (\cite[Theorem 4.33]{DaPrato+Zabczyk_2014}) together with the self-adjointness of $\rA$ and equality \eqref{eq-pazy thm2.4}
\begin{align}\label{eq-semigroup3}
&\int_0^t\int_0^s \Big\langle  S(s-r)g(u(r))\,dW(r), \rA \xi\Big\rangle \,ds=\int_0^t \Big\langle \int_0^s S(s-r)g(u(r))\,dW(r), \rA \xi\Big\rangle \,ds\nonumber\\
&= \int_0^t\int_r^t \Big\langle  S(s-r)g(u(r)),  \rA\xi\Big\rangle \,ds\,dW(r)= \int_0^t \Big\langle \rA \int_0^{t-r}S(s)g(u(r))\,ds,  \xi\Big\rangle \,dW(r)\nonumber\\
&=\int_0^t\Big\langle  g(u(r))-S(t-r)g(u(r)), \xi\Big\rangle\,dW(r).
\end{align}

Combining \eqref{eq-pairing-1} and \eqref{eq-semigroup1}--\eqref{eq-semigroup3}, and recalling that $\rA$ is self-adjoint, we deduce
\begin{equation}\label{eq-Astarpair}
\begin{aligned}
&\int_0^t \langle u(s),\rA \xi\rangle \,ds\\&=\int_0^t \langle S(s)u_0,\rA \xi\rangle \,ds-\int_0^t\int_0^s \langle S(s-r)B(u(r)),\rA \xi\rangle \,dr\,ds\\
&+\int_0^t\int_0^s \langle S(s-r)g(u(r))\,dW(r),\rA \xi\rangle \,ds\\
&= \langle u_0 - S(t)u_0,\xi\rangle
+ \Big\langle \int_0^t S(t-r)B(u(r))\,dr,\xi\Big\rangle - \Big\langle\int_0^t B(u(r))\,dr,\xi\Big\rangle\\
  &+ \int_0^t\Big\langle  g(u(r))\,dW(r), \xi\Big\rangle-\int_0^t\Big\langle S(t-r)g(u(r))\,dW(r), \xi\Big\rangle,\;\;\mathbb{P} \text {-a.s. },
\end{aligned}
\end{equation}
which implies
\begin{equation}\label{eq-Astarpair-1}
\begin{aligned}
&\langle S(t)u_0,\xi\rangle
  - \int_0^t \langle S(t-s)B(u(s)),\xi\rangle ds
  + \int_0^t \langle S(t-s)g(u(s))\,dW(s),\xi\rangle\\
&=-\int_0^t \langle u(s),\rA \xi\rangle \,ds+\langle u_0,\xi\rangle
  - \int_0^t \langle B(u(s)),\xi\rangle \,ds\\
&+ \int_0^t \langle g(u(s))\,dW(s),\xi\rangle,\;\;\mathbb{P} \text {-a.s. }.
\end{aligned}
\end{equation}
Therefore, in view of \eqref{eq-local mild solution} and \eqref{eq-Astarpair-1} we infer that
\begin{align}
&\langle u(t),\xi\rangle\\
&= \langle S(t)u_0,\xi\rangle
  - \int_0^t \langle S(t-s)B(u(s)),\xi\rangle ds
  + \int_0^t \langle S(t-s)g(u(s))\,dW(s),\xi\rangle\nonumber\\
&= \langle u_0,\xi\rangle
  - \int_0^t \langle u(s),\rA \xi\rangle ds
  - \int_0^t \langle B(u(s)),\xi\rangle ds
  + \int_0^t \langle g(u(s))\,dW(s),\xi\rangle,\;\;\mathbb{P} \text {-a.s. }.\nonumber
\end{align}
Since this holds for all $\xi\in D(\rA)$ and $t\leq\tau_m$, and $D(\rA)$ is dense in $\rH$, it follows that for $t\leq\tau_m$,
\[
u(t) = u_0 - \int_0^t \rA u(s)\,ds - \int_0^t B(u(s))\,ds + \int_0^t g(u(s))\,dW(s),\;\;\mathbb{P} \text {-a.s. }
\mbox{ in } \rH.
\]
Replacing $t$ by $t\wedge\tau_m$, we infer that \eqref{eq-local strong solution} holds.
The proof of \autoref{prop-equivalence} is complete.
\end{proof}

The above result is important whenever we apply It\^o's formula.

\begin{definition}[Maximal and global solution]\label{def-local solution maximal}
 A local solution $(u, \tau)$ is called maximal local solution if and only if for all other local solution $(\rv, \sigma)$, the following two conditions are satisfied $\mathbb{P}$-a.s.:
\begin{trivlist}
\item[(i)] $\sigma \leq \tau$;
\item[(ii)]$u\vert_{[0, \sigma)}=\rv$.
\end{trivlist}

If $(u, \tau)$ is a maximal local solution to  (\ref{eqn-NS03}), the stopping time $\tau$ is called its lifetime. The existence of a maximal local solution requires the uniqueness of local solutions.

Finally, a local solution $(u, \tau)$ is called a global solution if and only if  $\tau=\infty$.
 \end{definition}

\subsection{Main Theorem}
We are now ready to state the main result, which will be proved in \autoref{subsec-Poincar\'e domain-global} and \autoref{subsec-Full space-global}, respectively.
\begin{theorem}\label{thm-main}
Assume that \autoref{ass-prob} and \autoref{ass-g-loc Lip} are satisfied, and one of the following two assumptions is satisfied.
\begin{trivlist}
\item[(i)] Assume that $\dom$ is a bounded domain or an unbounded domain satisfying the Poincar\'e condition and function $\phi$ satisfies \autoref{assumption-phi}.
\item[(ii)] Assume that $\dom=\mathbb{R}^3$ and function
$\phi$ satisfies \autoref{assumption-phi} and also \autoref{assumption-phi-1}.
\end{trivlist}
Let $u_0$ be an $\mathscr{F}_0$-measurable $\rV$-valued random variable with $\mathbb{E}\vert u_0\vert_\rV^2 < \infty$.
Then there exists a unique global solution to the main problem \eqref{eqn-NS03}.
\end{theorem}

\section{Approximated problem}\label{sec-A priori estimates}
To prepare for the proof of local existence of solutions to \eqref{eqn-NS03}, we first introduce an approximated problem.

From this section to \autoref{subsec-local maximal}, we do not use the special form \eqref{eqn-g} of function $g$ and we use Assumption \ref{ass-g-loc Lip}.

Assumption \ref{ass-g-loc Lip}  implies the following boundedness on balls condition.
\begin{lemma}\label{lem-g-boundedness on balls}
For every $R>0$, there exists a constant $C_6:=C_6(R)>0$ such that
for all $u \in V$, if $ \vert u\vert_{\rV} \leq R$  then
\begin{align}\label{eqn-g bounded-100}
 \vert g(u) \vert_\rV  \leq C_6\vert u\vert_\rV.
\end{align}
\end{lemma}
\begin{proof}[Proof of \autoref{lem-g-boundedness on balls}]
Note that $g(0)=0$. By Assumption \ref{ass-g-loc Lip}, we infer that
\begin{align*} \vert g(u) \vert_\rV  &=
 \vert g(0)+ g(u) -g(0)\vert_\rV
 \leq
\vert g(0)\vert_\rV  +  \vert g(u) -g(0)\vert_\rV
 \\ &=  \vert g(u) -g(0)\vert_\rV \leq C_6\vert u-0\vert_\rV= C_6\vert u\vert_\rV.
\end{align*}
The proof is complete.
\end{proof}

As a first step, we define a smooth truncation function that will be used later.

\begin{definition}\label{def-theta}
Let $\theta: \mathbb{R}_{+} \rightarrow[0,1]$ be a non-increasing smooth function with compact support such that
\begin{align}
&\inf _{x \in \mathbb{R}_{+}} \theta^{\prime}(x) \geq-2,\label{eqn-theta prime}\\
&\theta(x)=
\begin{cases}
1,\;\;x \in[0,1]\\
0,\;\;x \in[2, \infty).
\end{cases} \label{eq-thata}
\end{align}

For $n \geq 1$ set
\[
\theta_n(\cdot):=\theta(\frac{\cdot}{n}),
\]
\end{definition}
We have the following easy lemma about $\theta$ as a consequence of previous description, see e.g.
\cite[Lemma 4.3]{Brz+Millet_2014}.

\begin{lemma}\label{lemma-theta}
If $h: \mathbb{R}_{+} \rightarrow \mathbb{R}_{+}$is a non-decreasing function, then for every $x, y \in \mathbb{R}$,
\begin{align}
\theta_n(x) h(x) \leq h(2 n), \quad\vert\theta_n(x)-\theta_n(y)\vert \leq \frac{2}{n}\vert x-y\vert.
\end{align}
\end{lemma}

Recall that $\rA$ is the Stokes operator defined in \eqref{operator A}. $S(t)=e^{-t\rA}, t \geq 0$ is the analytic semigroup generated by operator $-\rA$.

We consider the following approximated problem to \eqref{eqn-NS03}:
\begin{equation}\label{NS-approximated}
\begin{aligned}
u^n(t)= & S(t) u_0-\int_0^t S(t-r) \theta_n\left(\vert u^n\vert_{X_r}\right) B\left(u^n(r)\right) d r \\
& +\int_0^t S(t-r)\theta_n\left(\vert u^n\vert_{X_r}\right) g\left(u^n(r)\right) \,d W(r), \quad t \in[0, T]
\end{aligned}
\end{equation}

In the following, we summarize several important properties of the semigroup $S(t)$ for $t \geq 0$ that will be used later, see e,g. \cite{Brz+Hussain_2020}, \cite{Agranovich+Vishik_1964}, and \cite{Pardoux_1979}.

\begin{prp}\label{prp-S}
Assume that $T_0>0$. Then there exists  constants $C_7=C_7(T_0)$, $C_8=C_8(T_0)$ and $C_9=C_9(T_0)$ such that the following holds.
\begin{trivlist}
\item[(i)]  For every $T \in (0,T_0]$ and $f \in L^2(0, T ; \mathrm{H})$, function $u=S\ast f$ defined by
\[
u(t)=\int_0^t S(t-r) f(r) \,d r, \quad t \in[0, T]
\]
belongs to $X_T$ and satisfies
\begin{align}\label{eqn-S*f norm}
\vert u\vert _{X_T}^2 \leq C_7\Vert f\Vert _{L^2(0, T ; \rH)}^2.
\end{align}
Note that the map $S \ast: f \mapsto S \ast f$ from $L^2(0, T ; \rH)$ into $X_T$ is linear and bounded.
\item[(ii)]For every $T \in (0,T_0]$ and every process $\xi \in \mathbb{M}^2(0, T ; \rV)$, the process $u=S \diamond \xi$ defined by
\[
u(t)=\int_0^t S(t-r) \xi(r) \,d W(r), \quad t \in[0, T]
\]
belongs to $\mathbb{M}^2\left(X_T\right)$ and satisfies
\begin{align}\label{eqn-S diamond xi norm}
\vert u\vert_{\mathbb{M}^2\left(X_T\right)}^2 \leq C_8\vert \xi\vert_{\mathbb{M}^2(0, T ; \rV)}^2 .
\end{align}
Note that $S \diamond: \xi \mapsto S \diamond \xi$ is a linear and bounded map from $\mathbb{M}^2(0, T ; \rV)$ into $\mathbb{M}^2\left(X_T\right)$.
\item[(iii)] For every $T \in (0,T_0]$ and every $u_0 \in \rV$, a function $u=S u_0$ defined by
\begin{align}
u(t)=S(t) u_0, \quad t \in[0, T]
\end{align}
belongs to $X_T$ and satisfies
\begin{align}\label{eqn-10.14}
\vert u\vert_{X_T} \leq C_9\vert u_0\vert_\rV.
\end{align}
\end{trivlist}
Moreover, if the domain $\dom$ is a bounded domain or an unbounded domain satisfying the Poincar\'e condition, then the above is true with $T_0=\infty$.
\end{prp}
\begin{remark}\label{rem-S}
Let us point out that the above conditions are classical. Conditions (i) and (iii) are true for a self-adjoint operator $A$ in separable Hilbert space, see the book \cite[Chapter 1]{Vishik+Fursikov_1988}, see also \cite[Theorem 2]{Brz_1991}. Condition (ii) follows from Pardoux \cite{Pardoux_1979} in the Gelfand triple $D(A) \subset D(A^{1/2}) \subset H$.
\end{remark}

Before we discuss the following lemmas and properties, we define the following stopping times which we will use in the rest of \autoref{sec-A priori estimates}.

For every $T>0$, let us define the following functions, which incorrectly can be called "stopping times":
\begin{equation}\label{eq-tau_u}
\begin{aligned}
\tau_{2n}: X_T \ni u\mapsto \tau_{2n}(u) &:=\inf \{t\in[0,T]: \vert u\vert_{X_t}\geq 2n \}\wedge T.
\end{aligned}
\end{equation}
Since, for every $u \in X_T$ the map $[0, T] \ni t \mapsto\vert u\vert_{X_t}$ is a non-decreasing function, we deduce that
\begin{align}\label{eqn-1098}
\vert u\vert_{X_t}\geq 2n,\;\; \mbox{ if } t>\tau_{2n}(u)  \mbox{  and  }\vert u\vert_{X_t}\leq 2n,\;\; \mbox{ if }  t\leq\tau_{2n}(u).
\end{align}
Hence,  by  the definition of $\theta_n$, see \autoref{def-theta}, we infer that
\begin{align}\label{eq-theta=0}
\theta_n(\vert u\vert_{X_t})=0, \text{ for } t>\tau_{2n}(u),
\end{align}
and
\begin{align}\label{eqn-theta leq 1}
\vert\theta_n(\vert u\vert_{X_t})\vert\leq 1,\;\;t\in[0,\tau_{2n}(u)].
\end{align}
Moreover, by \eqref{eqn-1098} and the definition of the norm $\vert\cdot\vert_{X_T}$, see \eqref{norm X_T}, we infer
\begin{equation}\label{eqn-u_V DA-bound}
\begin{aligned}
&\sup _{0\leq t\leq \tau_{2n}(u)}\vert u(t) \vert_\rV^2\leq 4n^2,\\
&\int_0^{\tau_{2n}(u)} \vert u(t)\vert_{D(\rA)}^2 \,d t\leq 4n^2.
\end{aligned}
\end{equation}

In the following, we show that the truncated version of the nonlinear term $B$ is both bounded and Lipschitz continuous.

\begin{lemma}\label{lem-Phi_T,F^n-bounded}
Define $\Phi_{T,B}^n:X_T \rightarrow L^2(0,T;\rH)$ by
\begin{align}\label{eq-Phi_T,B}
(\Phi_{T,B}^nu)(t):=\theta_n(\vert u\vert_{X_t})B(u(t)),\;\;t\in[0,T].
\end{align}
Then,
\begin{trivlist}
\item[(i)] the map
$\Phi_{T,B}^n(u)$ is bounded, i.e. there exist a constant $C >0$, such that
for every $u \in X_T$,
\begin{align}\label{eqn-Phi_{T,F}^n norm}
\Vert\Phi_{T,B}^n(u)\Vert_{L^2(0,T;\rH)} \leq Cn^2T^{1/4}.
\end{align}
\item[(ii)]
$\Phi_{T,B}^n(u)$ is Lipschitz. More precisely,   there exists a constant $C>0$ such that
for all $u,\rv \in X_T$,
\begin{align}\label{eqn-Phi_{T,F}^n Lipschitz}
\Vert \Phi_{T,B}^n(u)-\Phi_{T,B}^n(\rv)\Vert_{L^2(0,T;\rH)} \leq C T^{1/4}n \vert u-\rv\vert_{X_T}.
\end{align}
\end{trivlist}
\end{lemma}
\begin{proof}[Proof of (i) in \autoref{lem-Phi_T,F^n-bounded}]
Inequality \eqref{eqn-Fu-0} implies that
\begin{equation}\label{eqn-4.29}
\begin{aligned}
\Vert\Phi_{T,B}^n(u)\Vert_{L^2(0,T;\rH)}^2&= \int_0^T \vert \theta_n(\vert u\vert_{X_t})\vert^2 \vert B(u(t))\vert_\rH^2\,dt\\
&\leq C\int_0^T \vert\theta_n(\vert u\vert_{X_t})\vert^2\vert u (t)\vert _{D(\rA)}\vert u(t)\vert_\rV^{3}\,dt.
\end{aligned}
\end{equation}
By \eqref{eq-theta=0} and \eqref{eqn-theta leq 1} we infer
\begin{equation}\label{eqn-4.32}
\begin{aligned}
\int_0^T \vert \theta_n(\vert u\vert_{X_t})\vert^2\vert u (t)\vert _{D(\rA)}\vert u(t)\vert_\rV^{3}\,dt&=\int_0^{\tau_{2n}(u)} \vert \theta_n(\vert u\vert_{X_t})\vert^2\vert u  (t)\vert _{D(\rA)}\vert u (t) \vert_\rV^{3}\,dt\\
&\leq \int_0^{\tau_{2n}(u)} \vert u  (t)\vert _{D(\rA)}\vert u (t) \vert_\rV^{3}\,dt.
\end{aligned}
\end{equation}
Then by inequalities \eqref{eqn-4.29}, \eqref{eqn-4.32} and \eqref{eqn-u_V DA-bound}, and the H\"{o}lder inequality we infer that
\begin{equation}\label{eqn-Phi_T,B}
\begin{aligned}
\Vert\Phi_{T,B}^n(u)\Vert_{L^2(0,T;\rH)}^2
&\leq C\int_0^{\tau_{2n}(u)} \vert u(t) \vert _{D(\rA)}\vert u (t) \vert_\rV^{3}\,dt\\
&\leq C\sup_{t\in [0,\tau_{2n}(u))} \vert u(t) \vert_\rV^3\int_0^{\tau_{2n}(u)} \vert u(t) \vert _{D(\rA)}\,dt\\
&\leq C T ^{1/2}\sup_{t\in [0,\tau_{2n}(u))} \vert u(t) \vert_\rV^3(\int_0^{\tau_{2n}(u)} \vert u(t)\vert_{D(\rA)}^2\,dt)^{1/2}\\
&\leq CT ^{1/2}\vert u\vert_{X_{\tau_{2n}(u)}}^4\leq C n^4 T ^{1/2}.
\end{aligned}
\end{equation}
We complete the proof of (i) in \autoref{lem-Phi_T,F^n-bounded}
\end{proof}
\begin{proof}[proof of (ii) in \autoref{lem-Phi_T,F^n-bounded}]
Without loss of generality we assume that $\tau_{2n}(u) \geq \tau_{2n}(\rv)$. Then by \eqref{eq-theta=0} we infer
\begin{align*}
&\theta_n(\vert u\vert_{X_t})=\theta_n(\vert \rv\vert_{X_t})=0,\;\;t>\tau_{2n}(u),\\
&\theta_n(\vert \rv\vert_{X_t})=0,\;\;t>\tau_{2n}(\rv).
\end{align*}
Therefore, by the Minkowski inequality, \autoref{lemma-theta} and \eqref{eqn-theta leq 1}, we infer
\begin{align}\label{eqn-Phi_T,B-difference}
&\Vert \Phi_{T,B}^n(u)-\Phi_{T,B}^n(\rv)\Vert_{L^2(0,T;\rH)}\\
&= \left(\int_0^{T} \vert(\theta_n(\vert u\vert_{X_t})-\theta_n(\vert \rv\vert_{X_t})) B(u(t))+\theta_n(\vert \rv\vert_{X_t}) (B(u(t))-B(\rv(t)))\vert_\rH^2\,dt\right)^{1/2}\nonumber\\
&\leq \left(\int_0^{\tau_{2n}(u)} \vert\theta_n(\vert u\vert_{X_t})-\theta_n(\vert \rv\vert_{X_t})\vert^2 \vert B(u(t))\vert_\rH^2\,dt\right)^{1/2}\nonumber\\
&+\left(\int_0^{\tau_{2n}(\rv)} \vert \theta_n(\vert \rv\vert_{X_t})\vert^2 \vert B(u(t))-B(\rv(t))\vert_\rH^2\,dt\right)^{1/2}\nonumber\\
&\leq \left(\int_0^{\tau_{2n}(u)} \frac{4}{n^2} \Big\vert\vert u\vert_{X_t}-\vert \rv\vert_{X_t}\Big\vert^2\vert B(u(t))\vert_\rH^2\,dt\right)^{1/2}+\left(\int_0^{\tau_{2n}(\rv)}  \vert B(u(t))-B(\rv(t))\vert_\rH^2\,dt\right)^{1/2}.\nonumber
\end{align}
By \autoref{lem-inequality}, the triangle inequality, the H\"older inequality and \eqref{eqn-u_V DA-bound}, we infer that
\begin{equation}\label{eqn-I_1}
\begin{aligned}
&\int_0^{\tau_{2n}(u)} \frac{4}{n^2} \Big\vert\vert u\vert_{X_t}-\vert \rv\vert_{X_t}\Big\vert^2\vert B(u(t))\vert_\rH^2\,dt\\
&\leq C\int_0^{\tau_{2n}(u)}\frac{1}{n^2} \Big\vert\vert u\vert_{X_t}-\vert \rv\vert_{X_t}\Big\vert^2 \vert u(t) \vert _{D(\rA)}\vert u(t) \vert _\rV^{3}\,dt\\
&\leq C\frac{1}{n^2}\vert u-\rv\vert_{X_T}^2 \sup_{0\leq t\leq \tau_{2n}(u)}\vert u(t) \vert _\rV^3\int_0^{\tau_{2n}(u)} \vert u(t) \vert _{D(\rA)}\,dt\\
&\leq CT^{1/2}\frac{1}{n^2}\vert u-\rv\vert_{X_T}^2 \sup_{0\leq t\leq \tau_{2n}(u)}\vert u(t) \vert _\rV^3\left(\int_0^{\tau_{2n}(u)} \vert u(t) \vert ^2_{D(\rA)}\,dt\right)^{1/2}\\
&\leq CT^{1/2}n^2\vert u-\rv\vert_{X_T}^2.
\end{aligned}
\end{equation}

By  \autoref{lem-inequality}, we infer that
\begin{align*}
&\int_0^{\tau_{2n}(\rv)}  \vert B(u(t))-B(\rv(t))\vert_\rH^2\,dt\\
&\leq C\int_{0}^{\tau_{2n}(\rv)}
(\vert u(t) \vert _{D(\rA)}^{1/2}\vert u (t)\vert_\rV^{1/2}\vert u(t)-\rv(t) \vert_\rV+\vert u(t)-\rv(t) \vert _{D(\rA)}^{1/2}\vert u(t)-\rv(t) \vert _\rV^{1/2}\vert \rv (t)\vert_\rV )^2 \,dt\\
&\leq C\int_{0}^{\tau_{2n}(\rv)}
\vert u(t) \vert _{D(\rA)}\vert u(t) \vert _\rV\vert u(t)-\rv(t) \vert _\rV^2 \,dt\\
&+C\int_{0}^{\tau_{2n}(\rv)} \vert u(t)-\rv(t) \vert _{D(\rA)}\vert u(t)-\rv(t) \vert _\rV\vert \rv(t) \vert _\rV^2 \,dt\\
&\leq C \sup_{0\leq t\leq \tau_{2n}(\rv)} \vert u(t) \vert _\rV \sup_{0\leq t\leq \tau_{2n}(\rv)} \vert u(t)-\rv(t) \vert _\rV^2 \int_{0}^{\tau_{2n}(\rv)}
\vert u(t) \vert _{D(\rA)}\,dt\\
&+C \sup_{0\leq t\leq \tau_{2n}(\rv)} \vert u(t)-\rv (t)\vert _\rV  \sup_{0\leq t\leq \tau_{2n}(\rv)} \vert \rv(t) \vert_\rV^2
\int_{0}^{\tau_{2n}(\rv)}
\vert u(t)-\rv(t) \vert _{D(\rA)} \,dt.
\end{align*}
Therefore, by the H\"older inequality and \eqref{eqn-u_V DA-bound} we infer
\begin{align}\label{eqn-I_2}
&\int_0^{\tau_{2n}(\rv)}  \vert B(u(t))-B(\rv(t))\vert_\rH^2\,dt\\
\nonumber
&\leq CT^{1/2} \sup_{0\leq t\leq \tau_{2n}(\rv)} \vert u(t) \vert _\rV \sup_{0\leq t\leq \tau_{2n}(\rv)} \vert u(t)-\rv(t) \vert_\rV^2 \left(\int_{0}^{\tau_{2n}(\rv)}
\vert u(t) \vert ^2_{D(\rA)}\,dt\right)^{1/2}\\
\nonumber
&+CT^{1/2} \sup_{0\leq t\leq \tau_{2n}(\rv)} \vert u(t)-\rv (t)\vert_\rV  \sup_{0\leq t\leq \tau_{2n}(\rv)} \vert \rv(t) \vert_\rV^2
\left(\int_{0}^{\tau_{2n}(\rv)}
\vert u(t)-\rv(t) \vert ^2_{D(\rA)} \,dt\right)^{1/2}\\
\nonumber
&\leq CT^{1/2}n^2 \vert u-\rv \vert_{X_T}^2.
\end{align}

By \eqref{eqn-Phi_T,B-difference}, \eqref{eqn-I_1} and \eqref{eqn-I_2} we conclude that
\begin{equation}
 \begin{aligned}
\Vert \Phi_{T,B}^n(u)-\Phi_{T,B}^n(\rv)\Vert_{L^2(0,T;\rH)}
&\leq C T^{1/4}n \vert u-\rv\vert_{X_T}.
 \end{aligned}
\end{equation}
This proves the desired Lipschitz bound.
\end{proof}

We now prove that the truncated diffusion term in \eqref{NS-approximated} is bounded and Lipschitz continuous.
\begin{prp}\label{prp-phi g-1}
Define $\Phi_{T,g}^n:X_T \rightarrow L^2(0,T;\rV)$ by
\begin{align}\label{eq-Phi_{T,g}^n-1}
\Phi_{T,g}^n(u)(t):= \theta_n(\vert u\vert_{X_t})g(u(t)),\;\;t\in[0,T].
\end{align}
Then $\Phi_{T,g}^n(u)$ is bounded and Lipschitz. More precisely,
for all  $u,\rv \in X_T$, there exists constants $C=C(n)>0$ such that
\begin{align}
\Vert\Phi_{T,g}^n(u)\Vert_{L^2(0,T;\rV)}^2 &\leq C T,  \label{eqn-Phi_{T,g}^n norm bound-1}\\
\Vert \Phi_{T,g}^n(u)-\Phi_{T,g}^n(\rv)\Vert_{L^2(0,T;\rV)}& \leq CT^{1/2}\vert u-\rv
\vert_{X_T}.\label{eqn-Phi_{T,g}^n Lipschitz-1}
\end{align}
\end{prp}
\begin{proof}
Let us choose and fix $u\in X_T$, by the definition of the stopping times \eqref{eq-tau_u}, we infer
\begin{align}\label{eqn-u Xtau norm}
\vert u\vert_{X_{\tau_{2n}(u)}}\leq 2n.
\end{align}
By definition of $X_{\tau_{2n}(u)}$ we infer that
\begin{align}\label{eqn-u Xtau norm-2}
\sup_{t \in [0, \tau_{2n}(u)]} \vert u(t)\vert_\rV\leq 2n    .
\end{align}
Thus,  by property  \eqref{eqn-g bounded-100} we infer that
\begin{align}\label{eqn-Phi bound}
\sup_{t\in[0,\tau_{2n}(u)]}\vert g(u(t))\vert_\rV   \leq 2nC_6(2n).
\end{align}
Therefore, by \eqref{eq-theta=0}, \eqref{eqn-theta leq 1} and \eqref{eqn-u_V DA-bound}, there exists $C=C(n)>0$ such that
\begin{equation}
\begin{aligned}
\Vert\Phi_{T,g}^n(u)\Vert_{L^2(0,T;\rV)}^2&= \int_0^T\Big \vert \theta_n(\vert u\vert_{X_t})g(u(t))\Big \vert_\rV^2\,dt   \leq \int_0^{\tau_{2n}(u)} \vert g(u (t)) \vert_\rV^2\,dt \\
&\leq T 2nC_6(2n)=:CT.
\end{aligned}
\end{equation}

For the second inequality, let us choose and fix $u\in X_T$.
Without loss of generality we assume $\tau_{2n}(u) \geq \tau_{2n}(\rv)$. By \eqref{eq-theta=0}, \eqref{eqn-theta leq 1} and the Minkowski inequality we obtain
\begin{align}\label{eqn-4.34}
&\Vert \Phi_{T,g}^n(u)-\Phi_{T,g}^n(\rv)\Vert_{L^2(0,T;\rV)}\\
&= \left(\int_0^T\Big \vert (\theta_n(\vert u\vert_{X_t})-\theta_n(\vert \rv\vert_{X_t}))g(u(t))-\theta_n(\vert \rv\vert_{X_t})(g(u(t))-g(\rv(t)))\Big \vert_\rV^2\,dt\right)^{1/2} \nonumber\\
&\leq \left(\int_0^{\tau_{2n}(u) }\vert \theta_n(\vert u\vert_{X_t})-\theta_n(\vert \rv\vert_{X_t}) \vert^2\vert g(u(t))\vert_\rV^2\,dt\right)^{1/2}+\left(\int_0^{\tau_{2n}(\rv) }\vert g(u(t))-g(\rv(t)) \vert_\rV^2\,dt\right)^{1/2}.\nonumber
\end{align}
In view of \autoref{lemma-theta}, the triangle inequality, inequalities \eqref{eqn-u_V DA-bound} and \eqref{eqn-Phi bound}, there exists $C=C(n)>0$ such that
\begin{align}\label{eqn-I11}
\int_0^{\tau_{2n}(u) }\vert \theta_n(\vert u\vert_{X_t})-\theta_n(\vert \rv\vert_{X_t}) \vert^2\vert g(u(t))\vert_\rV^2\,dt&\leq \int_0^{\tau_{2n}(u) }\frac{4}{n^2}\Big\vert \vert u\vert_{X_t}-\vert\rv
\vert_{X_t}\Big\vert^2\vert g(u(t))\vert^2\,dt\nonumber\\
&\leq CT\vert u-\rv
\vert_{X_T}^2\sup_{0\leq t\leq \tau_{2n}(u)}\vert g(u(t))\vert_{\rV}^2  \nonumber\\
&\leq CT\vert u-\rv
\vert_{X_T}^2.
\end{align}
By \autoref{ass-g-loc Lip} we infer that there exists $C=C(n)>0$ such that
\begin{align}\label{eqn-I21}
\int_0^{\tau_{2n}(\rv) }\vert g(u(t))-g(\rv(t)) \vert_\rV^2\,dt&\leq C\int_0^{\tau_{2n}(\rv) }\vert u(t)-\rv(t)\vert_\rV^2\,dt\leq CT\sup_{t\in[0,T]}\vert u(t)-\rv(t)
\vert_{\rV}^2 \nonumber\\
&\leq CT\vert u-\rv
\vert_{X_T}^2.
\end{align}
Therefore, substituting \eqref{eqn-I21} and \eqref{eqn-I11} into
\eqref{eqn-4.34} yields \eqref{eqn-Phi_{T,g}^n Lipschitz-1}.
We complete the proof of \autoref{prp-phi g-1}.
\end{proof}

Combining the semigroup estimates, see \autoref{prp-S}, with the above Lipschitz estimates, we define the map $\Psi_{T,u_0}^n$, which becomes a contraction for sufficiently small $T$.

\begin{prp}\label{prop-Lipschitz of Psi}
Define a map $\Psi_{T, u_0}^n$ : $\mathbb{M}^2\left(X_T\right) \rightarrow \mathbb{M}^2\left(X_T\right)$ by
\begin{align}\label{eq-Psi_{T, u_0}^n}
\Psi_{T, u_0}^n(u)=S u_0-S \ast \Phi_{T, B}^n(u)+S \diamond \Phi_{T,g}^n(u),
\end{align}
where $S \ast \cdot$ and $S \diamond\cdot$ is defined in \autoref{prp-S}. Then there exists $C_n^\ast:=C_n^\ast(n)>0$ such that for all $u, \rv \in \mathbb{M}^2\left(X_T\right)$,
\begin{align}\label{eqn-Psi_{T, u_0}^n contraction}
\vert\Psi_{T, u_0}^n\left(u\right)-\Psi_{T, u_0}^n\left(\rv\right)\vert_{\mathbb{M}^2\left(X_T\right)} \leq C_n^\ast(T^{1/2}+T^{1/4}) \vert u-\rv\vert_{\mathbb{M}^2\left(X_T\right)} .
\end{align}
\end{prp}
\begin{proof}[Proof of Proposition \ref{prop-Lipschitz of Psi}]
By the triangle inequality, \autoref{prp-S}, \autoref{lem-Phi_T,F^n-bounded} and \autoref{prp-phi g-1} we infer that there exists $C_n^\ast:=C_n^\ast(n)>0$ such that for all $u, \rv \in \mathbb{M}^2\left(X_T\right)$,
\begin{equation}\label{eqn-Phi-difference}
\begin{aligned}
&\vert\Psi_{T, u_0}^n\left(u\right)-\Psi_{T, u_0}^n\left(\rv\right)\vert_{\mathbb{M}^2\left(X_T\right)} \\
&\leq
[\mathbb{E}\vert S \ast [\Phi_{T, B}^n(u)-\Phi_{T, B}^n(\rv)]\vert_{X_T}^2]^{1/2}+ \vert S \diamond [\Phi_{T, g}^n(u)-\Phi_{T, g}^n(\rv)]\vert_{\mathbb{M}^2\left(X_T\right)}\\
&\leq C_7[\mathbb{E}\Vert \Phi_{T, B}^n(u)-\Phi_{T, B}^n(\rv) \Vert_{L^2(0,T;\rH)}^2]^{1/2}+C_8\vert \Phi_{T, g}^n(u)-\Phi_{T, g}^n(\rv) \vert_{\mathbb{M}^2(0, T ; \rV)}\\
&\leq C_n^\ast T^{1/4} [\mathbb{E}\vert u-\rv\vert_{X_T}^2]^{1/2} + C_n^\ast T^{1/2}[\mathbb{E}\vert u-\rv
\vert_{X_T}^2]^{1/2}\\
&= C_n^\ast(T^{1/2}+T^{1/4})\vert u-\rv\vert_{\mathbb{M}^2\left(X_T\right)}.
\end{aligned}
\end{equation}
The proof is complete.
\end{proof}

With the Lipschitz continuity of $\Psi_{T,u_0}^n$ established in \autoref{prop-Lipschitz of Psi}, we apply the Banach fixed-point theorem to obtain a unique local solution to \eqref{NS-approximated}.
\begin{cor}\label{cor-existence of u^n}
For every $n\in\mathbb{N}$, there exists $T_n>0$ such that the approximated problem \eqref{NS-approximated} has a local solution $(u^n, T_n)$.

Moreover, if $T^{(2)} \leq T^{(1)} \leq T_n$ and $u^{n,i} \in X_{T^{(i)}}$ satisfies
$\Psi_{T^{(i)}, u_0}^n(u^{n,i})=u^{n,i}$, $i=1,2$, then the restriction of $u^{n,1}$ to $[0,T^{(2)}]$ is equal to $u^{n,2}$.
\end{cor}
\begin{proof}
Let us take $T_n>0$ such that
\begin{align}\label{eqn-C_n^* T_n}
C_n^\ast(T_n^{1/2}+T_n^{1/4})\leq \frac{1}{2}.
\end{align}
Then in view of \autoref{prop-Lipschitz of Psi}, for every $T \in\left[0, T_n\right]$, the map $\Psi_{T, u_0}^n$ is $\frac{1}{2}-$strict contraction. By applying the Banach Fixed Point Theorem, see \cite[Theorem 9.23]{Rudin+Walter_1976}, the map $\Psi_{T, u_0}^n$ has a unique fixed point, i.e. there exists a unique $u^n \in X_T$, such that $\Psi_{T, u_0}^n(u^n)=u^n$. The uniqueness of the fixed point ensures that the solutions agree on common parts of the intervals. The proof is complete.
\end{proof}

\section{Existence of a local solution}\label{sec_Definition and existence of a local solution}
In this section, we proceed to show that, by introducing a suitable stopping time $\tau_n$, the approximated solution defined in \autoref{sec-A priori estimates} yields a local mild solution to the original problem \eqref{eqn-NS03}.

\begin{prop}\label{prop-local solution}
Assume that $\mathbb{E}\vert u_0\vert_\rV^2 <\infty$. Then the pair  $\left(u^n, \tau_n\right)$ is a local mild solution to the main problem \eqref{eqn-NS03}. The stopping time $\tau_n$ is defined by the following equality:
\begin{equation}\label{eq-tau_n}
\begin{aligned}
\tau_n:=\inf \left\{t \in[0, T_n]: \vert u^n\vert_{X_t} \geq n\right\} ,
\end{aligned}
\end{equation}
where $T_n>0$ is given by \autoref{cor-existence of u^n} and $u^n$ is the solution of the approximated problem \eqref{NS-approximated}. If $\vert u^n\vert_{X_t}(\omega) < n$ for all $t\in[0,T_n]$, we put $\tau_n(\omega)=T_n$.
\end{prop}
\begin{proof}
Let us choose and fix $n\in\mathbb{N}$. In view of \autoref{cor-existence of u^n}, for $t \in[0, T_n]$, $u^n$ satisfies the approximated equation
\begin{equation}\label{NS-approximated1}
\begin{aligned}
u^n(t)= & S(t) u_0-\int_0^t S(t-r) \theta_n\left(\vert u^n\vert_{X_r}\right)B\left(u^n(r)\right)\, d r \\
& +\int_0^t S(t-r)\theta_n\left(\vert u^n\vert_{X_r}\right)g\left(u^n(r)\right) \,d W(r), \quad \mathbb{P} \text {-a.s. }.
\end{aligned}
\end{equation}

Assume $M\geq 0$. Let $u_0$ be $\mathcal{F}_0$-measurable $\rV$-valued random variable satisfying $\mathbb{E}\vert u_0\vert_\rV^2 \leq M^2$. By \autoref{prp-S} we infer
\begin{align}\label{eqn-Su_0}
\vert S u_0\vert_{\mathbb{M}^2\left(X_T\right)}\leq  C_9 [\mathbb{E}\vert u_0\vert_\rV^2]^{1/2}\leq C_9 M.
\end{align}
Let $\rv=0$ in \autoref{prop-Lipschitz of Psi},  we infer that
\begin{equation}\label{eqn-Psi u-Su0}
\begin{aligned}
\vert\Psi_{T, u_0}^n\left(u\right)-Su_0\vert_{\mathbb{M}^2\left(X_T\right)} &\leq C_n^\ast(T^{1/2}+T^{1/4})\vert u\vert_{\mathbb{M}^2\left(X_T\right)},\;\;T\geq 0.
\end{aligned}
\end{equation}
Using the triangle inequality, by \eqref{eqn-Su_0} and \eqref{eqn-Psi u-Su0} we infer that for all $u \in \mathbb{M}^2\left(X_T\right)$ and $T\geq 0$,
\begin{equation}\label{eqn-psi u}
\begin{aligned}
\vert \Psi_{T, u_0}^n(u)\vert_{\mathbb{M}^2\left(X_T\right)} &\leq \vert S u_0\vert_{\mathbb{M}^2\left(X_T\right)}+\vert\Psi_{T, u_0}^n\left(u\right)-Su_0\vert_{\mathbb{M}^2\left(X_T\right)}\\
&\leq C_9 M+C_n^\ast(T^{1/2}+T^{1/4})\vert u\vert_{\mathbb{M}^2\left(X_T\right)}.
\end{aligned}
\end{equation}
In view of \autoref{cor-existence of u^n}, $C_n^\ast(T_n^{1/2}+T_n^{1/4})\leq \frac{1}{2}$ and $\Psi_{T, u_0}^n(u^n)=u^n$. Therefore, for every $T \in (0,T_n]$, from \eqref{eqn-psi u} we infer
\begin{align}
\vert u^n\vert_{\mathbb{M}^2\left(X_T\right)}=\vert\Psi_{T, u_0}^n\left(u^n\right)\vert_{\mathbb{M}^2\left(X_T\right)} \leq C_9 M+\frac{\vert u^n\vert_{\mathbb{M}^2\left(X_T\right)}}{2} ,
\end{align}
which implies
\begin{align}\label{eqn-u^n}
\vert u^n\vert_{\mathbb{M}^2\left(X_T\right)}\leq 2 C_9 M.
\end{align}
Since $t \wedge \tau_n \leq T_n$ for every $t\geq 0$, inequality \eqref{eqn-u^n} implies
\begin{equation}\label{eqn-u^n-t wedge tau_n}
\mathbb E
\left(
\sup_{s\in[0,t\wedge\tau_n]}|u^n(s)|_V^2
+
\int_0^{t\wedge\tau_n}|u^n(s)|_{D(A)}^2\,ds
\right)
=
|u^n|_{M^2(X_{t\wedge\tau_n})}^2
\leq 4C_9^2M^2,\quad t\geq 0.
\end{equation}

Note that processes involved in the both sides of (\ref{NS-approximated1}) are continuous, so this equality still holds when the deterministic time $t$ is replaced by $t \wedge \tau_n$,
\begin{align}\label{NS-approximated2}
u^n (t \wedge \tau_n)= & S(t\wedge \tau_n) u_0-\int_0^{t \wedge \tau_n} S (t \wedge \tau_n-r) \theta_n ( |u^n|_{X_r}) B (u^n(r))\, d r \\
& +\int_0^{t\wedge \tau_n} S(t\wedge \tau_n-r)\theta_n (\vert u^n\vert_{X_r})g (u^n(r)) \,d W(r), \quad \mathbb{P} \text {-a.s. },\;\;t\geq 0.\nonumber
\end{align}
Since $\vert u^n\vert_{X_{\tau_n}}\leq n $,
and the function $r \mapsto\vert u^n\vert_{X_r}$ is non-decreasing, we infer that for all $r \leq \tau_n$, $\vert u^n\vert_{X_r} \leq\vert u^n\vert_{X_{\tau_n}} \leq n$. By definition of $\theta_n$, we infer
\begin{align}\label{eq-theta}
\theta_n(\vert u^n\vert_{X_r})=1, \quad \text { for every } r \in\left[0, t \wedge \tau_n\right], t \geq 0.
\end{align}
Define
\begin{equation}
\begin{aligned}
I(t)&:=\int_0^{t} S(t-r)g\left(u^n(r)\right) \,d W(r),\\
I_{\tau_n}(t)&:=\int_0^{t} \mathbf{1}_{[0,\tau_n)}(r)S(t-r)g\left(u^n(r\wedge \tau_n)\right)\,d W(r).
\end{aligned}
\end{equation}
We claim for every $t \geq 0$,
\begin{equation}\label{eq-I}
\begin{aligned}
I(t\wedge \tau_n)=I_{\tau_n}(t\wedge \tau_n)\;\; \mathbb{P}\mbox{-a.s.}.
\end{aligned}
\end{equation}
Indeed, since $\tau_n$ is a stopping time, the indicator process $\mathbf{1}_{[0, \tau_n)}$ of the open interval $[0, \tau_n)$ is adapted and right-continuous, and therefore by \cite[Proposition 1.13]{Karatzas+Shreve_1991}, it is progressively measurable. And the stopped process $g\left(u^n(r\wedge \tau_n)\right)$, $r \geq 0$, is progressively measurable, see \cite[Proposition 2.18]{Karatzas+Shreve_1991}. Therefore, the process
\[
S(t-r)\left(\mathbf{1}_{[0, \tau_n)}(r) g\left(u^n(r\wedge \tau_n)\right)\right), \;\; r \in [ 0,t],\] is progressively measurable.
It was shown in \cite[Lemma A.1]{Brzezniak+Maslowski+Seidler_2005} that if we assume further that $I$ and $I_\tau$ have continuous paths almost surely, then (\ref{eq-I}) holds for every $t \geq 0$, $ \mathbb{P}$-a.s..
Then by (\ref{NS-approximated2}), (\ref{eq-theta}) and (\ref{eq-I}) we infer that for all $t \geq 0$
\begin{equation}\label{eq-u^n local mild solution}
\begin{aligned}
u^n (t \wedge \tau_n )= & S(t\wedge \tau_n) u_0-\int_0^{t \wedge \tau_n} S (t \wedge \tau_n-r )B (u^n(r) )\, d r \\
& +\int_0^{t\wedge \tau_n} \mathbf{1}_{[0,\tau_n)}(r)S(t\wedge \tau_n-r)g (u^n(r\wedge \tau_n) ) \,d W(r), \quad \mathbb{P} \text {-a.s. }.
\end{aligned}
\end{equation}
Therefore, $ (u^n, \tau_n )$ is a local mild solution to our main problem (\ref{eqn-NS03}).
In the following we aim to show that for every $\epsilon>0$ and $M>0$, there exists $T^\ast(\epsilon,M)>0$ and $n\in\mathbb{N}$ such that
\begin{align}\label{eqn-514}
\mathbb{P}(\tau_n\geq T^\ast)\geq 1-\epsilon.
\end{align}
Let us choose and fix $\varepsilon>0$ and $M>0$. Choose $n\in\mathbb N$ such that
\begin{equation}
n^2\geq \frac{4C_9^2M^2}{\varepsilon}.
\end{equation}
By the definition of $\tau_n$,
\begin{equation*}
\{\tau_n<T_n\}
\subset
\{|u^n|_{X_{T_n}}\geq n\}.
\end{equation*}
Therefore, by Chebyshev's inequality and inequality \eqref{eqn-u^n},
\begin{align*}
\mathbb P(\tau_n<T_n)
&\leq
\mathbb P\left(|u^n|_{X_{T_n}}\geq n\right) \leq
\frac{\mathbb E|u^n|_{X_{T_n}}^2}{n^2}
\leq
\frac{4C_9^2M^2}{n^2}
\leq \varepsilon.
\end{align*}
Taking $T^*(\varepsilon,M):=T_n$, we infer that \eqref{eqn-514} holds.
The proof is complete.
\end{proof}

\section{Local maximal solution} \label{subsec-local maximal}
This section is devoted to establishing the existence and uniqueness of the maximal local solution to problem \eqref{eqn-NS03}. In this section we will use Assumption \ref{ass-g-loc Lip}.

As a first step, we address the uniqueness of a local solution $(u,\tau)$.

\begin{theorem}\label{thm-local unique}
Local solution $(u,\tau)$ to problem \eqref{eqn-NS03} is unique. More precisely, let $(u,\tau)$ and $(\rv,\sigma)$ be two local mild solutions to problem \eqref{eqn-NS03} in the sense of \autoref{def-local solution}, with the same initial data $u_0$ such that $\mathbb{E}\vert u_0\vert_\rV^2 <\infty$. Then the following pathwise uniqueness property holds:
\[
u(t) = \rv(t),\;\; \mathbb{P}\text{-a.s.} \text{ on } \Omega_t(\tau \wedge \sigma).
\]
\end{theorem}
\begin{proof}
Let
\[\rho := \tau \wedge \sigma.\]
Since both $\tau$ and $\sigma$ are accessible, there exist approximating sequences $(\tau_n)_{n \in \mathbb{N}}$ and $(\sigma_n)_{n \in \mathbb{N}}$ for $\tau$ and $\sigma$ such that $\tau_n \nearrow \tau$ and $\sigma_n \nearrow \sigma$ almost surely. Define
\[\rho_n := \tau_n \wedge \sigma_n\]
so that $\rho_n \nearrow \rho$ and each $\rho_n$ is also an accessible stopping time.

Let us choose and fix $n \in \mathbb{N}$, by the definition of local mild solution \autoref{def-local solution} for $t\geq 0$
\begin{equation}\label{eq-local mild solution-u}
\begin{aligned}
& u\left(t \wedge \rho_n\right)=  S\left(t \wedge \rho_n\right) u_0-\int_0^{t \wedge \rho_n} S\left(t \wedge \rho_n-r\right) B(u(r)) \,d r \\
&+ \int_0^{t\wedge \rho_n} \mathbf{1}_{[0, \rho_n)}(r) S(t\wedge \rho_n-r) g\left(u\left(r \wedge \rho_n\right)\right)\, d W(r),  \quad
\mathbb{P} \text {-a.s. }
\end{aligned}
\end{equation}
and
\begin{equation}\label{eq-local mild solution-v}
\begin{aligned}
& \rv\left(t \wedge \rho_n\right)=  S\left(t \wedge \rho_n\right) u_0-\int_0^{t \wedge \rho_n} S\left(t \wedge \rho_n-r\right) B(\rv(r)) \,d r \\
&+ \int_0^{t\wedge \rho_n} \mathbf{1}_{[0, \rho_n)}(r) S(t\wedge \rho_n-r) g\left(\rv\left(r \wedge \rho_n\right)\right)\, d W(r),  \quad
\mathbb{P} \text {-a.s. }
\end{aligned}
\end{equation}
Let us choose and fix $k \in \mathbb{N}$. Define $\varrho_{n, k}=\tau_n \wedge \sigma_n \wedge \tilde{\tau}_k \wedge \tilde{\sigma}_k$, where
\[
\tilde{\tau}_k=\inf \left\{t \in[0, \infty):|u|_{X_t} \geq k\right\} \wedge \tau,
\]
and
\[
\tilde{\sigma}_k=\inf \left\{t \in[0, \infty):|\rv|_{X_t} \geq k\right\} \wedge \sigma.
\]
We observe that
\[\varrho_{n, k} \nearrow \sigma_n \wedge \boldsymbol{\tau}_n=\rho_n,\; \mathbb{P}-a.s., \;\;k \rightarrow \infty .\]
Define the difference
\[w(t) := u(t) - \rv(t),\;\;t\geq 0.\]
Since for every $t\geq 0$, $t\wedge \varrho_{n, k}\wedge \rho_n=t\wedge \varrho_{n, k}$. In view of \eqref{eq-local mild solution-u} and \eqref{eq-local mild solution-v}, $w$ satisfies:
\begin{align}\label{eq-local mild solution-w}
&w(t\wedge \varrho_{n, k}) =- \int_0^{t \wedge \varrho_{n, k}} S\left(t \wedge \varrho_{n, k}-r\right) [B(u(r)) - B(\rv(r))]\,dr \\
&+ \int_0^{t \wedge \varrho_{n, k}} \mathbf{1}_{[0, \varrho_{n, k})}(r) S(t\wedge \varrho_{n, k}-r) [g\left(u\left(r \wedge \varrho_{n, k}\right)\right) - g\left(\rv\left(r \wedge \varrho_{n, k}\right)\right)]\,dW(r),\;\;\mathbb{P} \text {-a.s. },\nonumber
\end{align}
which implies for $t\in[0,\varrho_{n, k})$
\[
w = -S \ast [B(u)-B(\rv)] + S \diamond \mathbf{1}_{[0, t)}[g(u) - g(\rv)].
\]
Therefore, by the triangle inequality, for every $t> 0$
\begin{align}\label{eqn-w}
\vert w\vert_{\mathbb{M}^2\left(X_{t\wedge \varrho_{n, k}}\right)}^2
&\leq \mathbb{E}\vert S \ast [B(u)-B(\rv)]\vert_{X_{t\wedge \varrho_{n, k}}}^2 +\vert S \diamond \mathbf{1}_{[0, t\wedge\varrho_{n, k})}[g(u) - g(\rv)]\vert_{\mathbb{M}^2\left(X_{t\wedge \varrho_{n, k}}\right)}^2.
\end{align}
Using \autoref{prp-S} and \autoref{lem-inequality}, we infer
\begin{align}\label{eqn-w-10}
&\mathbb{E}\vert S \ast [B(u)-B(\rv)]\vert_{X_{t\wedge \varrho_{n, k}}}^2\\&\leq C_7 \mathbb{E}\vert B(u)-B(\rv)\vert_{L^2(0,t\wedge\varrho_{n, k};\rH)}^2=C_7 \mathbb{E}\int_0^{t\wedge\varrho_{n, k}}\vert B(u(r))-B(\rv(r))\vert_\rH^2\,dr\nonumber\\
&\leq CC_7 \left[\mathbb{E}\int_0^{t\wedge \varrho_{n, k}}  \vert u(r) \vert _{D(\rA)}\vert u(r) \vert _\rV\vert w(r) \vert_\rV^2\,dr+  \mathbb{E}\int_0^{t\wedge \varrho_{n, k}}\vert w(r) \vert _{D(\rA)}\vert w(r) \vert_\rV\vert \rv(r) \vert_\rV^2\,dr\right].\nonumber
\end{align}
By the H\"older inequality and the Young inequality, we infer that for every $\epsilon>0$ there exists $C_\epsilon:=C(\epsilon)>0$ such that
\begin{equation}\label{eqn-w-11}
\begin{aligned}
&\mathbb{E}\int_0^{t\wedge \varrho_{n, k}}  \vert u(r) \vert _{D(\rA)}\vert u(r) \vert _\rV\vert w(r) \vert_\rV^2\,dr\\
&\leq  \mathbb{E}\left[ \sup_{r\in[0,t\wedge \varrho_{n, k}]} \vert u(r) \vert _\rV \left(\int_0^{t\wedge \varrho_{n, k}}  \vert u(r) \vert _{D(\rA)}^2\,dr\right)^{1/2}\left(\int_0^{t\wedge \varrho_{n, k}}  \vert w(r) \vert_\rV^4\,dr\right)^{1/2}\right]\\
&\leq \mathbb{E}\left[ \vert u\vert_{X_{t\wedge\varrho_{n, k}}}^2\left(\sup_{r\in[0,t\wedge \varrho_{n, k}]}\vert w(r) \vert_\rV\right) \left(\int_0^{t\wedge \varrho_{n, k}}  \vert w(r) \vert_\rV^2\,dr\right)^{1/2}\right]\\
&\leq \mathbb{E}\left[ \vert u\vert_{X_{t\wedge\varrho_{n, k}}}^2 \left(\epsilon\sup_{r\in[0,t\wedge \varrho_{n, k}]}\vert w(r) \vert_\rV^2+C_\epsilon\int_0^{t\wedge \varrho_{n, k}}  \vert w(r) \vert_\rV^2\,dr\right)\right].
\end{aligned}
\end{equation}

Similarly, by the H\"older inequality and the Young inequality, we infer that there exists $C_\epsilon^\prime:=C(\epsilon)>0$ such that
\begin{equation}\label{eqn-w-12}
\begin{aligned}
&  \mathbb{E}\int_0^{t\wedge \varrho_{n, k}}\vert w(r) \vert _{D(\rA)}\vert w(r) \vert_\rV\vert \rv(r) \vert_\rV^2\,dr\\
&\leq  \mathbb{E}\left[ \sup_{r\in[0,t\wedge \varrho_{n, k}]} \vert \rv(r) \vert _\rV ^2\left(\int_0^{t\wedge \varrho_{n, k}}  \vert w(r) \vert _{D(\rA)}^2\,dr\right)^{1/2}\left(\int_0^{t\wedge \varrho_{n, k}}  \vert w(r) \vert_\rV^2\,dr\right)^{1/2}\right]\\
&\leq  \mathbb{E}\left[ \vert \rv\vert_{X_{t\wedge\varrho_{n, k}}}^2[\epsilon\int_0^{t\wedge \varrho_{n, k}}  \vert w(r) \vert _{D(\rA)}^2\,dr+C_\epsilon^\prime\int_0^{t\wedge \varrho_{n, k}}  \vert w(r) \vert_\rV^2\,dr]\right].
\end{aligned}
\end{equation}
Since $\varrho_{n, k}\leq \tilde{\tau}_k \wedge \tilde{\sigma}_k$, for every $t>0$,
\[
\vert u\vert_{X_{t\wedge\varrho_{n, k}}}\leq k,\;\;\vert \rv\vert_{X_{t\wedge\varrho_{n, k}}}\leq k.
\]
Hence, by inequalities \eqref{eqn-w-10}, \eqref{eqn-w-11} and \eqref{eqn-w-12}, we infer
\begin{equation}\label{eqn-w-1}
\begin{aligned}
&\mathbb{E}\vert S \ast [B(u)-B(\rv)]\vert_{X_{t\wedge \varrho_{n, k}}}^2\\&\leq CC_7 k^2\mathbb{E}\left[ \epsilon\sup_{r\in[0,t\wedge \varrho_{n, k}]}\vert w(r) \vert_\rV^2+C_\epsilon\int_0^{t\wedge \varrho_{n, k}}  \vert w(r) \vert_\rV^2\,dr\right]\\
&+ CC_7 k^2\mathbb{E}\left[ \epsilon\int_0^{t\wedge \varrho_{n, k}}  \vert w(r) \vert _{D(\rA)}^2\,dr+C_\epsilon^\prime\int_0^{t\wedge \varrho_{n, k}}  \vert w(r) \vert_\rV^2\,dr\right]\\
&\leq CC_7 k^2\left[ \epsilon\vert w(r) \vert_{\mathbb{M}^2\left(X_t\wedge \varrho_{n, k}\right)}^2+(C_\epsilon+C_\epsilon^\prime)\mathbb{E}\int_0^{t\wedge \varrho_{n, k}}  \vert w(r) \vert_\rV^2\,dr\right].
\end{aligned}
\end{equation}

Using \autoref{prp-S}  and \autoref{ass-g-loc Lip}, we infer there exists $C_k=C(k)>0$ such that
\begin{equation}\label{eqn-w-2}
\begin{aligned}
\vert S \diamond \mathbf{1}_{[0, \varrho_{n, k})}[g(u) - g(\rv)]\vert_{\mathbb{M}^2\left(X_{t\wedge \varrho_{n, k}}\right)}^2&\leq C_8 \vert g(u) - g(\rv)\vert_{\mathbb{M}^2(0, t\wedge \varrho_{n, k} ; \rV)}^2\\
&=C_8\mathbb{E}\left(\int_0^{t\wedge \varrho_{n, k}}\vert g(u(r)) - g(\rv(r))\vert_\rV^2\,dr\right)\\
&\leq C_kC_8\mathbb{E}\left(\int_0^{t\wedge \varrho_{n, k}}\vert w(r)\vert_\rV^2\,dr\right).
\end{aligned}
\end{equation}
Therefore, in view of \eqref{eqn-w}, \eqref{eqn-w-1} and \eqref{eqn-w-2}, we infer there exists $C_{k,\epsilon}=C(k,\epsilon)>0$ such that
\begin{equation}\label{eqn-w-6}
\begin{aligned}
\vert w\vert_{\mathbb{M}^2\left(X_t\wedge \varrho_{n, k}\right)}^2
&\leq  C_{k,\epsilon}\int_0^{t\wedge \varrho_{n, k}} \mathbb{E} \vert w(t)\vert_\rV^2\,dt+ CC_7 k^2\epsilon \vert w\vert_{\mathbb{M}^2\left(X_t\wedge \varrho_{n, k}\right)}^2.
\end{aligned}
\end{equation}
Let us choose $\epsilon\in(0,\frac{1}{CC_7 k^2})$, which implies
\[
CC_7 k^2\epsilon\in(0,1).
\]
Therefore, \eqref{eqn-w-6} implies there exists $C_{k,\epsilon}^\prime=C(k,\epsilon)>0$ such that
\[
\mathbb{E} \vert w(t\wedge \varrho_{n, k})\vert_\rV^2\leq\vert w\vert_{\mathbb{M}^2\left(X_{t\wedge \varrho_{n, k}}\right)}^2\leq C_{k,\epsilon}^\prime\int_0^{t\wedge \varrho_{n, k}} \mathbb{E} \vert w(t)\vert_\rV^2\,dt.
\]
Applying Gronwall's lemma \cite{Gronwall_1919} we get
\[
\mathbb{E}  \vert w(t\wedge \varrho_{n, k})\vert_\rV^2 = 0,\;\;t>0.
\]
Letting $k \to \infty$ and then $n \to \infty$, we infer
\[
\mathbb{E}  \vert w(t\wedge \rho)\vert_\rV^2 = 0,\;\;t>0.
\]
which implies
\[
u(t) = \rv(t),\;\; \mathbb{P}\text{-a.s.} \text{ on } \Omega_t(\tau \wedge \sigma).
\]
We complete the proof of \autoref{thm-local unique}.
\end{proof}

Let us state the following important lemma (see \cite[Lemma 5.6]{Brz+Haus+Raza_2021}) and remark
(see \cite[Remark 5.7]{Brz+Haus+Raza_2021}).
\begin{lemma}\label{lemma-Amalgamation}
(The Amalgamation Lemma).
\begin{trivlist}
\item[(i)] Let  $\Lambda$ be a family of accessible stopping times taking values in $[0, \infty]$, such that the supremum of every finite subset of $\Lambda$ belongs to $\Lambda$. Then, a supremum of $\Lambda$, i.e. $\tau:=\sup \Lambda$, is an accessible stopping time with values in $[0, \infty]$ and there exists an increasing sequence $\left\{\alpha_n\right\}_{n=1}^{\infty}\subset \Lambda$ such that $\tau(\omega)=\lim _{n \rightarrow \infty} \alpha_n(\omega)$, for all $\omega \in \Omega$.
\item[(ii)] Assume also that for each $\alpha \in \Lambda$, $I_\alpha:[0, \alpha) \times \Omega \rightarrow \rV$ is an admissible process such that for all $\alpha, \beta \in \Lambda$ and every $t>0$,
\begin{align}\label{eq-3.1}
I_\alpha(t)=I_\beta(t) \quad \mathbb{P} \text {-a.s. on } \Omega_t(\alpha \wedge \beta).
\end{align}
Then, there exists an admissible process $\mathbf{I}:[0, \tau) \times \Omega \rightarrow \rV$, such that every $t>0$,
\begin{align}\label{eq-3.2}
\mathbf{I}(t)=I_\alpha(t) \quad \mathbb{P} \text {-a.s. on } \Omega_t(\alpha).
\end{align}
\item[(iii)]
Moreover, if $\tilde{\mathbf{I}}:[0, \tau) \times \Omega \rightarrow X$ is any process satisfying (\ref{eq-3.2}) then the process $\tilde{\mathbf{I}}$ is a version of the process $\mathbf{I}$, i.e. for all $t \in[0, \infty)$
\begin{align}
\mathbb{P}(\{\omega \in \Omega: t<\tau(\omega), \mathbf{I}(t, \omega) \neq \tilde{\mathbf{I}}(t, \omega)\})=0 .
\end{align}
In particular, if in addition $\tilde{\mathbf{I}}$ is an admissible process, then
\begin{align}
\mathbf{I}=\tilde{\mathbf{I}}.
\end{align}
\end{trivlist}
\end{lemma}
\begin{remark}\label{remark-I}
Let us note that because both processes $\mathbf{I}:[0, \tau) \times \Omega \rightarrow \rV$ and $I_\alpha$ : $[0, \alpha) \times \Omega \rightarrow \rV$ are admissible (and hence with almost sure continuous trajectories),
\begin{trivlist}
\item[(i)] since $\alpha \leq \tau$, condition (\ref{eq-3.2}) is equivalent to the following one:
\begin{align}
\mathbf{I}_{\mid[0, \alpha) \times \Omega}=I_\alpha.
\end{align}
\item[(ii)]
Similarly, condition (\ref{eq-3.1}) is equivalent to the following one:
\begin{align}
I_{\alpha \vert[0, \alpha \wedge \beta) \times \Omega}=I_{\beta\vert[0, \alpha \wedge \beta) \times \Omega}  .
\end{align}
\end{trivlist}
\end{remark}

Next, we introduce a partial order on the set of all local mild solutions to \eqref{eqn-NS03}, which will allow us to define and construct a maximal local solution.
\begin{definition}\label{def-partial ordered set}
Consider a family $\Xi$ of all local mild solutions $(u, \tau)$ to the problem (\ref{eqn-NS03}). For two elements $(u, \tau),(\rv, \sigma) \in \Xi$, let us define an order $\preceq$ as $(u, \tau) \preceq(\rv, \sigma)$ if and only if $\tau \leq \sigma ,\, \mathbb{P}$-a.s. and $\rv\vert_{[0, \tau) \times \Omega} \sim u$. Note that if $(u, \tau) \preceq(\rv, \sigma)$ and $(\rv, \sigma) \preceq(u, \tau)$, then $(u, \tau) \sim(\rv, \sigma)$. We will say $(u, \tau) \prec(\rv, \sigma)$ if and only if $(u, \tau) \preceq(\rv, \sigma)$ and $(u, \tau) \nsim(\rv, \sigma)$.
\end{definition}
\begin{remark}
It is not difficult to see that the order $\preceq$ on $\Xi$ is reflexive, anti-symmetric and transitive. Therefore, $(\Xi, \preceq)$ is a partial ordered set. Moreover, for any non-empty chain $\left(u_1, \tau_1\right) \preceq\left(u_2, \tau_2\right) \preceq\left(u_3, \tau_3\right) \preceq \cdots$, of elements from $(\Xi, \preceq)$, there exists an upper bound $(u, \tau)$ for the chain. Hence, there exists a maximal element $(u, \tau)$ in $(\Xi, \preceq)$. The existence of an upper bound of every non-empty chain of $(\Xi, \preceq)$ is justified by a version of slightly generalized Amalgamation Lemma, see \autoref{lemma-Amalgamation}. Each maximal element $(u, \tau)$ in the set $(\Xi, \preceq)$ is a maximal local solution to the problem \eqref{eqn-NS03}.
\end{remark}

We can now state the existence and uniqueness of a maximal local solution to our main problem \eqref{eqn-NS03}.
\begin{prop}\label{prop-local solution maximal}
Assume that $\mathbb{E}\vert u_0\vert_\rV^2 <\infty$. Problem \eqref{eqn-NS03} has a unique maximal local solution $(\hat{u}, \hat{\tau})$, i.e.
for all other local solution $(\rv, \sigma)$, $\sigma \leq \hat\tau$ and $\hat{u}\vert_{[0, \sigma)}=\rv$.
\end{prop}
\begin{proof}
\textbf{Step 1}: Let us choose and fix a local solution $\left(u^*, \tau^\ast\right)$ to problem (\ref{eqn-NS03}) and let us consider the family $\Xi_1$ of all local solutions $(u, \tau)$ to the problem (\ref{eqn-NS03}) such that $\left(u^*, \tau^\ast\right) \preceq(u, \tau)$. By \autoref{prop-local solution} this set is non-empty. It follows from \autoref{prop-local solution}, \autoref{thm-local unique} and the Amalgamation Lemma, see \autoref{lemma-Amalgamation}, there exists an accessible stopping time
\begin{align}\label{eq-hat tau}
\hat{\tau}:=\sup \{\tau:(u, \tau) \in \Xi_1\},
\end{align}
and an admissible process $\hat{u}:[0, \hat{\tau}) \times \Omega \rightarrow \rV$, such that for all $(u, \tau) \in \Xi_1$ and $t>0$
\begin{align}\label{eq-hat u}
\hat{u}(t)=u(t) \quad \mathbb{P} \text {-a.s. on } \Omega_t(\tau),
\end{align}
where $\Omega_t(\tau):=\{\omega \in \Omega: t<\tau(\omega)\}$. Moreover, there exists an increasing sequence $\left(\hat{\tau}_n\right)$ of accessible stopping times such that $\hat{\tau}(\omega)=\lim _{n \rightarrow \infty} \hat{\tau}_n(\omega)$, for all $\omega \in \Omega$.

\textbf{Step 2}: In order to complete the proof of the existence of a maximal local solution, we shall prove that $(\hat{u}, \hat{\tau}) \in \Xi_1$.

Let us define an auxiliary process $\hat{\eta}=(\hat{\eta}(t)), t \in[0, \hat{\tau})$, such that for each $n \in \mathbb{N}$ and $t \geq 0$, the following equality holds:
\begin{equation}\label{def-hat eta}
\begin{aligned}
\hat{\eta}\left(t \wedge \hat{\tau}_n\right)= & S(t\wedge \hat{\tau}_n) u_0-\int_0^{t \wedge \hat{\tau}_n} S\left(t \wedge \hat{\tau}_n-r\right)B\left(\hat{u}(r)\right)\, d r \\
& +\int_0^{t\wedge \hat{\tau}_n} \mathbf{1}_{[0,\hat{\tau}_n)}(r)S(t\wedge \hat{\tau}_n-r)g\left(\hat{u}(r\wedge \hat{\tau}_n)\right) \,d W(r), \quad \mathbb{P} \text {-a.s. }
\end{aligned}
\end{equation}
Define a process $\eta=(\eta(t)), t \in[0, \tau)$ by the above formulae (\ref{def-hat eta}) with $\hat{u}$ replaced by $u$ and the approximating sequence $\left({\hat\tau}_n\right)_{n\in \mathbb{N}}$ of the accessible stopping time $\hat{\tau}$ replaced by approximating sequence $\left(\tau_n\right)_{n\in \mathbb{N}}$, of the accessible stopping time $\tau$, i.e.
\begin{equation}\label{def-eta}
\begin{aligned}
\eta\left(t \wedge {\tau}_n\right)= & S(t\wedge {\tau}_n) u_0-\int_0^{t \wedge {\tau}_n} S\left(t \wedge {\tau}_n-r\right)B\left({u}(r)\right)\, d r \\
& +\int_0^{t\wedge {\tau}_n} \mathbf{1}_{[0,{\tau}_n)}(r)S(t\wedge {\tau}_n-r)g\left({u}(r\wedge {\tau}_n)\right) \,d W(r), \quad \mathbb{P} \text {-a.s. }
\end{aligned}
\end{equation}
Since the pair $(u, \tau)$ is a local solution satisfying \eqref{eq-local mild solution}, it follows that the process $\eta(t), t \in[0, \tau)$ is a version of the process $u(t), t \in[0, \tau)$. Therefore, we infer
\begin{align}\label{eqn-?}
&u\left(t \wedge \hat{\tau}_n\wedge\tau_n\right)\\
&=  S(t\wedge \hat{\tau}_n\wedge\tau_n) u_0-\int_0^{t \wedge \hat{\tau}_n\wedge\tau_n} S\left(t \wedge \hat{\tau}_n\wedge\tau_n-r\right)B\left(u(r)\right)\, d r \nonumber\\
& +\int_0^{t\wedge \hat{\tau}_n\wedge\tau_n} \mathbf{1}_{[0,\hat{\tau}_n\wedge\tau_n)}(r)S(t\wedge \hat{\tau}_n\wedge\tau_n-r)g\left(u(r\wedge \hat{\tau}_n\wedge\tau_n)\right) \,d W(r),\;\;\mathbb{P} \text {-a.s. }.\nonumber
\end{align}
Since $\hat{u}$ satisfies \eqref{eq-hat u}, in view of \cite[Proposition
2.10]{brzezniak+Elworthy_2000}, we infer that
\begin{equation}\label{eq-brzezniak+Elworthy_2000}
\begin{aligned}
&\int_0^{t\wedge \hat{\tau}_n\wedge\tau_n} \mathbf{1}_{[0,\hat{\tau}_n\wedge\tau_n)}(r)S(t\wedge \hat{\tau}_n\wedge\tau_n-r)g\left(u(r\wedge \hat{\tau}_n\wedge\tau_n)\right) \,d W(r)\\
&=\int_0^{t\wedge \hat{\tau}_n\wedge\tau_n} \mathbf{1}_{[0,\hat{\tau}_n\wedge\tau_n)}(r)S(t\wedge \hat{\tau}_n\wedge\tau_n-r)g\left(\hat{u}(r\wedge \hat{\tau}_n\wedge\tau_n)\right) \,d W(r).
\end{aligned}
\end{equation}
Hence, by \eqref{eq-hat u}, \eqref{def-hat eta} and \eqref{eq-brzezniak+Elworthy_2000}, we infer that
\begin{equation}\label{eq-u-hat eta}
\begin{aligned}
u\left(t \wedge \hat{\tau}_n\wedge\tau_n\right)
=\hat{\eta}\left(t \wedge \hat{\tau}_n\wedge\tau_n\right),\;\; \mathbb{P} \text {-a.s. }.
\end{aligned}
\end{equation}
Letting $n\to \infty$, we infer
\begin{align}\label{eq-hat eta}
\hat{\eta}(t)=u(t) \quad \mathbb{P} \text {-a.s. on } \Omega_t(\tau), \quad \text { for } t>0 .
\end{align}
Hence, by the part (iii) of the Amalgamation Lemma, see \autoref{lemma-Amalgamation}, and \eqref{eq-hat u}, we infer that the process $\hat{\eta}(t), t \in[0, \hat{\tau})$, is a version of the process $\hat{u}(t), t \in[0, \hat{\tau})$ and therefore we can replace $\hat{\eta}$ by $\hat{u}$ on the left-hand side of (\ref{def-hat eta}), i.e.
\begin{equation}\label{eqn-hat u}
\begin{aligned}
\hat{u}\left(t \wedge \hat{\tau}_n\right)= & S(t\wedge \hat{\tau}_n) u_0-\int_0^{t \wedge \hat{\tau}_n} S\left(t \wedge \hat{\tau}_n-r\right)B\left(\hat{u}(r)\right)\, d r \\
& +\int_0^{t\wedge \hat{\tau}_n} \mathbf{1}_{[0,\hat{\tau}_n)}(r)S(t\wedge \hat{\tau}_n-r)g\left(\hat{u}(r\wedge \hat{\tau}_n)\right) \,d W(r), \quad \mathbb{P} \text {-a.s. }
\end{aligned}
\end{equation}
Therefore, we deduce that the pair $(\hat{u}, \hat{\tau})$ belongs to the class $\Xi_1$, i.e. $(\hat{u}, \hat{\tau})$ is a maximal local solution of \eqref{eqn-NS03}.

\textbf{Step 3}: In the following, we aim to prove the uniqueness of the maximal local solutions.

To see this, let us suppose that $(u, \tau)$ and $(\rv, \sigma)$ are two maximal local solutions. Let us put $\tilde{\rho}=\tau \vee \sigma$. Then, $\tilde{\rho}$ is an accessible stopping time with approximating sequence $\tilde{\rho}_n:=\tau_n \vee \sigma_n$, for all $n \in \mathbb{N}$, where $\left(\tau_n\right)_{n \in \mathbb{N}}$ and $\left(\sigma_n\right)_{n \in \mathbb{N}}$ are approximating sequences of $\tau$ and $\sigma$, respectively. By \autoref{thm-local unique}, we infer that
\begin{align}\label{eq-3.13}
\left(u\vert_{[0, \tau \wedge \sigma)}, \tau \wedge \sigma\right) \sim\left(\rv\vert_{[0, \tau \wedge \sigma)}, \tau \wedge \sigma\right).
\end{align}

We shall now prove that $\tau=\sigma, \, \mathbb{P}$-a.s.. Suppose by contradiction that $\mathbb{P}(\{\tau \neq \sigma\})>0$. Let $\Omega_1:=\{\tau \geq \sigma\}$ and $\Omega_2:=\{\sigma>\tau\}$. We define a process $(\tilde{u}, \tilde{\rho})$ by the following formula:
\begin{align}\label{eq-3.14}
\tilde{u}(t, \omega)= \begin{cases}u(t, \omega) & \text { if } \omega \in \Omega_1 \quad \text { and } \quad t \in[0, \tau(\omega)) \\ \rv(t, \omega) & \text { if } \omega \in \Omega_2 \quad \text { and } \quad t \in[0, \sigma(\omega))
\end{cases}
\end{align}

We now claim that the process $(\tilde{u}, \tilde{\rho})$ is a local solution to problem (\ref{eqn-NS03}).

Firstly, we claim that $\tilde{u}(t), t \in[0, \tilde{\rho})$ is admissible.
We observe that since obviously the paths of the process $\tilde{u}$ are continuous, it is sufficient to show that for each $t \geq 0$, the function $\tilde{u}(t)=\tilde{u}(t, \cdot):\{t<\tilde{\rho} \} \rightarrow \rV$ is $\mathcal{F}_t$ measurable.
Note that
\[\tilde{u}(t)=u(t) \mathbf{1}_{\Omega_1}+\rv(t) \mathbf{1}_{\Omega_2},
\]
where $\mathbf{1}_{\Omega_i}$ denotes the indicator function of the set $\Omega_i$, for $i = 1,2$. Since both $\tau$ and $\sigma$ are stopping times, the sets $\Omega_1$ and $\Omega_2$ belong to the $\sigma$-algebra $\mathcal{F}_t$. Moreover, since $u(t)$ and $\rv(t)$ are $\mathcal{F}_t$-measurable, we infer the $\mathcal{F}_t$-measurability of $\tilde{u}(t)$.

Secondly, we aim to show $(\tilde{u}, \tilde{\rho})$ belongs to the class $\Xi$.
Let us choose and fix $n \in \mathbb{N}$ and $t \geq 0$. Let $\Omega^n_1:=\{\tau_n \geq \sigma_n\}$ and $\Omega^n_2:=\{\sigma_n>\tau_n\}$.
Let us observe that on $\Omega_1\cap \Omega^n_1$, we have
\[
\tilde{\rho}_n =\tau_n \vee \sigma_n=\tau_n<\tau.
\]
Therefore, since $(u, \tau)$ is a local mild solution, we deduce from \eqref{eq-3.14} and \eqref{eq-local mild solution} that for $t\geq 0$,
\begin{align}\label{eq-tilde u-1}
\tilde{u}\left(t \wedge {\tilde{\rho}}_n\right)=&u\left(t \wedge \tau_n\right)\\
= & S(t\wedge \tau_n) u_0-\int_0^{t \wedge \tau_n} S\left(t \wedge \tau_n-r\right)B\left(u(r)\right)\, d r \nonumber\\
& +\int_0^{t\wedge \tau_n} \mathbf{1}_{[0,\tau_n)}(r)S(t\wedge \tau_n-r)g\left(u(r\wedge \tau_n)\right) \,d W(r), \quad \mathbb{P} \text {-a.s. on }\Omega_1\cap \Omega^n_1.\nonumber
\end{align}
Hence, by \cite[Proposition
2.10]{brzezniak+Elworthy_2000} we deduce that
\begin{equation}\label{eq-tilde u-2}
\begin{aligned}
& \tilde{u}\left(t \wedge \tilde{\rho}_n\right)=  S\left(t \wedge \tilde{\rho}_n\right) u_0-\int_0^{t \wedge \tilde{\rho}_n} S\left(t \wedge \tilde{\rho}_n-r\right) B(\tilde{u}(r)) d r \\
&+ \int_0^{t\wedge \tilde{\rho}_n} \mathbf{1}_{[0, \tilde{\rho}_n)}(r) S(t\wedge \tilde{\rho}_n-r) g\left(\tilde{u}\left(r \wedge \tilde{\rho}_n\right)\right)\, d W(r),  \quad
\mathbb{P} \text {-a.s. on } \Omega_1\cap \Omega^n_1.
\end{aligned}
\end{equation}

On $\Omega_1\cap \Omega^n_2$, we have
\[
\tilde{\rho}_n=\sigma_n<\sigma=\tau\wedge \sigma.
\]
By \eqref{eq-3.13}, we infer
\begin{align}\label{eq-u=v}
\tilde{u}(t)=u(t)=\rv(t),\;\; \text{ on }\Omega_t(\tau\wedge\sigma).
\end{align}
Therefore, since $(\rv, \sigma)$ is a local mild solution, we deduce from \eqref{eq-u=v} and \eqref{eq-local mild solution} that for $t\geq 0$,
\begin{align}\label{eq-tilde u-4}\tilde{u}\left(t \wedge {\tilde{\rho}}_n\right)=&\tilde{u}\left(t \wedge \sigma_n\right)=u\left(t \wedge \sigma_n\right)=\rv\left(t \wedge \sigma_n\right) \\
= & S(t\wedge {\sigma}_n) u_0-\int_0^{t \wedge {\sigma}_n} S\left(t \wedge {\sigma}_n-r\right)B\left({\rv}(r)\right)\, d r \nonumber\\
& +\int_0^{t\wedge {\sigma}_n} \mathbf{1}_{[0,{\sigma}_n)}(r)S(t\wedge {\sigma}_n-r)g\left({\rv}(r\wedge {\sigma}_n)\right) \,d W(r), \quad \mathbb{P} \text {-a.s. on }\Omega_1\cap \Omega^n_1.\nonumber
\end{align}
Since $\sigma_n=\tilde{\rho}_n$, by using \eqref{eq-u=v} and \cite[Proposition
2.10]{brzezniak+Elworthy_2000} we deduce that
\begin{equation}\label{eq-tilde u-3}
\begin{aligned}
& \tilde{u}\left(t \wedge \tilde{\rho}_n\right)=  S\left(t \wedge \tilde{\rho}_n\right) u_0-\int_0^{t \wedge \tilde{\rho}_n} S\left(t \wedge \tilde{\rho}_n-r\right) B(\tilde{u}(r)) d r \\
&+ \int_0^{t\wedge \tilde{\rho}_n} \mathbf{1}_{[0, \tilde{\rho}_n)}(r) S(t\wedge \tilde{\rho}_n-r) g\left(\tilde{u}\left(r \wedge \tilde{\rho}_n\right)\right)\, d W(r),  \quad
\mathbb{P} \text {-a.s. on } \Omega_1\cap \Omega^n_2.
\end{aligned}
\end{equation}
Hence, $(\tilde{u}, \tilde{\rho})$ satisfies  (\ref{eq-local mild solution}) on $\Omega_1$. In a similar way, we can also show that $(\tilde{u}, \tilde{\rho})$ satisfies (\ref{eq-local mild solution}) on $\Omega_2$. Hence, the pair $(\tilde{u}, \tilde{\rho})$ belongs to the class $\Xi$.

Therefore, $(\tilde{u}, \tilde{\rho})$ is a local solution to problem (\ref{eqn-NS03}). Now, by construction we have $(u, \tau) \preceq(\tilde{u}, \tilde{\rho})$ and $(\rv, \sigma) \preceq(\tilde{u}, \tilde{\rho})$. Moreover, since $\mathbb{P}(\{\tau \neq \sigma\})>0$, we have $(u, \tau) \nsim(\tilde{u}, \tilde{\rho})$ or $(\rv, \sigma) \nsim(\tilde{u}, \tilde{\rho})$. This contradicts the maximality of $(u, \tau)$ or $(\rv, \sigma)$.
This completes the proof of \autoref{prop-local solution maximal}.
\end{proof}

\section{Preparations for the proof of Main Result}\label{sec-generalization}
In this section, we will give some propositions and lemmas that will be used in \autoref{sec-proof of the main result} to establish the global existence of the solution. From this section, we use Assumption \ref{assumption-phi}.

Assume that $\tau: \Omega \to (0,\infty]$ is a stopping time and
\[
u:[0,\tau) \times \Omega \to \rV
\]
with continuous trajectories which is a strong local maximal solution of our SNEs (\ref{eqn-NS03}) in the sense of \autoref{def-local solution maximal}.
The following result is important.
\begin{prop}\label{prop-maximal solution implies blowup}
Assume that $\tau: \Omega \to (0,\infty]$ is a stopping time and
\[
u:[0,\tau) \times \Omega \to \rV
\]
with continuous trajectories which is a strong local maximal solution of our SNEs (\ref{eqn-NS03}). In particular,
\begin{equation}\label{eqn-explosion}
\limsup_{t \toup \tau} \vert u\vert_{X_t}=\infty,\;\mathbb{P}-a.s. \mbox{ on the set } \{\tau<\infty\}.
    \end{equation}
\end{prop}

\begin{remark}\label{rem-GHZ} Our result is similar in spirit to assertion (2.35) in the paper
\cite{Glatt-Holtz+Ziane_2009}, which however was not proved in that paper.  Even though it were proven we would not be able to use it.
    \end{remark}

\begin{proof}[Proof of \autoref{prop-maximal solution implies blowup}]
Assume by contradiction that
\[
\widetilde\Omega:=\Big\{\tau<\infty\Big\}\cap
\Big\{\limsup_{t\toup\tau}\vert u\vert_{X_t}<\infty\Big\}
\]
has positive probability, i.e. $\mathbb{P}(\widetilde{\Omega})>0$. Note that
\begin{equation}\label{eqn-prob tilde Omega}
\begin{aligned}
0< \mathbb{P}\left( \widetilde\Omega \right)&=\mathbb{P}( \Big\{\tau<\infty\Big\}\cap \left\{\limsup_{t\toup\tau}\vert u\vert_{X_t}< \infty\right\})
\\
&=       \mathbb{P}(\Big\{\tau<\infty\Big\}\cap \bigcup_{R=1}^\infty \left\{\limsup_{t\toup\tau}\vert u\vert_{X_t}
     \leq R\right\})
     \\
&=     \lim_{R \to \infty} \mathbb{P}( \Big\{\tau<\infty\Big\}\cap\left\{ \limsup_{t\toup \tau}\vert u\vert_{X_t} \leq R \right\})
\\
&=     \lim_{R \to \infty} \mathbb{P}( \Big\{\tau<\infty\Big\}\cap \left\{
\sup_{t<\tau}\vert u(t)\vert_\rV^2+
\int_0^\tau \vert u(s)\vert_{D(\rA)}^2\,ds\leq R^2 \right\}).
\end{aligned}
\end{equation}

Let us recall the following notation, see \eqref{norm X_T},
\[
\vert u\vert_{X_t}^2:=\sup_{s\in[0,t]}\vert u(s)\vert_\rV^2+\int_0^t\vert u(s)\vert_{D(\rA)}^2\,ds.
\]
Define an event
\begin{align}\label{eq-OmegaR}
\widetilde{\Omega}_R:= \left\{ \omega \in  \widetilde\Omega : \sup_{t<\tau}\vert u(t)\vert_\rV^2+
\int_0^\tau \vert u(s)\vert_{D(\rA)}^2\,ds\leq R^2\right\}.
\end{align}
In view of \eqref{eqn-prob tilde Omega}, we infer
\begin{align}\label{eqn-omega_R-lim}
\mathbb{P}\left( \widetilde\Omega \right)= \lim_{R \to \infty} \mathbb{P}(
\widetilde{\Omega}_R ).
\end{align}

Let us choose and fix $R>0$. Let us define two auxiliary processes $\xi$ and $f$ on   $ \widetilde{\Omega}_R$,  by
\begin{align}\label{eq-xi f}
\xi(t) :=
\begin{cases}
g(u(t)), & t<\tau, \\
0, & t \geq \tau,
\end{cases}
\;\;\;\;
f(t) :=
\begin{cases}
B(u(t)), & t<\tau, \\
0, & t \geq \tau.
\end{cases}
\end{align}
 In view of the defiontion \eqref{eq-xi f}, Lemma \ref{lem-g-boundedness on balls} and the definition of $\Omega_R$ \eqref{eq-OmegaR} we infer that for every $T>0$,
\begin{equation}\label{eqn-E xi}
\begin{aligned}
\mathbb{E}   \int_0^T \vert \xi(s) \vert_{\rV}^2 \, ds&=\mathbb{E}   \int_0^{T\wedge\tau} \vert g(u(s)) \vert_{\rV}^2 \, ds\leq C_6\mathbb{E}   \int_0^{\tau} \vert u \vert_{\rV}^2 \, ds\\
&\leq C_6\mathbb{E}   \sup_{t<\tau}\vert u \vert_{\rV}^2 \leq C_6 R^2.
\end{aligned}
\end{equation}
Moreover, by the defiontion \eqref{eq-xi f}, Lemma \ref{lem-inequality} and the definition \eqref{eq-OmegaR} we infer that for every $T>0$,
\begin{equation}\label{eqn-E f}
\begin{aligned}
\mathbb{E}   \int_0^T \vert f(s) \vert_{\rH}^2 \, ds&=\mathbb{E}   \int_0^\tau \vert B(u(s)) \vert_{\rH}^2 \, ds\leq C\mathbb{E}   \int_0^{\tau} \vert u \vert_{\rV}^3\vert u\vert_{D(\rA)} \, ds\\
&\leq C\mathbb{E}   \sup_{t<\tau}\vert u \vert_{\rV}^3 \int_0^{\tau}\vert u\vert_{D(\rA)} \, ds\leq C R^4.
\end{aligned}
\end{equation}
Hence, inequality \eqref{eqn-E xi} and \eqref{eqn-E f} implies
\begin{equation}
    \mathbb{E} \left[   \int_0^T \vert \xi(s) \vert_{\rV}^2 \, ds+ \int_0^T \vert f(s) \vert_{\rH}^2 \, ds  \right] < \infty ,\;\;T>0
\end{equation}
and therefore
the following linear stochastic equation
\begin{equation}\label{eq-linear}
dv(t) + \rA v(t)\,dt + f(t)\,dt = \xi(t)\, dW(t),
\;\; v(0)=u_0.
\end{equation}
 has a unique strong solution $v$ which satisfies that for every $T>0$,
\[
v \in C([0,T];\rV)\cap L^2(0,T;D(\rA)), \;\; \mathbb{P}\text{-a.s.}
\]
see e.g. \cite[Proposition 6.11]{DaPrato+Zabczyk_2014}. Arguing as in the proof of Theorem 1.1 in \cite{Brzezniak+Maslowski+Seidler_2005}, we infer that

\begin{equation}\label{eq-vu equal}
v(t)=u(t), \;\; t<\tau ,\;\mathbb{P}-a.s.\mbox{ on }\widetilde{\Omega}_R
\end{equation}
where $u $ is the solution to   our main problem \eqref{eqn-NS03}, i.e.
\begin{align*}
&du(t)+\big[\rA u(t)+B(u(t))\big]\,dt=g(u(t))  \,dW(t), \;\; t \in[0, \infty),\\
&u(0)=u_0.
\end{align*}
Indeed,
\begin{equation}\label{eq-local mild solution-3}
\begin{aligned}
v\left(t \wedge \tau\right)=&   S\left(t \wedge \tau\right) u_0-\int_0^{t \wedge \tau} S\left(t \wedge \tau-r\right) f(r) \,d r \\
&+ \int_0^t \mathbf{1}_{\left[0, \tau\right)}(r) S(t\wedge \tau-r) \xi(r)\, d W(r)\\
=&   S\left(t \wedge \tau\right) u_0-\int_0^{t \wedge \tau} S\left(t \wedge \tau-r\right) B(u(r)) \,d r \\
&+ \int_0^{t\wedge \tau} \mathbf{1}_{\left[0, \tau\right)}(r) S(t\wedge \tau-r) g\left(u\left(r \wedge \tau\right)\right)\, d W(r)\\
=&u\left(t \wedge \tau\right),  \quad
\mathbb{P} \text {-a.s. }.
\end{aligned}
\end{equation}
Therefore, \eqref{eq-vu equal} follows. Thus,
\[
\lim_{t \toup \tau} u(t) = \lim_{t \toup \tau} v(t).
\]

We put
\begin{align}\label{eq-z_0}
z_0 := \lim_{t\toup \tau} u(t).
\end{align}
Since $v \in C([0,T];\rV)$, we infer that
\[
z_0\in\rV.
\]

Next we consider NSE with random noise:
\[
dz + \rA z\, dt + Bz(t)\,dt = g(z(t))\, dW(t), \;\; t>\tau,
\]
with initial data
\[
z(\tau) = z_0.
\]
By the local existence theorem \autoref{prop-local solution}, there exists $\tau^\ast>\tau$ and an unique local solution $(z,\tau^\ast)$ satisfies
\begin{align}\label{eq-solution z}
z \in C([\tau,\tau^\ast);\rV)\cap L^2(\tau,\tau^\ast;D(\rA)), \;\; \mathbb{P}\text{-a.s.}.
\end{align}

Put
\[
U(s) =
\begin{cases}
u(s), & s \in [0,\tau),\\
z(s), & s \in [\tau,\tau^\ast).
\end{cases}
\]
Therefore, in view of \eqref{eq-z_0} and \eqref{eq-solution z}
\begin{equation}
\begin{aligned}
&U:[0,\tau^\ast)\to \rV \text{ is continuous}, \\
&\int_0^{\tau^\ast} \vert U(s)\vert_{D(\rA)}^2 \,ds=\int_0^\tau \vert u(s)\vert_{D(\rA)}^2 \,ds+\int_\tau^{\tau^\ast} \vert z(s)\vert_{D(\rA)}^2 \,ds < \infty,
\;\; \mathbb{P}\text{-a.s.}.
\end{aligned}
\end{equation}

Thus $U$ is a solution of stochastic NSE \eqref{eqn-NS03} on $[0,\tau^\ast)$.
This contradicts the maximality of $u$ on $[0,\tau)$. Therefore, $\mathbb{P}(\widetilde\Omega_R)=0$ for every $R>0$. In view of \eqref{eqn-omega_R-lim},
\[
\mathbb{P}(\widetilde\Omega)= \lim_{R \to \infty} \mathbb{P}(
\widetilde{\Omega}_R )=0.\]
The proof of \autoref{prop-maximal solution implies blowup} is complete.
\end{proof}

\begin{prop}\label{prop-tau_m}
Define a sequence of stopping times $(\bar{\tau}_m)_{m\in \mathbb{N}}$ by
\begin{align}\label{eqn-tau_m}
\bar{\tau}_m:=\inf \left\{t \in[0, \tau):\vert u\vert_{X_t} \geq m\right\},\;\;m\in \mathbb{N}.
\end{align}
Then the following properties hold:
\begin{trivlist}
\item[(i)] $\bar{\tau}_m\toup\tau,\;\;\mathbb{P}-$ a.s.;
\item[(ii)] For every $T>0$,
\begin{equation}\label{eqn-tau-tau_m}
\mathbb{P}(\tau \leq  T) = \lim_{m \to \infty} \mathbb{P}\left(\bar{\tau}_m\leq T\right).
\end{equation}
\end{trivlist}
\end{prop}
\begin{proof}[Proof of part (i)]
Note that $\bar{\tau}_m\leq\tau$ and $m\mapsto\bar{\tau}_m$ is non-decreasing, hence
$\lim_{m\to \infty}\bar{\tau}_m$ exists and
\begin{align}\label{eq-lim bat tau leq tau}
\lim_{m\to \infty}\bar{\tau}_m\leq\tau.
\end{align}
In view of \autoref{def-local solution}, there exists an approximating sequence $\left(\tau_m\right)_{m \in \mathbb{N}}$ of $\tau$,
such that $\tau_n\toup\tau,\;\mathbb{P}-$a.s. and  for all $m \in \mathbb{N}$ and $t \geq 0$, we have
\[
\vert u\vert_{X_{\tau_m}}<\infty,\;\;\mathbb{P}-a.s..
\]
Define the sets
\begin{equation}
\begin{aligned}
&A:=\{\tau_n \toup \tau\},\\
&B_n:=\{\vert u\vert_{X_{\tau_n}}<\infty\},\ n\in\mathbb{N},\\
&\Omega_0:=\Big\{\tau_n\toup\tau\Big\}\ \cap\ \bigcap_{n\geq1}\Big\{\vert u\vert_{X_{\tau_n}}<\infty\Big\}=A\cap \bigcap_{n\geq1} B_n.
\end{aligned}
\end{equation}
Note that $\mathbb{P}(A)=1$ and $P(B_n)=1$ for every $n\in\mathbb{N}$, we infer
\[
\mathbb{P}(\Omega_0^c)
=\mathbb{P}\!\left( A^c \ \cup\  \bigcup_{n\ge1} B_n^c \right)
\leq \mathbb{P}(A^c) + \sum_{n\ge1}\mathbb{P}(B_n^c)
=0.
\]
which implies $\mathbb{P}(\Omega_0)=1$.

Let us choose and fix $\omega\in\Omega_0$.

Firstly, assume that $\tau(\omega)<\infty$. Let us choose and fix  $\eps\in(0,\tau(\omega))$.  Since $\tau_n\toup\tau$ on $\Omega_0$, there exists $k\in \mathbb{N}$ such that
\[
\vert u\vert_{X_{\tau(\omega)-\eps}}\leq \vert u\vert_{X_{\tau_{k}(\omega)}}<\infty.
\]
Therefore, there exists $m:=m(\eps)\in\mathbb{N}$ such that
\[
\vert u\vert_{X_{\tau(\omega)-\eps}}<m.
\]
Since $ t\mapsto \vert u\vert_{X_t}$ is non-decreasing,
$\vert u\vert_{X_t}<m$ for all $t\leq\tau(\omega)-\eps$. Hence, by the definition of $\bar{\tau}_m$ \eqref{eqn-tau_m} we infer
\[
\bar{\tau}_m(\omega)>\tau(\omega)-\eps.
\]
Note that $m\mapsto \bar{\tau}_m$ is non-decreasing,
\[
\lim_{m\to\infty}\bar{\tau}_m(\omega)>\tau(\omega)-\eps.
\]
Since $\eps$ is arbitrary, we infer that
\[
\lim_{m\to\infty}\bar{\tau}_m(\omega)\geq\tau(\omega).
\]
Together with \eqref{eq-lim bat tau leq tau} we infer that $\lim_{m\to\infty}\bar{\tau}_m=\tau$ on $\Omega_0$.
Thus, $\bar{\tau}_m\toup\tau,\mathbb{P}-$ a.s.

Now we consider $\tau(\omega)=\infty$.
Recall that for $\omega\in\Omega_0$,
\[
\tau_n(\omega) \toup \tau(\omega)=\infty \mbox{ and } \vert u\vert_{X_{\tau_n(\omega)}}<\infty \mbox{ for every } n\in\mathbb{N}.
\]
Let us choose and fix $L>0$. Therefore, there exists $n\in\mathbb{N}$ such that $\tau_n(\omega)>L$, we infer $\vert u\vert_{X_L}\leq\vert u\vert_{X_{\tau_n(\omega)}}<\infty$. Thus, there exists $m_L>0$ such that $\vert u\vert_{X_L}<m_L$. Since $ t\mapsto \vert u\vert_{X_t}$ is non-decreasing, we infer that for all $m\geq m_L$ and $t\leq L$, $\vert u\vert_{X_t}<m$, thus $\bar{\tau}_m(\omega)>L$. Since $L>0$ is arbitrary, it follows that $\bar{\tau}_m(\omega)\toup\infty=\tau(\omega)$ on $\Omega_0$. Therefore, $\bar{\tau}_m\toup\tau,\mathbb{P}-$ a.s..

The proof of of part (i) is thus complete.
\end{proof}
\begin{proof}[Proof of part (ii)]
Let us choose and fix $T>0$ and $\omega\in \Omega_0$.

Since $\bar{\tau}_m(\omega)\toup\tau(\omega)$, assume that $\tau\leq T$, we infer that $\bar{\tau}_m\leq \tau\leq T$ for all $m$, which implies $\{\tau\leq T\} \cap \Omega_0 \subseteq\bigcap_{m=1}^\infty \{\bar{\tau}_m\leq T\} \cap \Omega_0$.

Conversely, if $\bar{\tau}_m(\omega)\leq T$ for all $m$, taking limits yields $\tau(\omega)=\lim_m \bar{\tau}_m(\omega)\leq T$, which implies $\bigcap_{m=1}^\infty \{\bar{\tau}_m\leq T\}\cap \Omega_0 \subseteq\{\tau\leq T\}\cap \Omega_0 $. Hence, we infer
\[
\{\tau\leq T\}\cap \Omega_0=\bigcap_{m=1}^\infty \{\bar{\tau}_m\leq T\}\cap \Omega_0.
\]
Therefore, recall that $\mathbb{P}(\Omega_0)=1$ and $m\mapsto\bar{\tau}_m$ is non-decreasing, we infer that $\mathbf{1}_{\{\bar{\tau}_m\leq T\}}\downarrow \mathbf{1}_{\{\tau\leq T\}},\mathbb{P}-$a.s.. Taking expectations and applying the monotone convergence theorem, we infer that
\begin{equation}\label{eqn-tau-tau_m-2}
\mathbb{P}(\tau \leq  T) = \lim_{m \to \infty} \mathbb{P}\left(\bar{\tau}_m\leq T\right).
\end{equation}
The proof of \autoref{prop-tau_m} is complete.
\end{proof}

\begin{remark}
In order to establish the global existence of the solution, our goal is to prove that
\begin{align}\label{eqn-tau=infity}
\tau=\infty,\;\;\mathbb{P}-a.s..
\end{align}
It is sufficient to show that
\[
\lim_{m \to \infty} \mathbb{P}\left(\bar{\tau}_m\leq  T\right)=0,\;\;T>0.
\]
Indeed,
\[
\mathbb{P}(\tau<\infty)=\mathbb{P}\Big(\bigcup_{q\in\mathbb{Q}_+}\{\tau\leq q\}\Big)\leq \sum_{q\in\mathbb{Q}_+}\mathbb{P}(\tau\leq q)=\sum_{q\in\mathbb{Q}_+}\lim_{m \to \infty} \mathbb{P}(\bar{\tau}_m\leq q)=0,
\]
which implies $\mathbb{P}(\tau=\infty)=1$.
\end{remark}

Let us denote
\begin{align}\label{eq-C_0}
C_0 :=\frac{27}{16} C^4,
\end{align}
where $C$ is the constant from the Agmon inequality \eqref{eqn-the Agmon inequality}. We need the following proposition.
\begin{prp}\label{prop-Assumption1-1}
The following inequality holds.
\begin{align}\label{eqn-Generalized dissipativit-1}
-2\langle \rA u+B(u),  \rA u\rangle_{L^2} \leq -\vert \rA u\vert_{\rH}^2+C_0\vert \rA^\frac{1}{2}u\vert_\rH^6,\;\;u\in D(\rA).
\end{align}
\end{prp}

Let us recall the following notation
\begin{align}
     \Vert u\Vert_{L^{\infty}}^2&:= \esssup \vert u(x) \vert^2
  =\esssup \sum_{i=1}^3\vert u^i(x) \vert^2 \\
  \Vert \nabla u\Vert_{L^2}^2 &:=
  \int_{\dom} \sum_{i,j=1}^3
    \vert   D_j u^i(x)\vert^2\, dx.
\end{align}
It will be used below.

\begin{proof}[Proof of \autoref{prop-Assumption1-1}]
By the Cauchy–Schwarz inequality in $\mathbb{R}^3$ we have
\begin{equation}\label{eqn-(u cdot nabla)u-1}
\begin{aligned}
\Vert (u \cdot \nabla) u\Vert_{L^2}^2&= \int_{\dom}
\vert (u(x) \cdot \nabla) u(x)\vert^2\, dx= \int_{\dom} \sum_{i=1}^3 \vert \sum_{j=1}^3 u^j(x) D_j u^i(x) \vert^2\, dx \\
&\leq \int_{\dom} \sum_{i=1}^3 \left( \sum_{j=1}^3 |u^j(x)|^2 \right) \left( \sum_{j=1}^3 \vert D_j u^i(x)\vert^2 \right) \,dx \\
&= \int_{\dom} \vert u(x)\vert^2 \sum_{i,j=1}^3  \vert D_j u^i(x)\vert^2  \,dx \\
&\leq \Vert u\Vert_{L^\infty}^2 \int_{\dom} \sum_{i,j=1}^3  \vert D_j u^i(x)\vert^2  \,dx
= \Vert u\Vert_{L^\infty}^2 \Vert\nabla u\Vert_{L^2}^2.
\end{aligned}
\end{equation}

Hence, by the Cauchy-Scharz and Young inequalities,  the Agmon inequality \autoref{lem-Agmon} we get
\begin{equation}\label{eqn-BA-L2}
\begin{aligned}
-\langle B(u), \rA u\rangle_{L^2}
&\leq \Vert \rA u\Vert _{L^2} \Vert (u \cdot \nabla) u\Vert_{L^2} \leq
 \Vert \rA u\Vert _{L^2} \Vert u\Vert_{L^{\infty}}\Vert \nabla u\Vert_{L^2} \\
&
 \leq C  \vert \rA u\vert_{\rH}^{3/2} \Vert \nabla u\Vert_{L^2}^{3/2}
\leq  \frac{1}{2}\vert \rA u\vert_{\rH}^2+\frac{27}{32}C^4\Vert \nabla u\Vert_{L^2}^6 \\
&=  \frac{1}{2}\vert \rA u\vert_{\rH}^2+\frac{1}{2}C_0\vert \rA^\frac{1}{2}u\vert_\rH^6,
\end{aligned}
\end{equation}
where  $C_0$ is the constant defined in equality \eqref{eq-C_0}.

Therefore, we get
\begin{equation}\label{eqn-8.20}
\begin{aligned}
& -\langle \rA u+B(u), \rA u\rangle_{L^2}=-\langle \rA u,\rA u\rangle_{L^2}-\langle B(u),\rA u\rangle_{L^2}\\
\leq &-\vert \rA u\vert_{\rH}^2+ \frac{1}{2}\vert \rA u\vert_{\rH}^2+\frac{1}{2}C_0\vert \rA^\frac{1}{2}u\vert_\rH^6
= -\frac{1}{2}\vert \rA u\vert_{\rH}^2+\frac{1}{2}C_0\vert \rA^\frac{1}{2}u\vert_\rH^6,\,\, u\in D(\rA),
\end{aligned}
\end{equation}
which implies \eqref{eqn-Generalized dissipativit-1} holds.
\end{proof}

The next lemma will be essential in the proof of the global solution.
\begin{lemma}\label{lem-prob}
Assume that $(u,\tau)$ is the local maximal solution to problem \eqref{eqn-NS03} and $\bar{\tau}_m$ is defined in \eqref{eqn-tau_m}. Let $p\in(0,1)$. If $\kappa=2$, we further assume that  $(1-p)b^2>C_0$, where $C_0$ is the constant defined in \eqref{eq-C_0}. Then for every $R\geq 2^\frac{2}{2-p}$ and $T\geq 0$, there exist constants $\hat{C}^1_T:=\hat{C}^1_T(p,R_0,T)>0$ and $\hat{C}^2_T:=\hat{C}^2_T(p,R_0,T)>0$, independent of $m$ and $R$, such that the following estimates hold:
\begin{trivlist}
\item[(i)]
\begin{equation}
\begin{aligned}\label{eqn-6.23-1}
& \mathbb{P}\left(\sup _{t \in[0, T\wedge \bar{\tau}_m]}\vert \rA^\frac{1}{2}u(t)\vert_\rH^2 \geq R^2 \right) \leq\frac{\hat{C}^1_T }{R^p} .
\end{aligned}
\end{equation}
\item[(ii)]
\begin{equation}\label{eqn-Au-L2-prob-1}
 \begin{aligned}
\mathbb{P}\left(\int_0^{T\wedge \bar{\tau}_m}\vert \rA u(t)\vert_\rH^2 \, d t \geq R^2\right) \leq  \frac{\hat{C}^2_T }{R^p} .
\end{aligned}
\end{equation}
\end{trivlist}
\end{lemma}

\begin{proof}
Let $p\in(0,1)$. If $\kappa=2$, we further assume that
\begin{align*}
    (1-p)b^2>C_0,
\end{align*}
where $C_0$ is the constant defined in \eqref{eq-C_0}. Let us choose and fix $0<R_0<2^\frac{p}{2-p}$. Define the Lyapunov function by
\[
F(u):=l\left(\vert\rA^{\frac12}u\vert_{\rH}\right),
\;\; u\in D(\rA^{\frac12}),
\]
where $l\in C^2([0,\infty);[0,\infty))$ is non-decreasing such that
\begin{align}
l(\rho)=
\begin{cases}
\rho^2,&\;\; 0\leq \rho<R_0,\\
\rho^p,&\;\; \rho>2R_0.
\end{cases}
\end{align}

In view of \eqref{eqn-NS03} and \eqref{eqn-g}, for $t\leq\bar{\tau}_m$
\[
d\rA^{\frac12}u(t)
+
\left[
\rA^{\frac32}u(t)+\rA^{\frac12}B(u(t))
\right]dt
=
\phi(u(t))\rA^{\frac12}u(t)\,dW(t),
\;\; \mbox{in}\;\; \rV^\prime.
\]
Here $\rA^{\frac32}u$ and $\rA^{\frac12}B(u)$ are understood as
elements of $\rV'$ in the sense that, for every $v\in\rV$,
\begin{equation}
{}_{\rV'}\langle \rA^{\frac32}u,v\rangle_{\rV}
:=
\langle \rA u,\rA^{\frac12}v\rangle_{\rH},
\end{equation}
and
\begin{equation}
{}_{\rV'}\langle \rA^{\frac12}B(u),v\rangle_{\rV}
:=
\langle B(u),\rA^{\frac12}v\rangle_{\rH}.
\end{equation}

Then by the Itô formula for the function  $\vert \cdot\vert_\rH^2$, see \cite[Lemma 1.4]{Pardoux_1979},  for every $t>0$,
\begin{align}\label{eqn-u H1 norm-1}
&\left\vert \rA^\frac12u(t\wedge \bar{\tau}_m)\right\vert_\rH^2=\left\vert \rA^\frac12u(0)\right\vert_\rH^2\\
&+\int_0^{t\wedge \bar{\tau}_m }\left(2{}_{\rV^\prime}\left\langle
-[\rA^{\frac{3}{2}}u(s)+\rA^{\frac{1}{2}}B(u(s))], \rA^{\frac{1}{2}}u(s)\right\rangle_\rV+\vert\phi(u(s))\vert^2\vert\rA^{\frac{1}{2}}u(s)\vert_\rH ^2\right)\, d s \nonumber\\
& +2 \int_0^{t\wedge \bar{\tau}_m }\left\langle \rA^{\frac{1}{2}}u(s), \phi(u(s))\rA^{\frac{1}{2}}u(s) \right\rangle\,d W(s).\nonumber
\end{align}

Let
\[
Y(t):=|\rA^{\frac12}u(t)|_{\rH}^2,\;\;\rho(t):=\vert\rA^{\frac12}u(t)\vert_{\rH}.
\]
It follows from \eqref{eqn-u H1 norm-1} that, for
$t\leq\bar{\tau}_m$,
\begin{align*}
dY(t)
&=
\left[
-2\left\langle
\rA u(t)+B(u(t)),\rA u(t)
\right\rangle_{\rH}
+
|\phi(u(t))|^2Y(t)
\right]dt
+
2\phi(u(t))Y(t)\,dW(t).
\label{eqn-dY}
\end{align*}

Define
\begin{equation}
    h(y):=l(\sqrt{y}),
    \qquad y\geq0.
\end{equation}
Since $l(\rho)=\rho^2$ for $0\leq\rho<R_0$, we have
$h(y)=y$ for $0\leq y<R_0^2$. Moreover, $h\in C^2([0,\infty))$.
For $y>0$, writing $\rho=\sqrt{y}$, we have
\begin{equation*}
    h^{\prime}(y)=\frac{l^{\prime}(\rho)}{2\rho},
\end{equation*}
and
\begin{equation*}
    h^{\prime\prime}(y)
    =
    \frac{l^{\prime\prime}(\rho)}{4\rho^2}
    -
    \frac{l^{\prime}(\rho)}{4\rho^3}.
\end{equation*}
Applying It\^o's formula to $F(u(t))=h(Y(t))$,
for $t\in[0,\bar{\tau}_m]$,
\begin{equation}
\begin{aligned}
dF(u(t))
=&
-\frac{l^{\prime}(\rho(t))}{\rho(t)}
\left\langle \rA u(t)+B(u(t)),\rA u(t)\right\rangle dt  \\
&+
\frac12 l^{\prime\prime}(\rho(t))|\phi(u(t))|^2\rho^2(t)\,dt
+l^{\prime}(\rho(t))\phi(u(t))\rho(t)\,dW(t)\\
=&LF(u(t))\,dt+l^{\prime}(\rho(t))\phi(u(t))\rho(t)\,dW(t),
\end{aligned}
\end{equation}
where $L$ is  defined  by
\[
LF(u(t)):=-\frac{l^{\prime}(\rho(t))}{\rho(t)}
\left\langle \rA u(t)+B(u(t)),\rA u(t)\right\rangle +
\frac12 l^{\prime\prime}(\rho(t))|\phi(u(t))|^2\rho^2(t),\;\;t\geq 0.
\]
By \autoref{prop-Assumption1-1} we infer that
\begin{equation}
\begin{aligned}
-\frac{l^{\prime}(\rho)}{\rho}
\left\langle \rA u+B(u),\rA u\right\rangle \leq
-\frac12\frac{l^{\prime}(\rho)}{\rho}|\rA u|_{\rH}^2
+\frac12 C_0\frac{l^{\prime}(\rho)}{\rho}\rho^6,
\end{aligned}
\end{equation}
where  $C_0$ is the constant defined in equality \eqref{eq-C_0}.

Therefore,
\begin{align}\label{eqn-LF}
L F(u)
\leq
-\frac12\frac{l^{\prime}(\rho)}{\rho}|\rA u|_{\rH}^2
+\frac12 C_0\frac{l^{\prime}(\rho)}{\rho}\rho^6
+\frac12 l^{\prime\prime}(\rho)|\phi(u)|^2\rho^2 .
\end{align}

Firstly, we assume that $\rho>2R_0$, then $l(\rho)=\rho^p$. By part (i) of \autoref{assumption-phi}, we infer that
\begin{equation}
\begin{aligned}
L F(u)
&\leq
-\frac p2\rho^{p-2}|\rA u|_{\rH}^2+\frac12 C_0p\rho^{p+4}-\frac12p(1-p)b^2\rho^{p+2\kappa}.
\end{aligned}
\end{equation}
For $\kappa>2$, we have $p+4<p+2\kappa$. Hence, by Young's inequality, for
every $\xi>0$ and $x\geq0$,
\[
C_0x^{p+4}\leq \xi x^{p+2\kappa}+C_\xi,
\]
where
\[
C_\xi
=
\frac{p+4}{2\kappa-4}
\left(
\frac{2\kappa-4}{p+4}
\right)^{\frac{p+4}{2\kappa-4}}
C_0^{\frac{p+4}{2\kappa-4}}
\xi^{-\frac{p+4}{2\kappa-4}}.
\]
Taking
\begin{equation}
\xi=\frac12(1-p)b^2,
\end{equation}
we infer that for $\kappa>2$, there exists a constant $C>0$ such that
\begin{equation}
LF(u)
\leq
C-\frac p2\rho^{p-2}|\rA u|_{\rH}^2.
\end{equation}

If $\kappa=2$, since $(1-p)b^2>C_0$,
\begin{align}
LF(u)
&\leq
-\frac p2\rho^{p-2}|\rA u|_{\rH}^2
-\frac p2
\left[
(1-p)b^2-C_0
\right]\rho^{p+4}
\leq
-\frac p2\rho^{p-2}|\rA u|_{\rH}^2.
\end{align}
Since $p-2<0$,
\begin{equation}
\rho^{p-2}
\geq
(1+\rho^2)^{\frac{p-2}{2}},
\qquad \rho>0.
\end{equation}
Consequently, there exist $C:=C(p)>0$ and $c>0$ such that
\begin{align}\label{eqn-LF-1}
L F(u)
\leq
C-c(1+\rho^2)^{\frac{p-2}{2}}
|\rA u|_{\rH}^2,
\qquad
\rho>2R_0.
\end{align}

If $0\leq \rho\leq 2R_0$, there exists a constant $C>0$ such that
\[
\vert\frac{l^{\prime}(\rho)}{\rho}\vert+|l^{\prime\prime}(\rho)|+\rho\leq C.
\]
Moreover, in view of \autoref{ass-g-loc Lip} and \autoref{assumption-phi-1}, for $\rho\leq 2R_0$, there exists a constant $C_{R_0}>0$ such that
\[
|\phi(u)|^2\rho^2
\leq C_{R_0}.
\]
Therefore, there exists $C:=C(R_0)$ such that
\begin{align}\label{eqn-LF-11}
C_0\frac{l^{\prime}(\rho)}{\rho}\rho^6
+
\frac12 l^{\prime\prime}(\rho)|\phi(u)|^2\rho^2
\leq C,
\end{align}
Note that for $0<\rho<R_0$
\[
\frac{l^{\prime}(\rho)}{\rho}(1+\rho^2)^{\frac{2-p}{2}}=2(1+\rho^2)^{\frac{2-p}{2}}\geq 2,
\]
and for $R_0\leq \rho\leq 2R_0$, let
\[
c_l:=\min_{\rho\in[R_0,2R_0]}\frac{l^{\prime}(\rho)}{\rho}(1+\rho^2)^{\frac{2-p}{2}}>0.
\]
Taking
\[
c:=\min\{2,c_l\} >0,
\]
then we infer that
\[
\frac{l^{\prime}(\rho)}{\rho}
\geq
c(1+\rho^2)^{\frac{p-2}{2}},\;\; 0< \rho\leq 2R_0.
\]
Therefore, inequality \eqref{eqn-LF} and \eqref{eqn-LF-11} implies there exists $c>0$ such that
\begin{align}\label{eqn-LF-2}
L F(u)
\leq C-c(1+\rho^2)^{\frac{p-2}{2}}|\rA u|_{\rH}^2,
\;\;0\leq \rho\leq 2R_0.
\end{align}

Consequently, by inequality \eqref{eqn-LF}, \eqref{eqn-LF-1} and \eqref{eqn-LF-2}, for $\rho\geq 0$ and $t\in[0,\bar{\tau}_m]$, there exists $C:=C(p,R_0)$ and $c>0$ such that
\[
dF(u(t))
+c(1+\rho^2(t))^{\frac{p-2}{2}}\vert\rA u(t)\vert_{\rH}^2\,dt\leq C\,dt
+l^{\prime}(\rho(t))\phi(u(t))\rho(t)\,dW(t).
\]
Taking expectations, for every $T>0$,
\begin{equation}\label{eqn-EF}
\begin{aligned}
&\mathbb{E} F(u(T\wedge\bar{\tau}_m))+c\mathbb{E}\int_0^{T\wedge\bar{\tau}_m}
(1+\rho^2(s))^{\frac{p-2}{2}}|\rA u(s)|_{\rH}^2\,ds \leq \mathbb{E} F(u_0)+CT.
\end{aligned}
\end{equation}

In the following, we aim to prove \eqref{eqn-6.23-1}. Since $R\geq 2^\frac{2}{2-p}>2R_0>0$. Define
\[
\sigma_R:=\inf\left\{t\in[0,\tau):\rho(t)\geq R\right\}.
\]
Therefore,
\begin{equation}\label{eqn-6.23-1-1}
\begin{aligned}
\mathbb{P}\left(
\sup_{t\in[0,T\wedge\bar{\tau}_m]}\vert\rA^{\frac12}u(t)\vert_{\rH}\geq R\right)
\leq \mathbb{P}\left(\sigma_R\leq T\wedge\bar{\tau}_m\right)
\end{aligned}
\end{equation}
Since $\rho(\sigma_R)=R>2R_0$ on $\{\sigma_R\leq T\wedge\bar{\tau}_m\}$, we infer that
\[
F(u(\sigma_R))=l(R)=R^p,
\]
which implies
\[
F(u(T\wedge\sigma_R\wedge\bar{\tau}_m))
\geq
R^p\mathbf{1}_{\{\sigma_R\leq T\wedge\bar{\tau}_m\}}.
\]
Hence, in view of \eqref{eqn-EF}, there exists $\hat{C}^1_T>0$ such that
\begin{align}\label{eqn-6.23-1-2}
\mathbb{P}(\sigma_R\leq T\wedge\bar{\tau}_m)
\leq
\frac{1}{R^p}\mathbb{E} F(u(T\wedge\sigma_R\wedge\bar{\tau}_m))
\leq
\frac{1}{R^p}(\mathbb{E}F(u_0)+CT)=:\frac{\hat{C}^1_T}{R^p}.
\end{align}
Therefore, \eqref{eqn-6.23-1-1} and \eqref{eqn-6.23-1-2} implies \eqref{eqn-6.23-1} holds.

In the following, we aim to prove \eqref{eqn-Au-L2-prob-1}. Since $R\geq 2^\frac{2}{2-p}>2R_0>0$,
\begin{equation}
\begin{aligned}
&\mathbb{P}\left(
\int_0^{T\wedge\bar{\tau}_m}\vert\rA u(t)\vert_{\rH}^2\,dt\geq R^2
\right) \\
&\leq
\mathbb{P}\left(
\sup_{t\in[0,T\wedge\bar{\tau}_m]}\rho(t)\geq R
\right)+
\mathbb{P}\left(
\int_0^{T\wedge\bar{\tau}_m}
\vert\rA u(t)\vert_{\rH}^2\,dt
\geq R^2,\,
\sup_{t\in[0,T\wedge\bar{\tau}_m]}\rho(t)<R
\right).
\end{aligned}
\end{equation}
Recall that $p\in (0,1)$. Note that if $\sup_{t\in[0,T\wedge\bar{\tau}_m]}\rho(t)<R$, then
\[
(1+\rho^2(t))^{\frac{p-2}{2}} \geq (1+R^2)^{\frac{p-2}{2}},
\]
which implies
\[
\int_0^{T\wedge\bar{\tau}_m}
\vert\rA u(t)\vert_{\rH}^2\,dt
\leq(1+R^2)^{\frac{2-p}{2}}
\int_0^{T\wedge\bar{\tau}_m}(1+\rho^2(t))^{\frac{p-2}{2}}\vert\rA u(t)\vert_{\rH}^2\,dt.
\]
Therefore, by Markov's inequality, there exists $\hat{C}^2_T>0$ such that
\begin{equation}
\begin{aligned}
&\mathbb{P}\left(\int_0^{T\wedge\bar{\tau}_m}\vert\rA u(t)\vert_{\rH}^2\,dt\geq R^2\right) \\
&\leq
\frac{\hat{C}^1_T}{R^p}
+\mathbb{P}\left(\int_0^{T\wedge\bar{\tau}_m}(1+\rho^2(t))^{\frac{p-2}{2}}\vert\rA u(t)\vert_{\rH}^2\,dt
\geq\frac{R^2}{(1+R^2)^{\frac{2-p}{2}}}\right) \\
&\leq
\frac{\hat{C}^1_T}{R^p}+\frac{(\mathbb{E}F(u_0)+CT)(1+R^2)^{\frac{2-p}{2}}}{cR^2}\leq  \frac{\hat{C}^2_T }{R^p} .
\end{aligned}
\end{equation}

The proof of \autoref{lem-prob} is complete.

\end{proof}

\section{Completion of the proof of Main Result}\label{sec-proof of the main result}

The aim of this section is to prove that the local maximal solution from Proposition \ref{prop-local solution maximal} is global.

The proof is divided into two parts. In \autoref{subsec-Poincar\'e domain-global}, we will proof the first case of our main result \autoref{thm-main}, where $\dom$ is either a bounded domain or an unbounded domain satisfying the Poincar\'e condition. The second case, when $\dom=\mathbb{R}^3$, will be treated in \autoref{subsec-Full space-global}.
In the whole section we use Assumption \ref{assumption-phi} and we will further use \autoref{assumption-phi-1} in \autoref{subsec-Full space-global}.

\subsection{Poincar\'e domain}\label{subsec-Poincar\'e domain-global}
In this subsection, we proof the first case of our main result \autoref{thm-main}. We assume in this subsection that $\dom$ is a bounded domain or an unbounded domain satisfying the Poincar\'e condition.

\begin{theorem}\label{thm-main-bounded}
Assume that $u_0$ is an $\mathscr{F}_0$-measurable $\rV$-valued random variable with $\mathbb{E}\vert u_0\vert_\rV^2 < \infty$ and the function $\phi$ satisfies \autoref{assumption-phi}, then there exists a unique global solution to the main problem \eqref{eqn-NS03}.
\end{theorem}
\begin{proof}
Recall that $\vert \cdot\vert_{X_T}$ is equivalent to $\vert \cdot\vert_{X_T^0}$ defined in \eqref{eq-X_T^0}. Hence, there exists $C>0$ such that
\begin{align}
\vert u\vert_{X_{T\wedge\bar{\tau}_m}}\leq
C\vert u\vert_{X_{T\wedge\bar{\tau}_m}^0}.
\end{align}
Let us choose and fix $T>0$. We infer there exists $C:=C(T)>0$ such that
\begin{align}\label{eqn-tau_m<T-prob-1}
\mathbb{P}\left(\bar{\tau}_m \leq  T \right)&\leq \mathbb{P}\left(\vert u\vert_{X_{T\wedge\bar{\tau}_m}}\geq m \right)\leq \mathbb{P}\left(\vert u\vert_{X_{T\wedge\bar{\tau}_m}^0}\geq \frac{m}{C} \right)\\ &=\mathbb{P}\left(\sup _{t \in[0, T\wedge \bar{\tau}_m]}\vert \rA^\frac{1}{2}u(t)\vert_\rH^2 +\int_0^{T\wedge \bar{\tau}_m}\vert \rA u(t)\vert_\rH^2 \, d t \geq \frac{m^2}{C^2} \right)\nonumber\\
&\leq \mathbb{P}\left(\sup _{t \in[0, T\wedge \bar{\tau}_m]}\vert \rA^\frac{1}{2}u(t)\vert_\rH^2 \geq \frac{m^2}{2C^2} \right)+\mathbb{P}\left(\int_0^{T\wedge \bar{\tau}_m}\vert \rA u(t)\vert_\rH^2 \, d t \geq \frac{m^2}{2C^2} \right) \nonumber,
\end{align}
therefore, in view of \autoref{lem-prob},
\begin{align}\label{eqn-tau_m<T-prob-1}
\mathbb{P}\left(\bar{\tau}_m\leq T \right)&\leq \hat{C}^1_T (\frac{m^2}{2C^2})^{-\frac p2}
+ \hat{C}^2_T (\frac{m^2}{2C^2})^{-\frac p2}
\nonumber\\
&= (\hat{C}^1_T+\hat{C}^2_T)(\sqrt{2}C)^pm^{-p}
\to 0,\mbox{ as } m\to \infty.\nonumber
\end{align}
Thus, in view of \eqref{eqn-tau-tau_m-2},
\begin{equation*}
\mathbb{P}(\tau \leq  T) = \lim_{m \to \infty} \mathbb{P}\left(\bar{\tau}_m\leq  T\right)=0.
\end{equation*}
Since the above holds for every $T \geq 0$, we infer that
\begin{align}
\mathbb{P}(\tau = +\infty) = 1.
\end{align}
The proof of \autoref{thm-main-bounded} is complete.
\end{proof}

\subsection{Full space}\label{subsec-Full space-global}
In this subsection, we assume that $\dom=\mathbb{R}^3$ and \autoref{assumption-phi-1} is satisfied. Now we can prove the second case of our main result \autoref{thm-main}.
\begin{theorem}\label{thm-main-full space}
Assume that $u_0$ is an $\mathscr{F}_0$-measurable $\rV$-valued random variable with $\mathbb{E}\vert u_0\vert_\rV^2 < \infty$ and the function $\phi$ satisfies \autoref{assumption-phi} and \autoref{assumption-phi-1}, then there exists a unique global solution to the main problem \eqref{eqn-NS03}.
\end{theorem}
\begin{proof}
\textbf{Step 1} We will show that
\begin{align}\label{eqn-8.6}
\mathbb{P}\left(\vert u\vert_{X_{T\wedge \bar{\tau}_m}^0}^2 \geq  \frac{m^2}{2} \right)\to 0,\;\;\mbox{ as }m\to \infty.
\end{align}

In view of the definition of $\vert \cdot\vert_{X_{T\wedge\bar{\tau}_m}} $ and $\vert \cdot\vert_{X_{T\wedge \bar{\tau}_m}^0}$, see \eqref{norm X_T} and \eqref{eq-X_T^0}, we have
\begin{align*}
\vert u\vert_{X_{T\wedge\bar{\tau}_m}}^2
&=\sup _{t \in[0, T\wedge \bar{\tau}_m]}[\vert \rA^\frac{1}{2}u(t)\vert_\rH^2+\left\Vert u(t)\right\Vert_{L^2}^2]+\int_0^{T\wedge \bar{\tau}_m}\Vert u(t)\Vert_{L^2}^2 \, d t   +\int_0^{T\wedge \bar{\tau}_m}\vert \rA u(t)\vert_\rH^2 \, d t\\
&\leq (T+1)\sup _{t \in[0, T\wedge \bar{\tau}_m]}\left\Vert u(t)\right\Vert_{L^2}^2+\sup _{t \in[0, T\wedge \bar{\tau}_m]}\vert \rA^\frac{1}{2}u(t)\vert_\rH^2+\int_0^{T\wedge \bar{\tau}_m}\vert \rA u(t)\vert_\rH^2 \, d t\\
&=(T+1)\sup _{t \in[0, T\wedge \bar{\tau}_m]}\left\Vert u(t)\right\Vert_{L^2}^2+\vert u\vert_{X_{T\wedge \bar{\tau}_m}^0}^2.
\end{align*}
Therefore,
\begin{equation}\label{eqn-tau_m<T-prob-11}
\begin{aligned}
\mathbb{P}\left(\bar{\tau}_m < T \right)&\leq \mathbb{P}\left(\vert u\vert_{X_{T\wedge\bar{\tau}_m}}\geq m \right)\\ &\leq \mathbb{P}\left((T+1)\sup _{t \in[0, T\wedge \bar{\tau}_m]}\left\Vert u(t)\right\Vert_{L^2}^2+\vert u\vert_{X_{T\wedge \bar{\tau}_m}^0}^2\geq m^2 \right)\\
&\leq \mathbb{P}\left(\sup _{t \in[0, T\wedge \bar{\tau}_m]}\left\Vert u(t)\right\Vert_{L^2}^2 \geq \frac{m^2}{2(T+1)} \right)+\mathbb{P}\left(\vert u\vert_{X_{T\wedge \bar{\tau}_m}^0}^2 \geq  \frac{m^2}{2} \right) .
\end{aligned}
\end{equation}
In view of definition of $\vert u\vert_{X_{T\wedge \bar{\tau}_m}^0}$, see \eqref{eq-X_T^0} and  \autoref{lem-prob}, we infer that there exists $C=C(T)>0$ such that for every $R>0$
\begin{align}\label{eqn-8.47}
\mathbb{P}\left(\vert u\vert_{X_{T\wedge \bar{\tau}_m}^0}^2 \geq  R^2 \right) &=\mathbb{P}\left(\sup _{t \in[0, T\wedge \bar{\tau}_m]}\vert \rA^\frac{1}{2}u(t)\vert_\rH^2 +\int_0^{T\wedge \bar{\tau}_m}\vert \rA u(t)\vert_\rH^2 \, d t \geq R^2 \right)\nonumber\\
&\leq \mathbb{P}\left(\sup _{t \in[0, T\wedge \bar{\tau}_m]}\vert \rA^\frac{1}{2}u(t)\vert_\rH^2 \geq \frac{R^2}{2} \right)+\mathbb{P}\left(\int_0^{T\wedge \bar{\tau}_m}\vert \rA u(t)\vert_\rH^2 \, d t \geq \frac{R^2}{2} \right) \nonumber\\
&\leq (\hat{C}^1_T+\hat{C}^2_T)(\sqrt{2})^pR^{-p},
\end{align}
what  implies that \eqref{eqn-8.6} as claimed.

\textbf{Step 2} In the following, our aim is to prove that
\begin{equation}\label{eqn-8.48}
\mathbb{P}\left(\sup _{t \in[0, T\wedge \bar{\tau}_m]}\left\Vert u(t)\right\Vert_{L^2}^2 \geq \frac{m^2}{2(T+1)} \right) \to 0,\;\;\mbox{ as } m\to \infty.
\end{equation}
For this we will estimate the LHS of \eqref{eqn-8.48}.

In view of \eqref{eq-local strong solution}, by the It\^o formula for the function  $\vert \cdot\vert_\rH^2$, see \cite[Lemma 1.4]{Pardoux_1979}, we infer that for every $t>0$,
\begin{equation}
\begin{aligned}
\left\vert u(t\wedge \bar{\tau}_m)\right\vert_\rH^2= & \left\vert u(0)\right\vert_\rH^2+\int_0^{t\wedge \bar{\tau}_m }\left(2{}_{\rV^\prime}\left\langle
-[\rA u(s)+B(u(s))], u(s)\right\rangle_\rV+\vert \phi(u(s))\vert^2 \vert u(s)\vert_\rH ^2\right)\, d s \\
& +2 \int_0^{t\wedge \bar{\tau}_m } \phi(u(s))\vert u(s)\vert_\rH ^2\,d W(s).
\end{aligned}
\end{equation}
Since the  operator $\rA$ is self-adjoint, by equality \eqref{eq-<B(u),u>},  we infer that
\begin{align}
{}_{\rV^\prime}\left\langle
\rA u+B(u), u\right\rangle_\rV={}_{\rV^\prime}\left\langle
\rA u, u\right\rangle_\rV=\vert \rA^\frac{1}{2}u\vert_\rH^2.
\end{align}
Therefore we deduce that for $t \geq 0$,
\begin{equation}
\begin{aligned}
\left\vert u(t\wedge \bar{\tau}_m)\right\vert_\rH^2
= & \left\vert u(0)\right\vert_\rH^2+\int_0^{t\wedge \bar{\tau}_m }-2\vert \rA^\frac{1}{2}u(s)\vert_\rH^2 +\vert \phi(u(s))\vert^2\vert u(s)\vert_\rH ^2\, d s \\
& +2 \int_0^{t\wedge \bar{\tau}_m }\phi(u(s))\vert u(s)\vert_\rH ^2\,d W(s).
\end{aligned}
\end{equation}
We now use the logarithmic Lyapunov function below to estimate the
missing $L^2$-part.
By using the Itô formula for the logarithmic Lyapunov function,
\begin{equation}\label{eq-Lyapunov-1}
F^\prime(u):=\log (1+\Vert u\Vert_{L^2}^2), \;\; u \in D(\rA^\frac12),
\end{equation}
Therefore, for $t\geq 0$
\begin{equation}\label{eqn-Lyapunov-H norm-1}
\begin{aligned}
d\log \left(1+\left\Vert u(s)\right\Vert_{L^2}^2\right)=&-\frac{2\vert \rA^\frac{1}{2}u\vert_\rH^2}{1+\Vert u\Vert_{L^2}^2}\, d s+  \frac{ \vert \phi(u)\vert^2\Vert u\Vert_{L^2}^2}{1+\Vert u\Vert_{L^2}^2}\, d s-\frac{ 2\vert \phi(u)\vert^2\Vert u\Vert_{L^2}^4}{\left(1+\Vert u\Vert_{L^2}^2\right)^2}\, d s\\
&+ \frac{2\phi(u) \Vert u\Vert_{L^2}^2}{1+\Vert u\Vert_{L^2}^2} \, d W(s)\\
\leq& \vert \phi(u)\vert^2\, d s-\frac{ 2\vert \phi(u)\vert^2\Vert u\Vert_{L^2}^4}{\left(1+\Vert u\Vert_{L^2}^2\right)^2}\, d s+ \frac{2\phi(u) \Vert u\Vert_{L^2}^2}{1+\Vert u\Vert_{L^2}^2} \, d W(s).
\end{aligned}
\end{equation}
Define an $\mathbb{R}$-valued  process $\hat{\rM}=(\hat{\rM}_t: t\geq 0)$ by
\[
\hat{\rM}_t=2 \int_0^{t}\frac{\phi(u) \Vert u(s)\Vert_{L^2}^2}{1+\Vert u(s)\Vert_{L^2}^2} \, d W(s).
\]
Note that for every $t\geq 0$, by \autoref{assumption-phi}
\begin{equation}\label{eqn-EM_t1-1a}
\begin{aligned}
&\mathbb{E}  \int_0^{t\wedge \bar{\tau}_m} \frac{ 4\vert \phi(u)\vert^2\Vert u(s)\Vert_{L^2}^4 }{(1+\Vert u(s)\Vert_{L^2}^2)^2} \, ds
\leq 4\mathbb{E}  \int_0^{t  \wedge \bar{\tau}_m}   \vert \phi(u)\vert^2\, ds  <\infty.
\end{aligned}
\end{equation}
It follows that the stopped processes $\hat{\rM}_{t \wedge \bar{\tau}_m}$ are martingales. Moreover, by the properties of the Itô integral, the quadratic variation $\langle \hat{\rM}\rangle_t$ satisfies
\begin{equation}
\begin{aligned}\label{eq-M_t11a}
\langle \hat{\rM} \rangle_t &= 4\int_0^t \frac{ \vert \phi(u)\vert^2\Vert u(s)\Vert_{L^2}^4 }{(1+\Vert u(s)\Vert_{L^2}^2)^2} \, ds.
\end{aligned}
\end{equation}
Then inequality \eqref{eqn-Lyapunov-H norm-1} implies
\begin{equation}\label{eqn-Lyapunov-H norm-1a}
\begin{aligned}
&\log \left(1+\left\Vert u(t\wedge \bar{\tau}_m)\right\Vert_{L^2}^2\right)\\
&\leq \log \left(1+\left\Vert u_0\right\Vert_{L^2}^2\right)+  \int_0^{t\wedge \bar{\tau}_m} \vert \phi(u)\vert^2\, d s+\hat{\rM}_{t\wedge \bar{\tau}_m}- \frac{1}{2}\langle \hat{\rM} \rangle_{t\wedge \bar{\tau}_m}.
\end{aligned}
\end{equation}
Define the process $\hat{Z}$ by
\[\hat{Z}_t:=e^{\hat{\rM}_t- \frac{1}{2}\langle \hat{\rM} \rangle_t}.\]
Note that for every $R>0$
\begin{align*}
& \mathbb{P}\left(\sup _{t \in[0, T\wedge \bar{\tau}_m]}\left\Vert u(t)\right\Vert_{L^2}^2 \geq R \right)=\mathbb{P}\left(\sup _{t \in[0, T\wedge \bar{\tau}_m]}\log(1+\left\Vert u(t)\right\Vert_{L^2}^2) \geq \log(1+R) \right)\nonumber\\
\leq &\mathbb{P}\left(\sup _{t \in[0, T\wedge \bar{\tau}_m]}\left[  \int_0^{t\wedge \bar{\tau}_m} \vert \phi(u)\vert^2\, d s+\log\hat{Z}_{t\wedge \bar{\tau}_m}+\log(1+\Vert u_0\Vert_{L^2}^2)\right]\geq \log( 1+R) \right) \nonumber\\
\leq & \mathbb{P}\left(\sup _{t \in[0, T\wedge \bar{\tau}_m]}\int_0^{t} \vert \phi(u)\vert^2\, d s\geq \frac{1}{2} \log( 1+R)\right)\nonumber\\ &+\mathbb{P}\left(\sup _{t \in[0, T\wedge \bar{\tau}_m]}\log\hat{Z}_t \geq  \log\frac{(1+R)^{1/2}}{1+\left\Vert u_0\right\Vert_{L^2}^2} \right).
\end{align*}
By \autoref{prop-Z_t} we infer that
\begin{align}\label{eq-u L2 norm-prob-1}
& \mathbb{P}\left(\sup _{t \in[0, T\wedge \bar{\tau}_m]}\left\Vert u(t)\right\Vert_{L^2}^2 \geq R \right)\\
\leq & \mathbb{P}\left( \sup_{t \in [0, T \wedge \bar{\tau}_m]} \vert \phi(u)\vert^2
\geq  \frac{1}{2T} \log(1+R) \right)+\mathbb{P}\left(\sup _{t \in[0, T\wedge \bar{\tau}_m]}\hat{Z}_t \geq \frac{(1+R)^{1/2}}{1+\left\Vert u_0\right\Vert_{L^2}^2} \right)\nonumber\\
\leq & \mathbb{P}\left( \sup_{t \in [0, T \wedge \bar{\tau}_m]} \vert \phi(u)\vert^2
\geq  \frac{1}{2T} \log(1+R) \right)+\frac{1+\mathbb{E}\left\Vert u_0\right\Vert_{L^2}^2}{(1+R)^{1/2}}.
\end{align}
By \autoref{assumption-phi-1}, recall that $\Gamma$ is an increasing function such that
\begin{equation}\label{eqn-lim fx=infty-1}
\lim_{x\to \infty} \Gamma(x)= \infty.
\end{equation}
Therefore, we infer that
\begin{equation}\label{eqn-8.50}
\begin{aligned}
\mathbb{P}\left( \sup_{t \in [0, T \wedge \bar{\tau}_m]} \vert \phi(u)\vert^2
\geq  \frac{1}{2T} \log(1+R) \right)\leq &\mathbb{P}\left( \Gamma(\vert u\vert_{X_{T\wedge \bar{\tau}_m}^0}^2)
\geq  \frac{1}{2T} \log(1+R) \right)\\
=& \mathbb{P}\left(\vert u\vert_{X_{T\wedge \bar{\tau}_m}^0}^2
\geq \Gamma^{-1} (\frac{1}{2T} \log(1+R) )\right).
\end{aligned}
\end{equation}\
Thus, in view of \eqref{eqn-8.47}
\begin{align}\label{eqn-8.17}
\mathbb{P}\left(\vert u\vert_{X_{T\wedge \bar{\tau}_m}^0}^2
\geq \Gamma^{-1} (\frac{1}{2T} \log(1+\frac{m}{2(T+1)}) )\right)\to 0,\;\;\mbox{ as } m\to \infty.
\end{align}
Hence, by \eqref{eq-u L2 norm-prob-1}, \eqref{eqn-8.50} and \eqref{eqn-8.17} we deduce
\begin{equation*}\label{eqn-8.48-2}
\begin{aligned}
&\mathbb{P}\left(\sup _{t \in[0, T\wedge \bar{\tau}_m]}\left\Vert u(t)\right\Vert_{L^2}^2 \geq \frac{m}{2(T+1)} \right) \\
\leq & \mathbb{P}\left(\vert u\vert_{X_{T\wedge \bar{\tau}_m}^0}^2
\geq \Gamma^{-1} (\frac{1}{2T} \log(1+\frac{m}{2(T+1)}) )\right)+\frac{1+\left\Vert u_0\right\Vert_{L^2}^2}{(1+\frac{m}{2(T+1)})^{1/2}}\to 0,\;\;\mbox{ as } m\to \infty.
\end{aligned}
\end{equation*}
what implies \eqref{eqn-8.48}.

\textbf{Step 3} By \eqref{eqn-8.6} and \eqref{eqn-8.48}, we infer
\begin{equation}\label{eqn-tau_m<T-prob-111}
\begin{aligned}
\mathbb{P}\left(\bar{\tau}_m < T \right)\to 0,\mbox{ as } m\to \infty.
\end{aligned}
\end{equation}
Therefore, in view of \eqref{eqn-tau-tau_m-2},
\begin{equation*}
\mathbb{P}(\tau < T) = \lim_{m \to \infty} \mathbb{P}\left(\bar{\tau}_m< T\right)=0.
\end{equation*}
Since the above holds for every $T \geq 0$, we infer that
\begin{align}
\mathbb{P}(\tau = +\infty) = 1.
\end{align}
The proof of \autoref{thm-main-full space} is complete.
\end{proof}

\section{Examples}\label{Examples}
In this section, we provide examples of noises satisfying our assumptions. In particular, the framework considered by Hong, Li, and Liu \cite{Hong+Li+Liu_2024} is included as a special case of our setting. Consequently, we extend the corresponding result in their paper from bounded domains to unbounded domains.
\begin{example}\label{example-phi}
Assume that  $\kappa>2$, $c_1> 0$ and $c_0\geq 0$. Then the  following function
\begin{equation}\label{eqn-phi-ex}
    \phi: \rV \ni u \mapsto  c_1\vert \rA^\frac{1}{2}u\vert_\rH^\kappa+c_0 \in \mathbb{R}
\end{equation}
satisfies Assumption
\ref{assumption-phi} and \ref{assumption-phi-1}. Moreover, $g(u)=\phi(u)u$ satisfies Assumption \ref{ass-g-loc Lip}.
\end{example}
\begin{proof}
Obviously, $\phi$ satisfies (i) in \autoref{assumption-phi}. To continue let us formulate the following simple but useful lemma.
\begin{lemma}\label{lem-simple}
Assume that  $\kappa> 2$, $c_0\geq 0$ and $c_1>0$.     Then  there exists $\delta \in (0,1)$ and $C>0$ such that
\begin{equation}\label{eqn-A3a-1}
(c_1x^\kappa + c_0)^2x^2
\leq \delta c_1^2 x^{2\kappa+4}+2\delta c_1c_0 x^{\kappa+4}+c_0^2  x^{2}+C, \;\; \mbox{ for all } x \geq 0.
\end{equation}
\end{lemma}
\begin{proof}[Proof of Lemma \ref{lem-simple}]
Since $\kappa> 2$, by the classical Young inequality applied twice, we have for all $ x\geq 0$,
\begin{align}\label{eqn-A3a-11}
(c_1x^\kappa + c_0)^2x^2
&=
c_1^2 x^{2\kappa+2}+2c_1c_0 x^{\kappa+2}+c_0^2  x^{2}
\\
&\leq c_1^2(\frac{\kappa+1}{\kappa+2} x^{2\kappa+4}+\frac{1}{\kappa+2})+2 c_1c_0 (\frac{\kappa+2}{\kappa+4} x^{\kappa+4}+\frac{2}{\kappa+4})
+c_0^2  x^{2}\nonumber\\
&=c_1^2\frac{\kappa+1}{\kappa+2} x^{2\kappa+4}+2 c_1c_0 \frac{\kappa+2}{\kappa+4} x^{\kappa+4}+c_0^2  x^{2}+\frac{1}{\kappa+2}+\frac{2}{\kappa+4}.
\nonumber\\
&\leq \delta c_1^2 x^{2\kappa+4}+2\delta c_1c_0 x^{\kappa+4}+c_0^2  x^{2}+C,\nonumber
\end{align}
where
\begin{align}\label{eq-eta}
\delta:=\max\{\frac{\kappa+1}{\kappa+2},\frac{\kappa+2}{\kappa+4}\} \mbox{ and } C:=\frac{1}{\kappa+2}+\frac{2}{\kappa+4}.
\end{align}
Note that $\delta\in (0,1)$.
The proof of \autoref{lem-simple} is complete.
\end{proof}

Let us choose and fix $T>0$ and $R>0$, assume that $u \in X_T$ satisfies $\vert u \vert_{X_T} \leq R$. In view of \eqref{eqn-A1/2u X_t}, for every $t\in[0,T]$
\begin{align*}
\vert \rA^\frac{1}{2}u(t)\vert_\rH\leq \vert u\vert_{X_T}\leq R.
\end{align*}
Hence, we infer that
\begin{equation}\label{eqn-A3-1}
\begin{aligned}
\sup_{t\in [0,T]}\vert \phi(u(t))\vert &=\sup_{t\in [0,T]}c_1\vert \rA^\frac{1}{2}u\vert_\rH^\kappa+c_0\leq c_1R^\kappa+c_0,
\end{aligned}
\end{equation}
which implies condition (ii) in \autoref{assumption-phi} holds.

Let us define an auxiliary function
\[
\Gamma: [0,\infty) \ni x\mapsto (c_1^2+1)(x^\kappa+c_0^2).
\]
Obviously, this is an increasing function  satisfying condition \eqref{eqn-lim fx=infty}.

Let us choose and fix $T>0$. Note that for every $u\in X_T$, by the Young inequality and \eqref{eq-X_T^0}, we infer that
\begin{equation}
\begin{aligned}
\sup_{t\in [0,T]}\vert\phi(u(t))\vert^2&=\sup_{t\in [0,T]}\left[c_1^2\vert \rA^\frac{1}{2}u(t)\vert_\rH^{2\kappa}+2c_1c_0\vert \rA^\frac{1}{2}u(t)\vert_\rH^\kappa+c_0^2\right]\\
&\leq \sup_{t\in [0,T]}\left[(c_1^2+1)\vert \rA^\frac{1}{2}u(t)\vert_\rH^{2\kappa}+(c_1^2+1)c_0^2\right]\\
&= (c_1^2+1)\left[\sup_{t\in [0,T]}\vert \rA^\frac{1}{2}u(t)\vert_\rH^{2\kappa}+c_0^2\right]\\
&\leq (c_1^2+1)(\vert u\vert_{X_T^0}^{2\kappa}+c_0^2)= \Gamma(\vert u\vert_{X_T^0}^2).
\end{aligned}
\end{equation}
Therefore, \autoref{assumption-phi-1} holds.

Let us choose and fix $u,\rv\in \rV$ satisfies $\vert  u\vert_\rV\leq R$ and $\vert\rv\vert_\rV\leq R$. Therefore,
\begin{align}\label{eqn-3.12}
\max\{\vert \rA^\frac{1}{2}u(t)\vert_\rH, \vert \rA^\frac{1}{2}\rv(t)\vert_\rH\}\leq R.
\end{align}
Note that for $0\leq x_1<x_2$, by the Fundamental Theorem of calculus, since $\kappa>1$,
\begin{equation}\label{eqn-3.13}
\begin{aligned}
x_2^\kappa -x_1^\kappa =\int_{x_1}^{x_2}\kappa s^{\kappa-1}\,ds \leq \kappa x_2^{\kappa-1}(x_2-x_1).
\end{aligned}
\end{equation}
By inequalities \eqref{eqn-3.13}, \eqref{eqn-3.12} and the triangle inequality we infer that for every $t\in[0,T]$
\begin{align}\label{eqn-3.10-1}
\Big\vert\vert \rA^\frac{1}{2}u(t)\vert_\rH^\kappa-\vert \rA^\frac{1}{2}\rv(t)\vert_\rH^\kappa\Big\vert
 &\leq \kappa (\max\{\vert \rA^\frac{1}{2}u(t)\vert_\rH, \vert \rA^\frac{1}{2}\rv(t)\vert_\rH\})^{\kappa-1} \Big\vert \vert \rA^\frac{1}{2}u(t)\vert_\rH-\vert \rA^\frac{1}{2}\rv(t)\vert_\rH\Big\vert \nonumber\\
 &\leq \kappa R^{\kappa-1} \vert \rA^\frac{1}{2}u(t)-\rA^\frac{1}{2}\rv(t)\vert_\rH,
\end{align}
which implies
\begin{equation}\label{eqn-3.10}
\begin{aligned}
\vert \phi(u)-\phi(\rv)\vert &= c_1 \Big\vert\vert \rA^\frac{1}{2}u\vert_\rH^\kappa-\vert \rA^\frac{1}{2}\rv\vert_\rH^\kappa\Big\vert\\
 &\leq c_1\kappa R^{\kappa-1} \vert \rA^\frac{1}{2}u-\rA^\frac{1}{2}\rv\vert_\rH\\
&= c_1\kappa R^{\kappa-1} \vert u-\rv\vert_\rV.
\end{aligned}
\end{equation}
Therefore, \autoref{ass-g-loc Lip} holds.
The proof of \autoref{example-phi} is complete.
\end{proof}

\begin{remark}
Let $n\in\mathbb{N}$. Assume that constants $\kappa_1,\dots,\kappa_n>2$, $c_0, c_1,\dots,c_n\geq 0$ satisfy that there exists $i=1,2,\dots, n$ such that $c_i > 0$. The function $\phi:\rV\mapsto \mathbb{R}$ is defined by
\begin{align}\label{eqn-phi-ex-1}
\phi(u):= \sum_{i=1}^n c_i\vert \rA^\frac{1}{2}u\vert_\rH^{\kappa_i}+c_0,\;\;u\in \rV.
\end{align}
Arguing as in  \autoref{example-phi} we can prove  that function $\phi$ defined above  satisfies Assumptions \ref{assumption-phi} and \ref{assumption-phi-1}.
\end{remark}
\begin{cor}
Let $u_0$ be an $\mathscr{F}_0$-measurable $\rV$-valued random variable with $\mathbb{E}\vert u_0\vert_\rV^2 < \infty$, and the function $\phi$ is defined in \eqref{eqn-phi-ex} or \eqref{eqn-phi-ex-1}.
Then there exists a unique global solution to the main problem \eqref{eqn-NS03}.
\end{cor}

\section{Deterministic nonlinear damping}
\label{sec-deterministic damping}

In this section we study the Navier-Stokes Equation with strong nonlinear damping. We formulate a result when $\dom$ is a Poincar{\'e} domain.
The result below should be compared with Theorem \ref{thm-main}.

\begin{theorem}\label{thm-deterministic damping}
Assume that \autoref{ass-g-loc Lip} is  satisfied, and the following assumption is satisfied
\begin{trivlist}
\item[(i)]  $\dom$ is a bounded domain or an unbounded domain satisfying the Poincar\'e condition.
\end{trivlist}
We further assume that
\begin{trivlist}
\item[(iii)]  there exists constants $\kappa>4$ and $b>0$ such that
\begin{equation}\label{eqn-phi-deterministic}
  \vert\phi(u)\vert \geq b\vert \rA^\frac12 u\vert_\rH^\kappa, \;\; u \in D(\rA^\frac12).
  \end{equation}
  \end{trivlist}
or
\begin{trivlist}
\item[(iv)]  There exists $b> C_\nu$,
where
\begin{equation}\label{eqn-C_nu}
C_\nu:=\frac{27C^4}{256\nu^3}
\end{equation}
and $C$ is the constant from the Agmon inequality \eqref{eqn-the Agmon inequality}, such that
\begin{equation}\label{eqn-phi-deterministic-1}
  \vert\phi(u)\vert \geq b\vert \rA^\frac12 u\vert_\rH^4, \;\; u \in D(\rA^\frac12).
  \end{equation}
  \end{trivlist}
Then for every  $u_0 \in \rV$ and $f\in L^2_\loc(0,\infty;\rH)$, there exists a unique global solution
\[
u \in C([0,\infty);\rV) \cap L^2_{\loc}([0,\infty);D(\rA))
\]
 to the following nonlinearly damped NSEs:
\begin{equation}\label{eq-damped-2}
\begin{aligned}
&\frac{\partial u(t)}{\partial t} + \nu \rA u(t) + B(u(t)) = f(t) - \phi(u(t))\,u(t), \; t >0,\\
&u(0)=u_0.
\end{aligned}
\end{equation}
\end{theorem}

\begin{proof}[Sketch of the proof]
In view of \autoref{ass-g-loc Lip}, $-\phi(u)u$ is locally Lipschitz in $\rV$. By applying the Banach Fixed Point Theorem, there exists a unique maximal local strong solution $(u,\tau)$, where
$\tau \in (0,\infty]$ such that
\[
u \in C([0,\tau);\rV)\cap L^2_{\loc}([0,\tau);D(\rA)).
\]

To prove that there exists a global solution $u$ such that
\[
u \in C([0,\infty);\rV) \cap L^2_{\loc}([0,\infty);D(\rA)),
\]
it suffices to show that if $\tau<\infty$, then
\begin{align}\label{eqn-blowup}
\sup_{t\in [0,\tau)}\vert \rA^{1/2}u(t)\vert_\rH+\int_0^\tau \vert \rA u(t)\vert_\rH^2\, dt<\infty.
\end{align}

Let us choose and fix $\eps\in(0,1)$. Define
\begin{align}\label{eq-Cnueps}
C_{\nu,\eps}:=\frac{27C^4}{256\,\eps^3\nu^3}.
\end{align}
where $C$ is the constant from the Agmon inequality \eqref{eqn-the Agmon inequality}.

In view of \eqref{eqn-BA-L2}, for every $u\in D(\rA)$
\begin{align}\label{eqn-106}
-\langle B(u), \rA u\rangle_{\rH} \leq C\vert\rA u\vert_\rH^\frac32 \vert\rA^{1/2}u\vert_\rH^\frac32
\leq C_{\nu,\eps}\vert\rA^{1/2}u\vert_\rH^6+\eps\nu\vert\rA u\vert_\rH^2.
\end{align}
Moreover, by the Cauchy--Schwarz inequality and the Young inequality,
\begin{align}\label{eqn-107}
\langle f,\rA u\rangle_{\rH}\leq \vert f\vert_\rH\vert \rA u\vert_\rH
\leq C_{\nu,\eps}^\prime\vert f\vert_\rH^2+\frac12(1-\eps)\nu\vert \rA u\vert_\rH^2,
\end{align}
where
\[
C_{\nu,\eps}^\prime:=\frac{1}{2(1-\eps)\nu}.
\]

Taking the inner product of \eqref{eq-damped-2} with $\rA u$ in $\rH$, by inequalities \eqref{eqn-106} and \eqref{eqn-107} we infer that
\begin{align}\label{eqn-det-damp}
\frac12\frac{d}{dt}\vert \rA^{1/2}u\vert_\rH^2
+\nu\vert \rA u\vert_\rH^2
+\vert \phi(u)\vert\vert \rA^{1/2}u\vert_\rH^2
=&
\langle f,\rA u\rangle_{\rH}
-\langle B(u),\rA u\rangle_{\rH}\\
\leq & C_{\nu,\eps}^\prime\vert f\vert_\rH^2
+ C_{\nu,\eps} \vert\rA^{1/2}u\vert_\rH^{6}+\frac12(1+\eps)\nu\vert\rA u\vert_\rH^2,\nonumber
\end{align}
which implies
\begin{equation}\label{eqn-det-damp-y}
\frac12 \frac{d}{dt}\vert \rA^{1/2}u\vert_\rH^2+\frac12(1-\eps)\nu\vert \rA u\vert_\rH^2+\vert\phi(u)\vert\,\vert \rA^{1/2}u\vert_\rH^2
\leq C_{\nu,\eps}^\prime\vert f\vert_\rH^2+C_{\nu,\eps} \vert \rA^{1/2}u\vert_\rH^6.
\end{equation}

\textbf{Case 1:} Assume assumption (iii), i.e.   there exists constants $\kappa>4$ and $b>0$ such that
\eqref{eqn-phi-deterministic} holds. Then
\[
\vert\phi(u(t))\vert\vert \rA^{1/2}u(t)\vert_\rH^2\geq b \vert \rA^{1/2}u(t)\vert_\rH^{\kappa+2},\;\;t\geq 0.
\]
Since $\kappa+2>6$, by the Young inequality,
there exists a constant $C>0$ such that
\[
C_{\nu,\eps} \vert \rA^{1/2}u(t)\vert_\rH^6 \leq \frac12 b\vert \rA^{1/2}u(t)\vert_\rH^{\kappa+2} + C.
\]
We infer from \eqref{eqn-det-damp-y} that
\[
\frac12 \frac{d}{dt}\vert \rA^{1/2}u(t)\vert_\rH^2+\frac12(1-\eps)\nu\vert \rA u(t)\vert_\rH^2
+\frac{b}{2}\vert \rA^{1/2}u(t)\vert_\rH^{\kappa+2}
\leq C_{\nu,\eps}^\prime\vert f(t)\vert_\rH^2 + C.
\]
Therefore, there exists $C>0$ such that
\[
\frac{d}{dt}\vert \rA^{1/2}u(t)\vert_\rH^2+\frac12(1-\eps)\nu\vert \rA u(t)\vert_\rH^2\leq C\bigl(1+\vert f(t)\vert_\rH^2\bigr).
\]
Since $u_0 \in \rV$ and $f\in L^2_\loc(0,\infty;\rH)$, assume that $\tau<\infty$ then we infer that for $t\in [0,\tau)$,
\[
\vert \rA^{1/2}u(t)\vert_\rH^2+(1-\eps)\nu\int_0^t \vert \rA u(s)\vert_\rH^2\,ds
\leq \vert \rA^{1/2}u_0\vert_\rH^2+ Ct + C\int_0^t \vert f(s)\vert_\rH^2\,ds <\infty.
\]
Therefore, \eqref{eqn-blowup} holds.

\textbf{Case 2:} Assume assumption (iv), i.e. there exist $\eps>0$ and $b> C_\nu$,  with $b\geq C_{\nu,\eps}$, where $C_\nu$ and $C_{\nu,\eps}$
are defined in \eqref{eqn-C_nu} and \eqref{eq-Cnueps} respectively,
 such that condition \eqref{eqn-phi-deterministic-1} is satisfied.  Thus, by \eqref{eqn-phi-deterministic-1}, we infer that
\[
\vert\phi(u)\vert\vert \rA^{1/2}u(t)\vert_\rH^2\geq b\vert \rA^{1/2}u(t)\vert_\rH^6 \geq C_{\nu,\eps} \vert \rA^{1/2}u(t)\vert_\rH^6.
\]
Substituting this into \eqref{eqn-det-damp-y}, we infer that
\begin{align*}
\frac12 \frac{d}{dt}\vert \rA^{1/2}u(t)\vert_\rH^2+\frac12(1-\eps)\nu\vert \rA u(t)\vert_\rH^2
&\leq C_{\nu,\eps}^\prime\vert f(t)\vert_\rH^2+C_{\nu,\eps} \vert \rA^{1/2}u(t)\vert_\rH^6-\vert \phi(u)\vert \vert \rA^{1/2}u(t)\vert_\rH^2\\
&\leq C_{\nu,\eps}^\prime\vert f(t)\vert_\rH^2.
\end{align*}
Since $u_0 \in \rV$ and $f\in L^2_\loc(0,\infty;\rH)$, assume that $\tau<\infty$ then we infer that for $t\in [0,\tau)$,
\[
\vert \rA^{1/2}u(t)\vert_\rH^2+(1-\eps)\nu\int_0^t \vert \rA u(s)\vert_\rH^2\,ds
\leq \vert \rA^{1/2}u_0\vert_\rH^2+ 2C_{\nu,\eps}^\prime \int_0^t \vert f(s)\vert_\rH^2\,ds <\infty,
\]
which implies \eqref{eqn-blowup} holds.

The proof is complete.
\end{proof}

\begin{remark}
    Condition (iii) can be weakened to the following one:
\begin{trivlist}
\item[(iii')]  there exist constants $\kappa>4$,  $b>0$ and $R>0$ such that
\begin{equation}\label{eqn-phi-deterministic1}
\vert \phi(u)\vert \geq b\vert \rA^\frac12 u\vert_\rH^\kappa, \;\; \forall u \in D(\rA^\frac12)\mbox{ with } \vert \rA^{1/2}u\vert_{\rH}^2 \geq R.
  \end{equation}
  \end{trivlist}
  Condition (iv) can be weakened to the following one:
\begin{trivlist}
\item[(iv')]  There exists $b> C_\nu$, where $C_\nu$ is defined by \eqref{eqn-C_nu}, and $R>0$  such that
\begin{equation}\label{eqn-phi-deterministic-11}
\vert\phi(u)\vert \geq b\vert \rA^\frac12 u\vert_\rH^4, \;\; \forall u \in D(\rA^\frac12)\mbox{ with } \vert \rA^{1/2}u\vert_{\rH}^2 \geq R.
  \end{equation}
  \end{trivlist}

\end{remark}

\appendix
\section{Martingales}\label{sec-Martingales}
To prove the global existence result, we require a martingale inequality, which we state and prove here for completeness.
\begin{prp}\label{prop-Z_t}
Let $\rM_{t \wedge \bar{\tau}_m}$ be a martingale and $a>0$. Define the process $Z$ by
\[
Z_t := e^{ a \rM_t - \frac{a^2}{2} \langle \rM \rangle_t }.
\]
Then, for every $\lambda > 0$,
\begin{align}\label{eqn-Z_t-constant}
\mathbb{P}\left(\sup_{t \in [0, T]} Z_{t \wedge \bar{\tau}_m} \geq \lambda \right) \leq \frac{1}{\lambda}.
\end{align}

Moreover, if $\lambda$ is replaced by a positive $\mathcal{F}_0$-measurable
random variable $\Lambda$, then
\begin{align}\label{eqn-Z_t-rv}
\mathbb{P}\left(
\sup_{t\in[0,T]}Z_{t\wedge\bar{\tau}_m}\geq \Lambda
\right)
\leq \mathbb{E}\left(
\frac{1}{\Lambda}\right).
\end{align}
\end{prp}

\begin{proof}
Since $\rM_t$ is a local martingale, it follows from \cite[Chapter 4, Proposition 3.4]{Revuz+Yor_1999} that $Z_t$ is a nonnegative local martingale. By \cite[Chapter 1, Proposition 5.19(ii)]{Karatzas+Shreve_1991}, $Z_t$ is thus a nonnegative supermartingale.

Applying the maximal inequality for nonnegative supermartingales \cite[Theorem 1.7 in Chapter II]{Revuz+Yor_1999}, we deduce that
\[
\mathbb{P} \left( \sup_{t \in [0, T]} Z_{t \wedge \bar{\tau}_m} \geq \lambda \right)
\leq \frac{1}{\lambda} \sup_{t \in [0, T]} \mathbb{E}  Z_{t \wedge \bar{\tau}_m}
\leq \frac{1}{\lambda} \mathbb{E}  Z_0
= \frac{1}{\lambda}.
\]

In the following we aim to prove inequality \eqref{eqn-Z_t-rv}.
Firstly we aim to show that \eqref{eqn-Z_t-rv} holds for simple random variables.
Assume that $\tilde\Lambda$ is a simple random variable, such that
\[
\tilde\Lambda=\sum_{i=1}^N \lambda_i \mathbf{1}_{A_i},
\]
where $\lambda_i>0$, $A_i\in\mathcal{F}_0$, $A_i\cap A_j=\varnothing\ (i\neq j)$ and $\bigcup_{i=1}^N A_i=\Omega$.

Let us choose and fix $i=1,\cdots,N$ and $0\leq s\leq t\leq T$. Since $A_i\in\mathcal{F}_0\subset \mathcal{F}_s$
\begin{align*}
\mathbb{E}\left[\mathbf{1}_{A_i} Z_{t\wedge\bar{\tau}_m}| \mathcal F_s\right]
&=
\mathbf{1}_{A_i}
\mathbb{E}\left[Z_{t\wedge\bar{\tau}_m}|\mathcal F_s\right]
\leq
\mathbf{1}_{A_i} Z_{s\wedge\bar{\tau}_m},
\end{align*}
which implies
$\{\mathbf{1}_{A_i}Z_{t\wedge\bar{\tau}_m}\}_{t\in[0,T]}$
is a nonnegative supermartingale.
Therefore, by the maximal inequality for nonnegative supermartingales \cite[Theorem 1.7 in Chapter II]{Revuz+Yor_1999}, we infer that
\[
\mathbb{P}\left(
A_i\cap \{\sup_{t\in[0,T]}Z_{t\wedge\bar{\tau}_m}\geq \lambda_i\}
\right)
\leq \frac{1}{\lambda_i}
\mathbb{E}\left[\mathbf{1}_{A_i}Z_0\right]=\frac{1}{\lambda_i}\mathbb{P}(A_i).
\]
Therefore,
\begin{equation}\label{eqn-simple rv}
\begin{aligned}
\mathbb{P}(\sup_{t\in[0,T]}Z_{t\wedge\bar{\tau}_m}\geq \tilde\Lambda)
&=
\sum_{i=1}^N
\mathbb{P}\left(
A_i\cap \{\sup_{t\in[0,T]}Z_{t\wedge\bar{\tau}_m}\geq \lambda_i\}
\right)  \\
&\leq
\sum_{i=1}^N
\frac{\mathbb{P}(A_i)}{\lambda_i} =
\mathbb{E}\left(\frac{1}{\tilde\Lambda}\right).
\end{aligned}
\end{equation}

Now assume that $\Lambda$ is a positive $\mathcal F_0$-measurable
random variable. There exists a sequence of positive simple
$\mathcal F_0$-measurable random variables $\{\Lambda_n\}_{n\in\mathbb N}$
such that
\[
0<\Lambda_n\leq \Lambda,\mbox{ and }  \Lambda_n\uparrow \Lambda\;\;
\mathbb{P}\text{-a.s.}
\]
By \eqref{eqn-simple rv}, we have
\[
\mathbb{P}\left(
\sup_{t\in[0,T]}Z_{t\wedge\bar{\tau}_m}\geq \Lambda_n
\right)
\leq
\mathbb{E}\left[\frac1{\Lambda_n}\right].
\]
Let $n\to\infty$, by the monotone convergence theorem we infer that
\[
\mathbb{P}\left(
\sup_{t\in[0,T]}Z_{t\wedge\bar{\tau}_m}\geq \Lambda
\right)
\leq
\mathbb{E}\left[\frac1\Lambda\right].
\]

The proof of \autoref{prop-Z_t} is complete.
\end{proof}

\section*{Acknowledgments}
The first named author would like to thank Benedetta Ferrario and Enrico Priola, the organizers of a Pavia "One-day workshop on SPDEs", where he presented a talk based on this paper. The authors would like to thank Jiahui Zhu for informing them about the \cite{Hong+Li+Liu_2024} paper. The second named author acknowledges financial support from the China Scholarship Council. This work was partially supported by the Key Laboratory of Nonlinear Analysis and its Applications (Chongqing University), Ministry of Education; NNSF of China (Grant No.12371146); the Fundamental Research Funds for the Central Universities (Grant No. 2025CDJZKPT-09); Natural Science Foundation Project of CQ (Grant No. CSTB2025NSCQ-GPX0722).

    \section*{Declarations:}

	\noindent 	\textbf{Ethical Approval:}   Not applicable

	\noindent  \textbf{Conflict of interest: } On behalf of all authors, the corresponding author states that there is no conflict of interest.

	\noindent 	\textbf{Authors' contributions:} All authors have contributed equally.

	\noindent 	\textbf{Availability of data and materials:} Not applicable.

\bibliographystyle{alpha}
\bibliography{biblio_NSEs-1}

\end{document}